\documentclass[a4paper,reqno,11pt]{amsart} 

\usepackage[T1]{fontenc}
\usepackage[utf8]{inputenc}
\usepackage{xspace,
	amsfonts}
\usepackage{amsmath} 
\usepackage{amssymb}
\usepackage{graphicx,
	bbm,
	tikz,
	subfig}
\usepackage{xfrac}
\usepackage{dsfont}
\usepackage{amsthm}
\usepackage{mathtools}
\usepackage{enumitem}
\usepackage{verbatim}
\usepackage[autostyle]{csquotes}

\usepackage{cite}

\mathtoolsset{showonlyrefs}

\usepackage{geometry}
\newtheorem{theorem}{Theorem}[section]
\newtheorem{lemma}[theorem]{Lemma}
\newtheorem{proposition}[theorem]{Proposition}
\newtheorem{corollary}[theorem]{Corollary}

\newtheorem{remark}[theorem]{Remark}

\newcommand{\R}{\mathbb{R}}

\newcommand{\G}{\mathcal{G}}

\newcommand{\E}{\mathbb{E}}

\newcommand{\V}{\mathbb{V}}

\newcommand{\J}{\mathcal{J}}

\newcommand{\D}{\mathcal{D}}

\newcommand{\dd}{\,d}

\newcommand{\NN}{\mathcal N}
\newcommand{\LL}{\mathcal L}
\newcommand{\JJ}{\mathcal J}
\newcommand{\EE}{\mathcal{E}}

\newcommand\vv{\textsc{v}}

\tikzstyle{nodino}=[circle,draw,fill,inner sep=0pt,minimum size=0.5mm]
\tikzstyle{infinito}=[circle,inner sep=0pt,minimum size=0mm]
\tikzstyle{nodo}=[circle,draw,fill,inner sep=0pt, minimum size=0.5*width("k")]
\tikzstyle{nodo_vuoto}=[circle,draw,inner sep=0pt, minimum size=0.5*width("k")]
\usetikzlibrary{graphs}
\tikzset{every loop/.style={min distance=10mm,in=300,out=240,looseness=10}}
\tikzset{place/.style={circle,thick,draw=blue!75,fill=blue!20,minimum
		size=6mm}}
\tikzset{place2/.style={circle,thick,draw=red!75,fill=red!20,minimum
		size=6mm}}

\title[ ]{Uniqueness and non--uniqueness of least energy normalized solutions for nonlinear Schr\"odinger equations on compact metric graphs}
\author[ ]{Simone Dovetta}
\address[Simone Dovetta]{Politecnico di Torino, Dipartimento di Scienze Matematiche ``G.L. Lagrange'', Corso Duca degli Abruzzi, 24, 10129 Torino, Italy}
\email{simone.dovetta@polito.it}

\author[ ]{Lun Guo}
\address[Lun Guo]{School of Mathematics and Statistics, South-Central Minzu University,  Minzu Road 182, 430074 Wuhan,  P.R. China.}
\email{lguo@mails.ccnu.edu.cn}

\begin{document}
	
	\maketitle
	
	\begin{abstract}
		We investigate uniqueness and non-uniqueness of least energy normalized and least energy nodal normalized solutions for nonlinear Schr\"odinger equations on compact metric graphs. We first prove existence of least energy nodal normalized solutions in the $L^2$-subcritical regime and, at the critical exponent, below a graph-dependent threshold. We then establish a conditional non-uniqueness result for slightly $L^2$-subcritical powers and identify broad classes of graphs for which the required condition either holds or fails. In particular, we prove uniqueness of least energy nodal normalized solutions on the interval for every mass and every $p\in(2,6)$. Finally, using ODE and phase-plane techniques, we show that least energy normalized solutions on the interval are unique for $p$ sufficiently close to $2$. Overall, the results reveal a strong dependence of the uniqueness picture on the nonlinearity power, the topology, and the metric of the graph, and a structural difference between the constant sign and the sign-changing settings.
		
	\end{abstract}
	
	\smallskip
	
	
	\section{Introduction}
	
	This paper investigates uniqueness and non-uniqueness phenomena for least energy normalized solutions and least energy nodal normalized solutions of nonlinear Schr\"odinger (NLS) equations on compact metric graphs with homogeneous Kirchhoff conditions at the vertices. Specifically, we are interested in solutions to the problem
	\begin{equation}
		\label{eq:nls}
		\begin{cases}
			u''+|u|^{p-2}u=\lambda u & \text{on each edge }e\in\mathbb E_\G\\
			u \text{ is continuous} & \text{on }\G\\
			\sum_{e\succ\vv}u_e'(\vv)=0 & \text{at every vertex }\vv\in\V_\G\\
			\|u\|_{L^2(\G)}^2=\mu\,, & 
		\end{cases}
	\end{equation}
	where $\G=(\V_\G,\E_\G)$ is a connected compact metric graph (i.e. a graph with finitely many vertices and edges, all bounded), and $p>2$, $\lambda\in\R$ and $\mu>0$ are real parameters. 
	
	A solution $u$ to \eqref{eq:nls} is called {\em normalized} because its mass (the $L^2$-norm) is fixed a priori. From a variational point of view, normalized solutions are critical points of the energy functional $E_p:H^1(\G)\to\R$
	\begin{equation}
		\label{eq:energy}
		E_p(u,\G):=\frac12\|u'\|_{L^2(\G)}^2-\frac1p\|u\|_{L^p(\G)}^p
	\end{equation}
	restricted to the mass-constrained space
	\[
	H_\mu^1(\G):=\left\{u\in H^1(\G)\,:\,\|u\|_{L^2(\G)}^2=\mu\right\}.
	\]
	The parameter $\lambda$ in \eqref{eq:nls}, which is sometimes called the frequency of the solution, is then a Lagrange multiplier associated to the mass constraint and it is determined by the formula
	\begin{equation}
		\label{eq:Lu}
		\lambda=\mathcal L_p(u,\G):=\frac{\|u\|_{L^p(\G)}^p-\|u'\|_{L^2(\G)}^2}{\|u\|_{L^2(\G)}^2}\,.
	\end{equation}
	Among all critical points of $E_p$ with prescribed mass, a distinguished role is played by {\em least energy} normalized solutions, namely solutions $u$ to \eqref{eq:nls} minimizing the energy among all solutions with the same mass
	\begin{equation}
	\label{eq:least}
	E_p(u,\G)=\inf\left\{E_p(v,\G)\,:\,v\text{ solves }\eqref{eq:nls}\text{ for some }\lambda\in\R\right\}.
	\end{equation}
	Their nodal counterparts are the {\em least energy nodal }normalized solutions, i.e. sign-changing functions $u$ satisfying \eqref{eq:nls} and
	\begin{equation}
		\label{eq:leastnod}
	E_p(u,\G)=\EE_{p,\G}^{nod}(\mu):=\inf\left\{E_p(v,\G)\,:\,v\text{ is a sign-changing solution of }\eqref{eq:nls}\text{ for some }\lambda\in\R\right\}.
	\end{equation}
	Although this need not hold in full generality, least energy normalized solutions are heuristically expected to have constant sign. The minimization problems \eqref{eq:least} and \eqref{eq:leastnod} are thus meant to select genuinely different classes of solutions.
	
	\smallskip
	Nonlinear Schrödinger equations on metric graphs have attracted intense attention over the last decades. Fueled by potential applications to nonlinear optics and to the propagation of matter waves on ramified structures in the theory of Bose-Einstein condensates (see e.g. \cite{lorenzo} and references therein), these models have been gathering significant interest in the mathematical community, as they often display new phenomena with no direct analogue in the corresponding problems posed on $\R^N$. An extensive literature is now available on the existence of {\em positive} normalized solutions to NLS equations (see e.g. \cite{ABR, ACFN-AIHP-2014, ACFN-JDE-2016, ADST, ASTcmp, AST, AST2, ACT2, ACT1, ACT3, BMP, BDL1, BDL2, BC, BD, BCJS, CGJT, CJS, DT1, GCM, KMPX, LLZ, NP, pankov, PeSc, PS, PSV, ST, SV, T} and the monographs \cite{ASTparma, KNP} for comprehensive overviews on the subject). In particular, a large body of work has been devoted to {\em energy ground states}, defined as global minimizers of the energy among all functions with prescribed mass
	\begin{equation}
		\label{eq:GS}
		E_p(u,\G)=\EE_{p,\G}(\mu):=\inf_{v\in H_\mu^1(\G)}E_p(v,\G)\,.
	\end{equation}  
	Energy ground states play a central role because every ground state is a least energy normalized solution, while, from a variational standpoint, the minimization problem \eqref{eq:GS} on the whole manifold $H_\mu^1(\G)$ is considerably more tractable than the minimization over the set of solutions in \eqref{eq:least}. Moreover, ground states always have constant sign. Thus, whenever a ground state exists, problem \eqref{eq:least} is attained, and its minimizers are distinct from those selected by the nodal problem \eqref{eq:leastnod}.
	
	Existence of energy ground states is by now well understood. No energy ground states exist in the $L^2$-supercritical regime $p>6$, since $\EE_{p,\G}(\mu)=-\infty$ for every $\mu>0$. In the $L^2$-subcritical $2<p<6$ and $L^2$-critical $p=6$ regimes, instead, the problem exhibits a rich phenomenology and it is strongly sensitive to the topology and the metric of the graph, as well as to the actual value of the mass (we refer e.g. to the series of papers \cite{AST,AST2,ASTcmp} and to references therein). 
	
	By contrast, the uniqueness of energy ground states has remained largely unexplored until very recently. This is hardly surprising, as uniqueness questions for NLS equations are notoriously challenging even beyond the framework of metric graphs. On graphs, the problem is further complicated by the lack of symmetry and scale invariance of the underlying domain, which rules out a direct use of the few techniques available in related settings. However, the recent work \cite{D25}, inspired by the analysis developed in \cite{ACT1}, established a rather surprising general non-uniqueness result. Roughly speaking, \cite{D25} shows that, except for the line, the half-line, and their isomorphic variants in the tower of bubbles, every graph admitting ground states for all masses and all $p\in(2,6)$ has non-unique ground states at some mass for slightly $L^2$-subcritical powers $p$. Precisely, for every $p$ sufficiently close to $6$, there exists a mass $\mu_p$ and two ground states $u,v$ of mass $\mu_p$ with distinct Lagrange multipliers
	\[
	\LL_p(u,\G)\neq \LL_p(v,\G)\,.
	\]
	This non-uniqueness result is unexpected for three reasons. First, it is {\em ubiquitous}: it occurs on {\em every} graph for which existence of ground states holds for every mass in the $L^2$-subcritical regime, independently of any further property of the graph. This is the case, for instance, for every compact graph \cite{D18} and for a variety of noncompact ones \cite{AST, AST2, D19}. Second, \cite{DST20} showed that the set of masses admitting two ground states with distinct frequencies is always at most countable, suggesting that such non-uniqueness should be exceptional. Third, it sharply contrasts with previously known uniqueness results for ground states with fixed mass on bounded domains \cite{NTV}. Moreover, it remains unclear whether an analogous phenomenon holds with the same level of generality in higher dimensions (see \cite{DST26} for a discussion on this point in two dimensions).
	
	In contrast with positive solutions, the study of sign-changing solutions to NLS equations on graphs is still at a very early stage. To the best of our knowledge, the only general results on nodal normalized solutions on graphs are those of \cite{GW,JS} (whereas existence results without the mass constraint are given in \cite{DDGST}). In \cite{JS}, the existence of multiple sign-changing normalized solutions to \eqref{eq:nls} is established on compact graphs in the $L^2$-supercritical regime $p>6$, provided the mass is sufficiently small. Using different techniques, \cite{GW} extends this multiplicity result to more general nonlinearities and arbitrary masses $\mu>0$. As far as we know, no general existence result is available when $p\in(2,6]$, and the problem becomes even more delicate when one asks for least energy nodal solutions. Indeed, since on compact graphs ground states exist for every $\mu>0$ and every $p\in(2,6)$ (see \cite{D18}), problems \eqref{eq:least} and \eqref{eq:leastnod} are genuinely distinct, at least in the $L^2$-subcritical regime. However, no direct nodal analogue of the variational formulation \eqref{eq:GS} seems to be available, perhaps because of the difficulty of simultaneously incorporating the mass and the sign constraint into a tractable minimization problem. Letting $u^+=\max(u,0)$
and $u^-=\min(u,0)$ be the positive and negative parts of $u$, two natural candidates would be for instance 
	\[
	\inf_{\substack{u\in H_\mu^1(\G) \\ u^\pm\not\equiv0}}E_p(u,\G)\qquad\text{or}\qquad\inf_{\substack{u^+\in H_{\mu_+}^1(\G),\,u^-\in H_{\mu_-}^1(\G) \\ \mu_++\mu_-=\mu}}E_p(u,\G)\,.
	\]
	However, the first infimum equals the ground state level $\EE_{p,\G}(\mu)$, but it is not attained in the sign-changing class. In the second problem, the two mass constraints generally produce two distinct Lagrange multipliers, so a priori a minimizer need not solve \eqref{eq:nls}. 
	
	\medskip
	The present paper advances this line of research by addressing the following questions:

	\begin{enumerate}
		
		\item[1.] does the non-uniqueness phenomenon established for slightly $L^2$-subcritical energy ground states extend to least energy nodal normalized solutions?
		\smallskip
		
		\item[2.] Is this non-uniqueness confined to a perturbative regime near the $L^2$-critical exponent $p=6$, or does it persist for general $p\in(2,6)$? 
	
	\end{enumerate}
	\smallskip
	
	Clearly, before addressing uniqueness, the existence of least energy nodal normalized solutions must be established. A new approach to this problem was recently developed in \cite{DDGS}. The starting point is the problem obtained from \eqref{eq:nls} by fixing $\lambda$ and dropping the mass constraint
	\begin{equation}
		\label{eq:nlsnomass}
		\begin{cases}
			u''+|u|^{p-2}u=\lambda u & \text{on each edge }e\in\mathbb E_\G\\
			u \text{ is continuous} & \text{on }\G\\
			\sum_{e\succ\vv}u_e'(\vv)=0 & \text{at every vertex }\vv\in\V_\G\,,\\
		\end{cases}
	\end{equation}
	for which well-established variational methods are available. For  fixed $\lambda\in\R$,  let $J_{p,\lambda}:H^1(\G)\to\R$ be the {\em action }functional
	\begin{equation}
		\label{eq:J}
		J_{p,\lambda}(u,\G):=\frac12\|u'\|_{L^2(\G)}^2+\frac\lambda2\|u\|_{L^2(\G)}^2-\frac1p\|u\|_{L^p(\G)}^p\,.
	\end{equation}
	Constant sign solutions to \eqref{eq:nlsnomass} can be obtained by minimizing $J_{p,\lambda}$ on the associated Nehari manifold
	\begin{equation}
		\label{eq:Jlev}
	\JJ_{p,\G}(\lambda):=\inf_{v\in\NN_{p,\lambda}(\G)}J_{p,\lambda}(v,\G)\,,
	\end{equation}
	where
	\[
	\begin{split}
	\NN_{p,\lambda}(\G):=&\,\left\{u\in H^1(\G)\setminus\left\{0\right\}\,:\,J_{p,\lambda}'(u,\G)u=0\right\}\\
	=&\,\left\{u\in H^1(\G)\setminus\left\{0\right\}\,:\,\|u'\|_{L^2(\G)}^2+\lambda\|u\|_{L^2(\G)}^2=\|u\|_{L^p(\G)}^p\right\}\,.
	\end{split}
	\]
	The nodal counterpart is obtained by the minimization problem
	\begin{equation}
		\label{eq:Jnodlev}
	\JJ_{p,\G}^{nod}(\lambda):=\inf_{v\in\NN_{p,\lambda}^{nod}(\G)}J_{p,\lambda}(v,\G)
	\end{equation}
	on the nodal Nehari set
	\[
	\NN_{p,\lambda}^{nod}(\G):=\left\{u\in H^1(\G)\,:\, u^\pm\in\NN_{p,\lambda}(\G)\right\}.
	\]
	Minimizers in \eqref{eq:Jlev} and \eqref{eq:Jnodlev} are usually called {\em action ground states} and {\em nodal action ground states}, respectively, and their existence is often known and, in any case, is more tractable than in the normalized setting. The strategy proposed in \cite{DDGS} is to analyze how the masses of action and nodal action ground states vary with $\lambda$, and then identify regimes in which these minimizers are also least energy normalized and least energy nodal normalized solutions.
	
	In \cite{DDGS} this strategy was successfully implemented for NLS equations on bounded domains of $\R^N$ with homogeneous Dirichlet boundary conditions. Here, we extend this approach to compact metric graphs. This yields our first main result, establishing existence of least energy nodal normalized solutions for every mass in the $L^2$-subcritical regime and for masses up to a threshold value at the $L^2$-critical exponent. 
	
	Recall that a {\em pendant} is an edge with an endpoint of degree one. We set
	\begin{equation}
		\label{eq:muGnod}
		\mu_\G^{nod}:=\begin{cases}
			2\mu_{\R^+} & \text{if }\G\text{ has at least two pendants}\\
			\mu_{\R^+}+\mu_\R & \text{if }\G\text{ has exactly one pendant}\\
			2\mu_\R & \text{if }\G\text{ has no pendant}\,,
		\end{cases}
	\end{equation}
	where $\displaystyle\mu_\R=\sqrt3\pi/2$ and $\displaystyle\mu_{\R^+}=\mu_\R/2$ are the critical masses for the $L^2$-critical ground state problem on the line and on the half-line, respectively (see e.g. \cite{ASTcmp}). 
	
	\begin{theorem}
		\label{thm:nodex}
		Let $\G$ be a compact graph. Assume one of the following holds:
		\begin{itemize}
			\item[(i)] $p\in(2,6)$ and $\mu>0$;
			
			\item[(ii)] $p=6$ and $\mu\in(0,\mu_\G^{nod})$;
			
			\item[(iii)] $p=6$, $\mu=\mu_\G^{nod}$ and $\EE_{6,\G}^{nod}(\mu_\G^{nod})<0$.
		
		\end{itemize}
		Then there exists a least energy nodal normalized solution of \eqref{eq:nls} with mass $\mu$. Moreover, every  least energy nodal normalized solution $u$ is a nodal action ground state in $\NN_{p,\lambda}^{nod}(\G)$, where $\lambda=\LL_{p}(u,\G)$. Conversely, any other nodal action ground state $v\in \NN_{p,\lambda}^{nod}(\G)$ satisfies $\|v\|_{L^2(\G)}^2=\mu$ and it is a least energy nodal normalized solution of \eqref{eq:nls} with mass $\mu$. 
 	\end{theorem}
 	
 	The existence of solutions below a threshold determined solely by the number of pendants of the graph is a characteristic feature of these $L^2$-critical problems, already observed for energy ground states in \cite[Theorem 1.2]{D18}. In contrast with \cite{D18}, however, Theorem \ref{thm:nodex} leaves open the possibility that least energy nodal normalized solutions may also exist above this threshold.
 	
 	With the existence issue settled, we now turn to the question whether the non-uniqueness phenomenon established for energy ground states has a nodal counterpart. Our first result in this direction is the following. 
 	\begin{theorem}
 		\label{thm:nonun} 	
 	Let $\G$ be a compact graph such that $\EE_{6,\G}^{nod}(\mu_\G^{nod})<0$. Then there exists a value $p_\G>2$ (depending on $\G$) such that, for every $p\in[p_\G,6)$, there exist $\mu_p>0$ and two least energy nodal normalized solutions $u,v$ to \eqref{eq:nls} with mass $\mu_p$ and $\LL_{p}(u,\G)\neq\LL_p(v,\G)$. 
 	\end{theorem}
	Theorem \ref{thm:nonun} differs from \cite[Theorem 1.1]{D25} in one crucial respect: whereas non-uniqueness of energy ground states is unconditional on compact graphs, its nodal analogue requires the existence of an $L^2$-critical nodal solution to \eqref{eq:nls} at the threshold mass $\mu_\G^{nod}$ with strictly negative energy. From a technical point of view, the result of \cite{D25} is based on three key ingredients:
	\smallskip
	
	\begin{itemize}
		
		\item[(i)] if $u_1, u_2$ are ground states of $E_p$ with mass $\mu_1<\mu_2$, then $\LL_p(u_1,\G)<\LL_p(u_2,\G)$;
		\smallskip
		
		\item[(ii)] every energy ground state is also an action ground state;
		\smallskip
		
		\item[(iii)] when $p=6$, there exist $\lambda_1<\lambda_2$ and two action ground states $v_1\in\NN_{6,\lambda_1}(\G)$, $v_2\in\NN_{6,\lambda_2}(\G)$ such that $\|v_1\|_{L^2(\G)}>\|v_2\|_{L^2(\G)}$.
	\end{itemize}
	\smallskip
	Properties (i) and (ii) have been proved in \cite[Lemma 4.1]{DST20} and in \cite[Theorem 1.3]{DST23}, and their proofs make essential use of the fact that energy ground states minimize $E_p$ on the whole mass-constrained manifold $H_\mu^1(\G)$. Property (iii) is proved in \cite[Section 4]{D25} and it heavily relies on the fact that $L^2$-critical energy ground states exist not only below a threshold value of the mass, but also {\em at the threshold}. 
	
	For the purposes of Theorem \ref{thm:nonun}, one then needs to understand whether (i), (ii) and (iii) remain true when replacing energy ground states with least energy nodal normalized solutions and action ground states with nodal action ground states. Extending each of these properties to the nodal setting is in fact nontrivial and requires new ideas.  Theorem \ref{thm:nodex} above gives the analogue of (ii), and a deeper analysis of the relation between least energy nodal normalized solutions and nodal action ground states allows us to obtain the desired extension of (i) (see Proposition \ref{prop:Lpm} below). Notably, these two steps do not require any specific assumption on the graph. This is no longer true when turning to (iii). To obtain its nodal counterpart, we need to know that $\EE_{6,\G}^{nod}(\mu_\G^{nod})$ is attained. By Theorem \ref{thm:nodex}, the infimum is attained whenever it is strictly negative. This explains the assumption $\EE_{6,\G}^{nod}(\mu_\G^{nod})<0$ in Theorem \ref{thm:nonun}.
	
	At this stage, however, Theorem \ref{thm:nonun} does not yet show that nodal non-uniqueness is genuinely less ubiquitous than its ground state counterpart. To this end, it is crucial to understand whether the additional assumption $\EE_{6,\G}^{nod}(\mu_\G^{nod})<0$ is actually satisfied by some, or possibly all, compact graphs. Moreover, in our analysis this condition arises from the specific argument used in the proof. A priori, it is therefore not clear whether it reflects a structural feature of the nodal problem or it is merely a technical limitation of the present approach. 
	
	The distinction is in fact genuine. Even if the assumption in Theorem \ref{thm:nonun} may not be the sharpest possible, it is not merely an artifact of the proof, but reflects an inherent feature of the nodal problem. Indeed, we identify a broad class of graphs for which it is satisfied and nodal non-uniqueness follows, while under different topological conditions we prove that it always fails. More strikingly, within this second class we exhibit explicit graphs on which least energy nodal normalized solutions are unique for every mass and every $L^2$-subcritical exponent.
	
	To exhibit graphs to which Theorem \ref{thm:nonun} applies, consider the following construction. Let $\G_1,\G_2$ be two identical copies of the same compact graph with no pendant. Fix two corresponding vertices in $\G_1$ and $\G_2$, attach to each of them an interval of length $\ell>0$ at one of its endpoints, and then glue together the obtained graphs identifying the other endpoints of the intervals of length $\ell$. Denote by $\D_\ell$ the resulting dumbbell graph with a vertex of degree 2 at the middle point of an edge of length $2\ell$. We say that a compact graph $\G$ contains a copy of $\D_\ell$ if $\D_\ell$ is a subgraph of $\G$, i.e. if a copy of $\D_\ell$ is attached to a vertex of $\G$ (see Figure \ref{fig:Dl}).  
	
	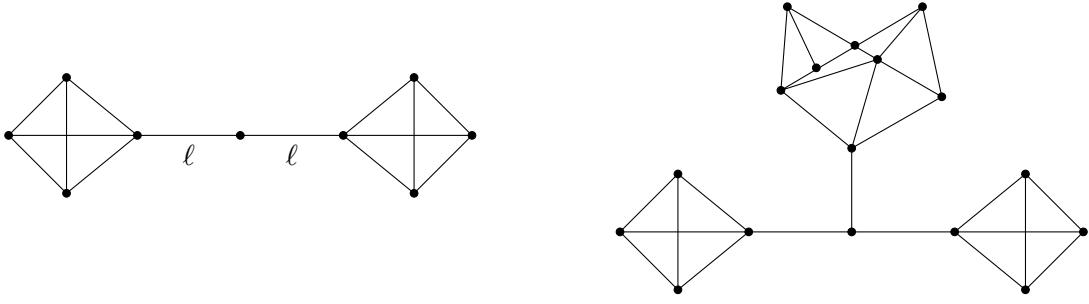
\begin{figure}[t]
	\begin{tikzpicture}[scale=0.85]
		
		
		\begin{scope}[shift={(0,0)}]
			
			\node at (0,0) [nodo] {};
			\node at (-1.1,0.9) [nodo] {};
			\node at (-1.1,-0.9) [nodo] {};
			\node at (-2,0) [nodo] {};

			\draw (0,0) -- (-1.1,0.9) -- (-2,0) -- (-1.1,-0.9) -- cycle;
			\draw (-1.1,0.9) -- (-1.1,-0.9);
			\draw (-2,0)--(0,0);

			\node at (3.2,0) [nodo] {};
			\node at (4.3,0.9) [nodo] {};
			\node at (4.3,-0.9) [nodo] {};
			\node at (5.2,0) [nodo] {};
			
			\draw (3.2,0) -- (4.3,0.9) -- (5.2,0) -- (4.3,-0.9) -- cycle;
			\draw (4.3,0.9) -- (4.3,-0.9);
			\draw (3.2,0)--(5.2,0);
		
			\draw (0,0) -- (3.2,0);
			
			\node at (1.6,0) [nodo] {};
			\node at (0.8,-.3) [infinito] {$\small\ell$};
			\node at (2.4,-.3) [infinito] {$\small\ell$};

		\end{scope}

		\begin{scope}[shift={(9.5,-1.5)}]
			
			\node at (0,0) [nodo] {};
			\node at (-1.1,0.9) [nodo] {};
			\node at (-1.1,-0.9) [nodo] {};
			\node at (-2,0) [nodo] {};

			\draw (0,0) -- (-1.1,0.9) -- (-2,0) -- (-1.1,-0.9) -- cycle;
			\draw (-1.1,0.9) -- (-1.1,-0.9);
			\draw (-2,0)--(0,0);
			
			\node at (3.2,0) [nodo] {};
			\node at (4.3,0.9) [nodo] {};
			\node at (4.3,-0.9) [nodo] {};
			\node at (5.2,0) [nodo] {};
			
			\draw (3.2,0) -- (4.3,0.9) -- (5.2,0) -- (4.3,-0.9) -- cycle;
			\draw (4.3,0.9) -- (4.3,-0.9);
			\draw (3.2,0)--(5.2,0);
			
			\draw (0,0) -- (3.2,0);
			
			\node at (1.6,0) [nodo] {};

			\node at (1.6,1.3) [nodo] {};
			\node at (0.5,2.2) [nodo] {};
			\node at (2,2.68) [nodo] {};
			\node at (3,2.1) [nodo] {};
			\node at (2.7,3.5) [nodo] {};
			\node at (0.6,3.5) [nodo] {};
			\node at (1.65,2.9) [nodo] {};
			\node at (1.05,2.55) [nodo] {};

			\draw (1.6,0) -- (1.6,1.3)
			(1.6,1.3) -- (0.5,2.2)
			(1.6,1.3) -- (3,2.1)
			(0.5,2.2) -- (2,2.68)
			(2,2.68) -- (3,2.1)
			(0.5,2.2) -- (0.6,3.5)
			(0.6,3.5) -- (2,2.68)
			(2,2.68) -- (2.7,3.5)
			(2.7,3.5) -- (3,2.1)
			(0.6,3.5) -- (1.05,2.55)
			(2.7,3.5) -- (0.5,2.2)
			(1.6,1.3) -- (2,2.68);

		\end{scope}
		
	\end{tikzpicture}
	\caption{A dumbbell graph $\D_\ell$ (left) and a general compact graph $\G$ containing a copy of $\D_\ell$ (right).}
	\label{fig:Dl}
	\end{figure}
	\begin{theorem}
		\label{thm:D}
		Let $\G$ be a compact graph with no pendant containing a copy of $\D_\ell$. Then $\mu_{\G}^{nod}=2\mu_\R$ and, when $\ell$ is large enough, $\EE_{6,\G}^{nod}(\mu_{\G}^{nod})<0$. Hence, for every $\ell$ sufficiently large, Theorem \ref{thm:nonun} applies to $\G$. 
	\end{theorem}
	Theorem \ref{thm:D} shows that not only the topology, but also the metric plays a crucial role in the occurrence of the non-uniqueness phenomenon. The class of graphs identified in Theorem \ref{thm:D} is rather general, since the graphs $\G_1,\G_2$ in the dumbbell structure are only required to be identical and to have no pendant. We actually believe that the result may extend to cases in which $G_1$ and $\G_2$ are not equal, but the argument we used to prove Theorem \ref{thm:D} does not work at this level of generality.
	
	In contrast with the previous result, it is not difficult to find graphs where the hypotheses of Theorem \ref{thm:nonun} are not satisfied. Hence, it is easy to see that if $\G$ has either at least two pendants or it admits a cycle covering (i.e. it can be covered by a finite number of cycles, see \cite{ASTcmp}), then $\EE_{6,\G}^{nod}(\mu_\G^{nod})\geq0$ (see Proposition \ref{prop:E>0} below). This obstruction reflects a genuine uniqueness mechanism. In fact, nodal non-uniqueness does not hold in general within these classes. To show this, we consider the simplest compact graph with two pendants, the interval $[0,1]$, on which \eqref{eq:nls} is simply
	\begin{equation}
		\label{eq:main}
		\begin{cases}
		u''+|u|^{p-2}u=\lambda u & \text{in }(0,1)\\
		u'(0)=u'(1)=0 & \\
		\|u\|_{L^2(0,1)}^2=\mu. &  
		\end{cases}
	\end{equation}
	Inspired by the recent papers \cite{ACT1,ACT2}, we analyze the problem of least energy nodal normalized solutions for \eqref{eq:main} by phase-plane methods and a detailed study of the associated period functions. This yields a complete uniqueness result.
	\begin{theorem}
		\label{thm:uniqnod}
		For every $p\in(2,6)$ and every $\mu>0$, up to multiplication by a constant complex phase there exists a unique least energy nodal normalized solution to \eqref{eq:main} with mass $\mu$. 
	\end{theorem}
	The interval therefore provides a sharp separation between the nodal and positive theories: least energy nodal normalized solutions are unique for every $p\in(2,6)$, whereas energy ground states are non-unique at some mass when $p$ is sufficiently close to $6$.
	
	In fact, the ODE approach developed for Theorem \ref{thm:uniqnod} can be refined to shed further light on the nature of this non-uniqueness phenomenon. Even for energy ground states, non-uniqueness is currently known only for $p$ sufficiently close to $6$, and whether it persists away from the critical regime has remained open. Our ODE analysis gives a negative answer on the interval.
	\begin{theorem}
		\label{thm:uniq}
		There exists $\varepsilon>0$ such that, for every $p\in(2,2+\varepsilon)$ and every $\mu>0$, up to multiplication by a constant complex phase and the reflection $x\mapsto 1-x$, there exists a unique energy ground state of $E_p$ on $[0,1]$ with mass $\mu$. 
	\end{theorem}
	Thus, even for energy ground states, the non-uniqueness mechanism near $p=6$ is genuinely perturbative, at least on the interval.
	
	A simple reflection argument transfers Theorems \ref{thm:uniqnod}--\ref{thm:uniq} to the circle, the simplest compact graph admitting a cycle covering. Overall, our results reveal a markedly richer uniqueness picture than in the ground state theory, governed jointly by the exponent, the topology, and the metric of the graph.
	
	\smallskip
	The remainder of the paper is organized as follows. In Section \ref{sec:prel} we derive some preliminary results. Section \ref{sec:exun} deals with existence and non-uniqueness of least energy nodal normalized solutions, providing the proof of Theorems \ref{thm:nodex}--\ref{thm:nonun}. In Section \ref{sec:hyp} we discuss the applicability of Theorem \ref{thm:nonun}, proving in particular Theorem \ref{thm:D}. Finally, Sections \ref{sec:uniqnod}--\ref{sec:uniqp=2} address uniqueness of least energy solutions on the interval: the former focuses on the nodal setting and proves Theorem \ref{thm:uniqnod}, the latter is devoted to energy ground states and to the proof of Theorem \ref{thm:uniq}.
	
	\medskip
	{\bf Notation.} In what follows, whenever possible we use short notations like $\|u\|_p$ to denote $L^p$ norms, avoiding to state explicitly the domain of integration unless necessary.
	
	\section{Preliminaries}
	\label{sec:prel}
	
	This section collects some results on action ground states and nodal action ground states that will be used in the following to prove Theorems \ref{thm:nodex}--\ref{thm:nonun}. As we anticipated in the Introduction, the proof of Theorem \ref{thm:nodex} is done by adapting to compact graphs the strategy originally developed in \cite{DDGS} for bounded domains of $\R^N$ with homogeneous Dirichlet boundary conditions. The approach of \cite{DDGS} is in fact sufficiently general that a substantial part of the argument does not depend on whether the problem is posed on a bounded domain or on a compact metric graph, nor on the specific boundary conditions involved. In order to avoid unnecessary repetitions, here and in the next section we shall thus discuss in detail only those steps that require a genuine adaptation to the graph setting. At all other stages, we will indicate that the argument is unchanged and refer to the corresponding result in \cite{DDGS} for the proof.
	
	As we will repeatedly use it, recall that for every $u\in\NN_{p,\lambda}(\G)$ it holds
		\begin{equation}\label{B-1}
		J_{p,\lambda}(u,\mathcal{G})=\left(\frac{1}{2}-\frac{1}{p}\right)\|u\|_{p}^{p}=\left(\frac{1}{2}-\frac{1}{p}\right)(\|u'\|_{2}^{2}+ \lambda \|u\|_{2}^{2})\,,
	\end{equation}
	and that, given $\lambda\in\R$ and $u\in H^1(\G)$ such that $\|u'\|_2^2+\lambda\|u\|_2^2>0$, the quantity
	\[
	n_{p,\lambda}(u):=\left(\frac{\|u'\|_2^2+\lambda\|u\|_2^2}{\|u\|_p^p}\right)^{\frac1{p-2}}
	\]
	is well-defined and $n_{p,\lambda}(u)u\in\NN_{p,\lambda}(\G)$. 

	 Let
	\[
	\mathfrak{A}_{p,\G}(\lambda):=\left\{u\in \NN_{p,\lambda}(\G)\,:\, J_{p,\lambda}(u,\G)=\JJ_{p,\G}(\lambda)\right\}
	\]
	be the set of all action ground states in $\NN_{p,\lambda}(\G)$, and set
	\[
	Q_{p,\G}^-(\lambda):=\inf_{u\in\mathfrak{A}_{p,\G}(\lambda)}\|u\|_2^2\,,\qquad Q_{p,\G}^+(\lambda):=\sup_{u\in\mathfrak{A}_{p,\G}(\lambda)}\|u\|_2^2\,.
	\]
	The next result on action ground states is the content of \cite[Proposition 2.2]{D25}.
	\begin{proposition}
		\label{prop:AGS}
		Let $\G$ be a compact graph. For every $p>2$ the action ground state level $\JJ_{p,\G}$ defined in \eqref{eq:Jlev} satisfies
		\begin{equation}
		\JJ_{p,\G}(\lambda)\begin{cases}
			=0 & \text{if }\lambda\le0\\
			>0 & \text{if }\lambda>0\,,
		\end{cases}
		\end{equation}
		it is continuous on $\R$, strictly increasing on $\R^+$, and it is attained if and only if $\lambda>0$. Moreover, for every $\lambda>0$ the quantities $Q_{p,\G}^\pm(\lambda)$ are attained, they satisfy
		\[
		\left(\J_{p,\G}\right)'_{\mp}(\lambda)=\frac{Q_{p,\G}^\pm(\lambda)}{2}\,,
		\]
		and there exists an at most countable set $\widetilde Z_{p,\G}\subset\R^+$ such that
		\[
		Q_{p,\G}^-(\lambda)=Q_{p,\G}^+(\lambda)\qquad\forall\lambda\in\R^+\setminus\widetilde Z_{p,\G}\,.
		\]
		In particular, $\JJ_{p,\G}$ is locally Lipschitz on $\R$ and differentiable in $\R^+\setminus\widetilde Z_{p,\G}$. 
	\end{proposition}
	\begin{remark}
		\label{rem:extAGS}
		The threshold $\lambda=0$ separating the different regimes in Proposition \ref{prop:AGS} is the first eigenvalue of the Laplace operator $-d^2/dx^2$ on $\G$ endowed with homogeneous Kirchhoff conditions. Actually, a direct inspection of the proof of \cite[Proposition 2.2]{D25} immediately shows that Proposition \ref{prop:AGS} above remains true even with more general vertex conditions, such as mixed homogeneous Kirchhoff and Dirichlet conditions, simply replacing the value $\lambda=0$ with $\lambda=-\nu_1(\G)$, where $\nu_1(\G)$ is the bottom of the spectrum of the associated Laplace operator. 
	\end{remark}
	
	Denote now by
	\[
	\mathfrak{A}_{p,\G}^{nod}(\lambda):=\left\{u\in \NN_{p,\lambda}^{nod}(\G)\,:\, J_{p,\lambda}(u,\G)=\JJ_{p,\G}^{nod}(\lambda)\right\}
	\]
	the set of nodal action ground states in $\NN_{p,\lambda}^{nod}(\G)$, and 
	\[
	Q_{p,\G}^{nod,-}(\lambda):=\inf_{u\in\mathfrak{A}_{p,\G}^{nod}(\lambda)}\|u\|_2^2\,,\qquad Q_{p,\G}^{nod,+}(\lambda):=\sup_{u\in\mathfrak{A}_{p,\G}^{nod}(\lambda)}\|u\|_2^2\,.
	\]
	To prove Theorem \ref{thm:nodex} we need the following nodal counterpart of Proposition \ref{prop:AGS}.
	
	\begin{proposition}
		\label{prop:NAGS}
		Let $\G$ be a compact graph and $\lambda_2(\G)$ be the second eigenvalue of $-d^2/dx^2$ on $\G$ endowed with homogeneous Kirchhoff conditions at the vertices. For every $p>2$ the nodal action ground state level $\JJ_{p,\G}^{nod}$ defined in \eqref{eq:Jnodlev} satisfies
		\begin{equation}
			\label{eq:Jnod_segno}
		\JJ_{p,\G}^{nod}(\lambda)\begin{cases}
			=0 & \text{if }\lambda\le-\lambda_2(\G)\\
			>0 & \text{if }\lambda>-\lambda_2(\G)\,,
		\end{cases}
		\end{equation}
		it is continuous on $\R$, strictly increasing on $[-\lambda_2(\G),+\infty)$, and it is attained if and only if $\lambda>-\lambda_2(\G)$. Moreover, for every $\lambda>-\lambda_2(\G)$ the quantities $Q_{p,\G}^{nod,\pm}(\lambda)$ are attained and satisfy
		\begin{equation}
			\label{eq:derJnod}
			\limsup_{\varepsilon\to0^+}\frac{\JJ_{p,\G}^{nod}(\lambda+\varepsilon)-\JJ_{p,\G}^{nod}(\lambda)}\varepsilon\leq\frac{Q_{p,\G}^{nod,-}(\lambda)}2\leq\frac{Q_{p,\G}^{nod,+}(\lambda)}2\leq\liminf_{\varepsilon\to0^-}\frac{\JJ_{p,\G}^{nod}(\lambda+\varepsilon)-\JJ_{p,\G}^{nod}(\lambda)}\varepsilon\,.
		\end{equation}
		Furthermore, $\JJ_{p,\G}^{nod}$ is locally Lipschitz on $\R$. In particular, it is differentiable almost everywhere and there exists a negligible set $\widetilde{Z}_{p,\G}^{nod}\subset[-\lambda_2(\G),+\infty)$ such that
		\[
		(\JJ_{p,\G}^{nod})'(\lambda)=\frac{Q_{p,\G}^{nod,-}(\lambda)}2=\frac{Q_{p,\G}^{nod,+}(\lambda)}2\qquad\forall\lambda\in [-\lambda_2(\G),+\infty)\setminus\widetilde{Z}_{p,\G}^{nod}.
		\]
	\end{proposition}
	\begin{proof}
		We split the proof in various steps.
		
		\smallskip{\em Step 1: proof of \eqref{eq:Jnod_segno} and existence of nodal action ground states.} We first recall that the following estimates hold
		\begin{equation}
			\label{eq:lev_l_up_down}
			\JJ_{p,\G}^{nod}(\lambda)\leq C_1(\lambda+\lambda_2(\G))^{\frac{p}{p-2}}\,,\qquad 
				\mathcal{J}_{p,\mathcal{G}}^{nod}(\lambda)\geq C_2 \min \left(1, \frac{\lambda+\lambda_2(\mathcal{G})}{\lambda_2(\mathcal{G})} \right)^{\frac{p}{p-2}}
		\end{equation}
		for every $\lambda\ge-\lambda_2(\G)$, with suitable constants $C_1,C_2>0$ depending only on $p$ and $\G$. The proof of these estimates is identical to that of \cite[Proposition 2.5]{DDGS}. The unique minor difference arises when deriving the intermediate inequality
		\begin{equation}
		\label{eq:intest}
		\|v'\|_{2} \geq C\min\left(1, \frac{\lambda + \lambda_2(\mathcal{G})}{\lambda_2(\mathcal{G})}\right)^{\frac{1}{p-2}}\qquad\forall\, v\in\NN_{p,\lambda}^{nod}(\G)\text{ such that }\int_\G v\,dx=0\,,
		\end{equation}
		with a suitable constant $C>0$ depending only on $p$ and $\G$, which is needed in the proof of the second estimate in \eqref{eq:lev_l_up_down}. To prove \eqref{eq:intest}, note first that since $v$ has zero average on $\G$, it is orthogonal to constant functions, that are the eigenfunctions associated to the first eigenvalue of $-d^2/dx^2$ with homogeneous Kirchhoff conditions on $\G$. Hence, by definition of $\lambda_2(\G)$, 
		 \begin{equation}\label{GL}
		 	\| v' \|_{2}^2 + \lambda \|v\|_{2}^2 \geq \min\left(1, \frac{\lambda + \lambda_2(\mathcal{G})}{\lambda_2(\mathcal{G})} \right) \|v'\|_{2}^2\,.
		 \end{equation}
		 Moreover, by the standard $L^\infty$ Gagliardo-Nirenberg inequality
		 \[
		 \|u\|_\infty\leq C_\infty\|u\|_2^{1/2}\|u\|_{H^1}^{1/2}\qquad\forall u\in H^1(\G)
		 \]
		 with $C_\infty>0$ depending only on $\G$, we obtain
		 \begin{equation}\label{New-2}
		 		\| v\|_{\infty} \leq C_{\infty}  \frac{\|v'\|_{2}^{1/2}}{\lambda_2(\mathcal{G})^{1/4}}\left( 1+\frac{1}{\lambda_2(\mathcal{G})}\right)^{1/4} \|v'\|_{2}^{1/2} \leq  C_{\infty}  \left(1+\frac{1}{\lambda_2(\mathcal{G})} \right)^{1/2} \|v'\|_{2}\,,
		 \end{equation}
		that together with \eqref{GL}, the fact that $v\in\NN_{p,\lambda}^{nod}(\G)$ and that $\G$ is compact, gives
		 \begin{equation*}
		 		\min\left(1, \frac{\lambda + \lambda_2(\mathcal{G})}{\lambda_2(\mathcal{G})} \right) \|v'\|_{2}^2\leq\| v' \|_{2}^2 + \lambda \|v\|_{2}^2=\|v\|_{p}^{p}\leq |\mathcal{G}| \| v\|_{\infty}^{p} \leq   |\mathcal{G}|  C_{\infty}^p \left(1+\frac{1}{\lambda_2(\mathcal{G})}\right)^{\frac{p}{2}} \|v'\|_{2}^p\,,
		 \end{equation*}
		 yielding \eqref{eq:intest}, and thus allowing to recover \eqref{eq:lev_l_up_down} as in \cite[Proposition 2.5]{DDGS}. 
		 
		 By \eqref{eq:lev_l_up_down}, it directly follows that $\JJ_{p,\G}^{nod}(\lambda)>0$ if $\lambda>-\lambda_2(\G)$. Moreover, that $\JJ_{p,\G}^{nod}(\lambda)$ is attained for every $\lambda\in(-\lambda_2(\G),+\infty)$ can be proved with the same argument developed in \cite{BW} for bounded domains (no difference arises on compact graphs, as already pointed out in \cite[Proposition 2.3]{DDGST}). 
		 
		 We then show that $\JJ_{p,\G}^{nod}(\lambda)=0$ whenever $\lambda\leq-\lambda_2(\G)$. To this end, let $\varphi_2\in H^1(\G)$ be an eigenfunction associated to $\lambda_2(\G)$, and set
		 \[
		 \G_+:=\left\{x\in\G\,:\,\varphi_2(x)>0\right\},\qquad \G_-:=\left\{x\in\G\,:\,\varphi_2(x)<0\right\},\qquad\G_0=\G\setminus\left(\G_+\cup\G_-\right).
		 \]
		 Denote by $\overline{\G_\pm}$ the closure of $\G_\pm$ in $\G$. Since $\varphi_2$ is a non-negative eigenfunction of $-d^2/dx^2$ on $\overline{\G_+}$ endowed with Kirchhoff conditions at the vertices of $\overline{\G_+}\cap\G_+$ and homogeneous Dirichlet conditions at $\overline{\G_+}\setminus\G_+$, then it is an eigenfunction associated to the first eigenvalue $\lambda_{1,+}$ of this operator on $\overline{\G_+}$. Arguing analogously on $\G_-$ we obtain
		 \[
		 \lambda_2(\G)=\lambda_{1,+}=\lambda_{1,-}\,.
		 \]
		By Remark \ref{rem:extAGS}, the action ground state level on $\overline{\G_\pm}$ with corresponding mixed Kirchhoff-Dirichlet conditions is equal to zero for every $\lambda\leq-\lambda_{1,\pm}$. Hence, for every $\lambda\leq-\lambda_2(\G)$ there exist two sequences of functions $u_n^\pm\in \NN_{p,\lambda}(\overline{\G_\pm})$ such that
		\[
		u_n^+\geq0\text{ on }\overline{\G_+}\,,\qquad u_n^+\equiv0 \text{ on }\overline{\G_+}\setminus\G_+\,,\qquad u_n^-\leq0\text{ on }\overline{\G_-}\,,\qquad u_n^-\equiv0 \text{ on }\overline{\G_-}\setminus\G_-\,,
		\] 
		and $\displaystyle  \JJ_{p,\G_\pm}(u_n)\to0$ as $n\to\infty$. 
		Thus, defining
		\[
		v_n(x):=\begin{cases}
			u_n^+(x) & \text{if }x\in\G_+\\
			u_n^-(x) & \text{if }x\in\G_-\\
			0 & \text{if }x\in\G_0\,,
		\end{cases}
		\]
		we have $v_n\in\NN_{p,\lambda}^{nod}(\G)$ for every $n$ (by construction $v_n\in H^1(\G)$ because $u_n^\pm\equiv0$ on the boundary of $\overline{\G_\pm}$) and $J_{p,\lambda}(v_n,\G)\to0$ as $n\to\infty$. This proved that $\JJ_{p,\G}^{nod}(\lambda)=0$ for every $\lambda\leq-\lambda_2(\G)$, which also ensures that it is not attained by \eqref{B-1}.
		
		\smallskip
		{\em Step 2: continuity of $\JJ_{p,\G}^{nod}$.} By \eqref{eq:Jnod_segno} and \eqref{eq:lev_l_up_down}, to show that $\JJ_{p,\G}^{nod}$ is continuous on $\R$ we are left to prove the continuity on $(-\lambda_2(\G),+\infty)$. Let then $\lambda>-\lambda_2(\G)$ be fixed and $\lambda_n\to\lambda$ as $n\to\infty$, so that $\lambda_{n} \geq \alpha > -\lambda_{2}(\mathcal{G})$ for some fixed $\alpha$ and for sufficiently large $n$. By Step 1, it thus exists $u_n \in \mathcal{N}_{p,\lambda_n}^{nod}(\mathcal{G})$ such that $J_{p,\lambda_n}(u_n, \mathcal{G})=\mathcal{J}_{p,\mathcal{G}}^{nod}(\lambda_n)$. The  boundedness of $\lambda_n$ yields that of $\mathcal{J}_{p,\mathcal{G}}^{nod}(\lambda_n)$ by \eqref{eq:lev_l_up_down}, which together with the compactness of $\G$ gives that of $(u_n)_n$ in $H^1(\G)$.  Hence, there exists $u\in H^1(\G)$ such that, up to subsequences, $u_n\rightharpoonup u$ in $H^1(\G)$ and $u_n\to u$ in $L^q(\G)$ for every $q\geq2$ as $n\to\infty$. As a consequence, by \eqref{B-1} and \eqref{eq:lev_l_up_down}
		\[
		\left(\frac{1}{2}-\frac{1}{p}\right)\|u\|_{p}^p = \lim_{n\to +\infty} \left(\frac{1}{2}-\frac{1}{p}\right)\|u_n\|_{p}^p \geq  \liminf_{n\to\infty} \mathcal{J}_{p, \mathcal{G}}^{nod}(\lambda_n) \geq C_2 \min\left(1, \frac{\lambda + \lambda_2}{\lambda_2}\right)^{\frac{p}{p-2}} > 0\,,
		\]
		so that $u\not\equiv0$ on $\G$. Moreover, $u\in\NN_{p,\lambda}^{nod}(\G)$. Indeed, since $(u_n,\lambda_n)$ satisfy \eqref{eq:nlsnomass} and $\lambda_n\to\lambda$, passing to the limit in the equation entails
		\begin{equation}\label{add-New-2}
			\begin{cases}
				u'' + |u|^{p-2} u = \lambda u, &  \text{on each edge} \ e \in \mathbb{E}_\G, \\
				u\text{ is continuous} & \text{on }\G\\
				\sum_{e\succ \vv} u_e'(\vv)=0 & \text{at every vertex } \vv\in \mathbb{V}_\G\,. 
			\end{cases}
		\end{equation}
		Now, if $u\geq0$ on $\G$, then \eqref{add-New-2} implies $u>0$ on $\G$. Indeed, Kirchhoff conditions ensure that $u'$ would be zero at any global minimum point of $u$ on $\G$, and if $u$ and $u'$ vanished at the same point, the well-posedness for the Cauchy problem associated to \eqref{add-New-2} would yield $u\equiv0$ on $\G$, which is impossible here. However, if $u>0$ everywhere on $\G$, then evaluating the first line of \eqref{add-New-2} at a local minimum point of $u$ gives $\lambda>0$. The fact that $u_n\in\NN_{p,\lambda_n}^{nod}(\G)$, the convergence of $u_n$ to $u$ and Proposition \ref{prop:AGS} then lead to
		\begin{equation*}
			\left(\frac{1}{2}-\frac{1}{p}\right) \|u^-\|_{p}^{p}=\lim_{n\to \infty}\left(\frac{1}{2}-\frac{1}{p}\right) \|u_{n}^{-}\|_{p}^{p}=\lim_{n\to \infty}  J_{p,\lambda_n}(u_{n}^{-},\mathcal{G}) \geq \liminf_{n\to\infty} \mathcal{J}_{p,\mathcal{G}}(\lambda_n)=\mathcal{J}_{p,\mathcal{G}}(\lambda)>0,
		\end{equation*}
		which is a contradiction with the assumption that $u\geq 0$.
		Hence, $u\in \mathcal{N}_{p,\lambda}^{nod}(\mathcal{G})$ and
		\begin{equation}\label{add-new-3}
			\lim_{n\to\infty} \mathcal{J}_{p,\mathcal{G}}^{nod}(\lambda_n)
			=\lim_{n\to\infty} \left(\frac{1}{2}-\frac{1}{p}\right)\|u_n\|_{p}^{p}
			=\left(\frac{1}{2}-\frac{1}{p}\right)\|u\|_{p}^{p} \geq \mathcal{J}_{p,\mathcal{G}}^{nod}(\lambda).
		\end{equation}
		Conversely, taking $\tilde{u} \in \mathcal{N}_{p,\lambda}^{nod}(\mathcal{G})$ such that $J_{p,\lambda}(\tilde{u}, \mathcal{G}) =\mathcal{J}_{p,\mathcal{G}}^{nod}(\lambda)$, we have 
		\begin{equation*}
			\begin{aligned}
				n_{\lambda_n,p}^{p-2}(\widetilde{u}^{\pm})= \frac{\|(\widetilde{u}^{\pm})'\|_{2}^{2}+\lambda_n \|\widetilde{u}^{\pm}\|_{2}^{2}}{\|\widetilde{u}^{\pm}\|_{p}^{p}} &=\frac{\|(\widetilde{u}^{\pm})'\|_{2}^{2}+\lambda \|\widetilde{u}^{\pm}\|_{2}^{2}}{\|\widetilde{u}^{\pm}\|_{p}^{p}} + 
				\frac{(\lambda_n-\lambda)\|\widetilde{u}^{\pm}\|_{2}^{2}}{\|\widetilde{u}^{\pm}\|_{p}^{p}}\\
				&=1 + (\lambda_n-\lambda)\frac{\|\widetilde{u}^{\pm}\|_{2}^{2}}{\|\widetilde{u}^{\pm}\|_{p}^{p}} \to 1\ \ \text{as} \ n \to+\infty\,,
			\end{aligned}
		\end{equation*}
		so that
		\begin{equation*}
			\begin{aligned}
				\mathcal{J}_{p,\mathcal{G}}^{nod}(\lambda_n) &\leq J_{p,\lambda_n}(n_{p,\lambda_n}(\widetilde{u}^+)\widetilde{u}^+ + n_{p,\lambda_n}(\widetilde{u}^-)\widetilde{u}^-) \\
				& =\left(\frac{1}{2} -\frac{1}{p} \right) n_{p,\lambda_n}^p(\widetilde{u}^+) \|\widetilde{u}^+\|_{p}^{p}+ \left(\frac{1}{2} -\frac{1}{p} \right) n_{p,\lambda_n}^p(\widetilde{u}^-) \|\widetilde{u}^-\|_{p}^{p} \\
				&=\left(\frac{1}{2} -\frac{1}{p} \right) \|\widetilde{u}\|_{p}^{p}+o_{n}(1)= \mathcal{J}_{p,\mathcal{G}}^{nod}(\lambda)+o_{n}(1)\,.
			\end{aligned}
		\end{equation*}
		Therefore, 
		\begin{equation*}
			\lim_{n \to\infty}  \mathcal{J}_{p,\mathcal{G}}^{nod}(\lambda_n)\leq \mathcal{J}_{p,\mathcal{G}}^{nod}(\lambda),
		\end{equation*}
		which coupled with \eqref{add-new-3} gives the desired continuity of $\JJ_{p,\G}^{nod}$ at $\lambda$. Observe also that, up to subsequences, this argument proves that nodal action ground states in $\NN_{p,\lambda_n}^{nod}(\G)$ converge strongly in $L^q(\G)$, for every $q\geq2$, to nodal action ground states in $\NN_{p,\lambda}^{nod}(\G)$ as $\lambda_n\to\lambda$, for every $\lambda>-\lambda_2(\G)$. 
		
		\smallskip
		{\em Step 3: conclusion of the proof.} The strong convergence in $L^q(\G)$ of nodal action ground states obtained at the end of Step 2 ensures that $Q_{p,\G}^{nod,\pm}(\lambda)$ are attained for every $\lambda>-\lambda_2(\G)$. 
		
		Such a convergence also implies that the $L^p$ norm of both the positive and negative parts of nodal action ground states in $\NN_{p,\lambda}^{nod}(\G)$ is uniformly bounded away from zero whenever $\lambda$ is uniformly bounded away from $-\lambda_2(\G)$. Indeed, if this were not the case, there would exist $(\lambda_n)_n$ bounded away from $-\lambda_2(\G)$ and nodal action ground states $u_n\in\NN_{p,\lambda_n}^{nod}(\G)$ such that $\|u_n^+\|_p\to0$ as $n\to\infty$. Since $u_n^+\in\NN_{p,\lambda_n}(\G)$, this would imply $\JJ_{p,\G}(\lambda_n)\to0$ too, and thus $\displaystyle\limsup_{n\to\infty}\lambda_n\leq0$ by Proposition \ref{prop:AGS}. Hence, up to subsequences $\lambda_n\to\lambda>-\lambda_2(\G)$ and, by the already established convergence of nodal action ground states, $u_n\to u$ in $L^q(\G)$ for every $q\geq2$, for some nodal action ground state $u\in\NN_{p,\lambda}^{nod}(\G)$. The strong convergence of $u_n^+$ to $u$ would then yield $u^+\equiv0$, contradicting $u\in\NN_{p,\lambda}^{nod}(\G)$. 
		
		The previous property is needed to repeat verbatim the argument in the proof of \cite[Proposition 3.2]{DDGS}, which finally proves the monotonicity of $\JJ_{p,\G}^{nod}$, the estimates in \eqref{eq:derJnod}, and the fact that $\JJ_{p,\G}^{nod}$ is locally Lipschitz on $\R$, and thus differentiable almost everywhere. 
		 \end{proof}
		 
		 \begin{remark}
		 	\label{rem:2nod}
		 	Every nodal action ground state in $\NN_{p,\lambda}^{nod}(\G)$ has exactly two nodal regions. This is given by the argument in \cite{BW}, but it can also be seen directly as follows. Assume by contradiction that there exists a nodal action ground state $u\in\NN_{p,\lambda}^{nod}(\G)$ with more than two nodal regions. Let $\G_1$ be a connected component of $\text{supp}(u^+)$ and $\G_2$ be a connected component of $\text{supp}(u^-)$. Let then $u_1:=u\chi_{\G_1}$, $u_2:=u\chi_{\G_2}$, where $\chi_{\G_1},\,\chi_{\G_2}$ are the characteristic functions of $\G_1,\,\G_2$, respectively. Since $u$ is a nodal action ground state, $J_{p,\lambda}'(u,\G)u_1=J_{p,\lambda}'(u_1,\G)u_1=0$ and $J_{p,\lambda}'(u,\G)u_2=J_{p,\lambda}'(u_2,\G)u_2=0$. Hence, $u_1,u_2\in\NN_{p,\lambda}(\G)$, so that $u_1+u_2\in\NN_{p,\lambda}^{nod}(\G)$. Moreover,  since by assumption $w:=u-u_1-u_2\not\equiv0$ on $\G$, one obtains
		 	\[
		 	\JJ_{p,\G}^{nod}(\lambda)=J_{p,\lambda}(u,\G)=J_{p,\lambda}(w,\G)+J_{p,\lambda}(u_1+u_2,\G)>J_{p,\lambda}(u_1+u_2,\G)\geq\JJ_{p,\G}^{nod}(\lambda)\,,
		 	\]
		 	which is the contradiction we seek.
		 \end{remark}
		 
		 We conclude this section with the next asymptotic results.
		 \begin{proposition}
		 	\label{prop:as}
		For every $p\in(2,6)$, it holds
		\begin{equation}
			\label{eq:asJ/l}
			\lim_{\lambda\to+\infty}\frac{\JJ_{p,\G}^{nod}(\lambda)}\lambda=+\infty\,.
		\end{equation}
		Moreover, 	let $\mu_\G^{nod}$ be the number defined in \eqref{eq:muGnod}. There exists $\gamma>0$ such that
		\begin{equation}
			\label{eq:asJ6}
			\JJ_{6,\G}^{nod}(\lambda)=\frac{\lambda\mu_\G^{nod}}2+o\left(e^{-\gamma\sqrt\lambda}\right)\qquad\text{as }\lambda\to+\infty\,.
		\end{equation}
		Furthermore, for every $\varepsilon>0$ there exists $\widetilde\lambda$ (depending on $\G$ and $\varepsilon$) such that
		\begin{equation}
			\label{eq:asM6}
			\mu_\G^{nod}-\varepsilon\leq Q_{6,\G}^{nod,-}(\lambda)\leq Q_{6,\G}^{nod,+}(\lambda)\leq \mu_\G^{nod}+\varepsilon\qquad\forall \lambda\geq\widetilde\lambda\,.
		\end{equation}
		\end{proposition}
		\begin{proof}
		Observe first that, for every $\lambda\in\R$ and every $v\in\NN_{p,\lambda}(\G)$, we have
		\begin{equation}
		\label{eq:J>E}
		J_{p,\lambda}(v,\G)\geq\EE_{p,\G}(\mu)+\frac\lambda2\mu\qquad\forall \mu>0.
		\end{equation}
		Indeed, since $v\in\NN_{p,\lambda}(\G)$, then $J_{p,\lambda}(v,\G)\geq J_{p,\lambda}(tv,\G)$ for every $t>0$, that is
		\[
			J_{p,\lambda}(v,\G)\geq J_{p,\lambda}(tv,\G)=E_{p}(tv,\G)+\frac{\lambda\|tv\|_2^2}{2}\geq\EE_{p,\G}(\|tv\|_2^2)+\frac{\lambda\|tv\|_2^2}{2}\qquad\forall t>0\,,
		\]
		that is \eqref{eq:J>E} since $\|tv\|_2$ covers the whole interval $(0,+\infty)$ as $t$ varies in $(0,+\infty)$.
		
		Let now $p\in(2,6)$. For fixed $\mu>0$, taking $v$ to be a nodal action ground state in $\NN_{p,\lambda}^{nod}(\G)$ as $\lambda\to+\infty$, applying \eqref{eq:J>E} to $v^+,v^-$ and summing together gives
		\[
		\liminf_{\lambda\to+\infty}\frac{\JJ_{p,\G}^{nod}(\lambda)}{\lambda}\geq \liminf_{\lambda\to+\infty}\left(\frac{2\EE_{p,\G}(\mu)}\lambda+\mu\right)=\mu\,.
		\]
		By the arbitrariness of $\mu>0$,  \eqref{eq:asJ/l} follows.
		
		Let then $p=6$. Since the actual value of $\mu_\G^{nod}$ depends on the total number of pendants in $\G$, consider first the case of a graph with at least two pendants. By Proposition \ref{prop:NAGS} and Remark \ref{rem:2nod}, if $\lambda$ is large, 
 nodal action ground states $u_\lambda\in\NN_{6,\lambda}^{nod}(\G)$ exist and always have exactly two nodal regions. Let $\G_\lambda^\pm$ be the support of $u_\lambda^\pm$, and observe that there exists $\alpha>0$ independent of $\lambda$ such that
		\begin{equation}
			\label{eq:supp>a}
		|\G_\lambda^\pm|\geq\alpha\,.
		\end{equation}
		Indeed, if this were not the case, there would exist $\lambda_n\to+\infty$ and, for instance, $|\G_{\lambda_n}^-|\to0$ as $n\to\infty$. Letting $x_n\in\G$ be such that $\|u_{\lambda_n}^-\|_\infty=-u_{\lambda_n}(x_n)$, this would imply that $u_{\lambda_n}(x_n+\delta_n)=0$ for some $\delta_n\to0^+$ as $n\to\infty$. Note that, since $\G$ has finitely many edges, without loss of generality we can always assume that the points $x_n+\delta_n$, $x_n+3\delta_n$, $x_n+5\delta_n$ belong to the interior of the same edge of $\G$ for every $n$. However, since on every edge of the graph $u_{\lambda_n}$ is the restriction of a periodic solution to \eqref{eq:nlsnomass}, it would then follow that $u_{\lambda_n}$ vanish at these three points and it changes sign when crossing each of them. Therefore, $u_{\lambda_n}$ would be strictly negative on $(x_n+3\delta_n,x_n+5\delta_n)$ and strictly positive both in a left neighborhood of $x_n+3\delta_n$ and in a right neighborhood of $x_n+5\delta_n$, namely $(x_n+3\delta_n,x_n+5\delta_n)$ would be a subset of $\G_{\lambda_n}^-$ not connected to $x_n$, which belongs to $\G_{\lambda_n}^-$ too by assumption. This is impossible, as $\G_{\lambda_n}^{-}$ is connected by Remark \ref{rem:2nod}, and we obtain \eqref{eq:supp>a}.
		
		Observe that $\G_\lambda^\pm$ contain at least one pendant each. Indeed, 
		denoting by $m$ the length of the shortest pendant in $\G$, since there are at least two pendants it holds
		\begin{equation}
			\label{eq:supJ}
		\JJ_{6,\G}^{nod}(\lambda)\leq 2\inf_{\substack{v\in\NN_{6,\lambda}([0,m]) \\ v(m)=0}}J_{6,\lambda}(v,[0,m])=2\JJ_{6,\R^+}(\lambda)+o\left(e^{-\gamma\sqrt\lambda}\right)\qquad\text{as }\lambda\to+\infty
		\end{equation}
		(the asymptotic expansion of the action ground state level on $[0,m]$ with homogeneous Dirichlet condition at the final endpoint of the interval is given by \cite[Proposition 3.2]{D25}). Assume then by contradiction that $\G_\lambda^-$ does not contain any pendant, namely $u_\lambda^-$ vanishes at every vertex of degree one of $\G$. Let then $\widetilde\G$ be the graph obtained by $\G$ attaching to each of its vertices of degree one an additional loop of length $1$.  Extending $u_\lambda^-$ to zero on the additional loops of $\widetilde\G$, we can then think of $u_\lambda^-$ as a function in $\NN_{6,\lambda}(\widetilde\G)$. Since by construction $\widetilde\G$ is a graph with no pendant, by \cite[Proposition 3.2]{D25} it would follow
		\[
		J_{6,\lambda}(u_\lambda^-,\G)=J_{6,\lambda}(u_\lambda^-,\widetilde\G)\geq\JJ_{6,\widetilde\G}(\lambda)=\JJ_{6,\R}(\lambda)+o\left(e^{-\gamma_1\sqrt\lambda}\right)
		\]
		as $\lambda\to+\infty$, where the constant $\gamma_1>0$ depends only on the fixed graph $\widetilde\G$. Conversely, since $u_\lambda^+\in\NN_{6,\lambda}(\G)$ and $\G$ has pendants, by \cite[Proposition 3.2]{D25} again
		\begin{equation}
			\label{eq:Jbelu+}
		J_{6,\lambda}(u_\lambda^+,\G)\geq\JJ_{6,\G}(\lambda)=\JJ_{6,\R^+}(\lambda)+o\left(e^{-\gamma_2\sqrt\lambda}\right)\qquad\text{as }\lambda\to+\infty
		\end{equation}
		with $\gamma_2>0$ depending only on $\G$. Coupling with the previous estimate contradicts \eqref{eq:supJ} as soon as $\lambda$ is large enough, showing that $\G_\lambda^-$ must contain at least one pendant. The same argument works for $\G_\lambda^+$.
		
		Since both $\G_\lambda^\pm$ contain a vertex of degree one and they are connected sets according to Remark \ref{rem:2nod}, by \eqref{eq:supp>a} it follows that they both contain an interval of length at least $\sigma:=\min\left\{\alpha,m\right\}$ starting at a vertex of degree one (here $\alpha$ is the number in \eqref{eq:supp>a} and $m$ is the length of the shortest pendant of $\G$). Hence, if we consider the minimization problems
		\[
		\widetilde\JJ_{6,\G_\lambda^\pm}(\lambda):=\inf\left\{ J_{6,\lambda}(v,\G_\lambda^\pm)\,:\, v\in\NN_{6,\lambda}(\G_\lambda^\pm)\,,\quad v=0\text{ on }\overline{\G_\lambda^\pm}\setminus\G_\lambda^\pm\right\}\,,
		\]
		we have
		\begin{equation}
		\label{eq:Jtildeup}
		\widetilde\JJ_{6\,\G_\lambda^\pm}(\lambda)\leq\inf_{\substack{v\in\NN_{6,\lambda}([0,\sigma]) \\ v(\sigma)=0}}J_{6,\lambda}(v,[0,\sigma])=\JJ_{6,\R^+}(\lambda)+ o\left(e^{-\gamma_3\sqrt\lambda}\right)\qquad\text{as }\lambda\to+\infty
		\end{equation}
		where $\gamma_3>0$ depends only on $\sigma$ again by \cite[Proposition 3.2]{D25}. 
		
		Observe now that $u_\lambda^\pm$ belong to the minimization spaces defining $\widetilde\JJ_{6,\G_\lambda^\pm}(\lambda)$ and they satisfy $J_{6,\lambda}(u_\lambda^\pm,\G_\lambda^\pm)=\widetilde\JJ_{6,\G_\lambda^\pm}(\lambda)$. That $\widetilde{\JJ}_{6,\lambda}(\G_\lambda^\pm)$ is attained for every sufficiently large $\lambda$ is guaranteed by Remark \ref{rem:extAGS}, and if $u_\lambda^\pm$ were not action ground states for these problems, one could replace them with the actual ground states to build a new function in $\NN_{6,\lambda}^{nod}(\G)$ with a lower action level, violating that $u_\lambda$ is a nodal action ground state in $\NN_{6,\lambda}^{nod}(\G)$. Hence, by \eqref{eq:Jtildeup} we obtain
		\[
		J_{6,\lambda}(u_\lambda^\pm,\G)\leq\JJ_{6,\R^+}(\lambda)+o\left(e^{-\gamma_3\sqrt\lambda}\right)\qquad\text{as }\lambda\to+\infty\,,
		\]
		and since \eqref{eq:Jbelu+} holds also for $u_\lambda^-$, it yields 
		\begin{equation}
		\label{eq:asJpm}
		J_{6,\lambda}(u_\lambda^{\pm},\G)=\JJ_{6,\R^+}(\lambda)+o\left(e^{-\gamma\sqrt\lambda}\right)\qquad\text{as }\lambda\to+\infty
		\end{equation}
		for a suitable $\gamma>0$ independent of $\lambda$ and of $u_\lambda$.  Recalling that $\displaystyle \JJ_{6,\R^+}(\lambda)=\lambda\mu_{\R^+}/2$ for every $\lambda$, and that $\mu_\G^{nod}=2\mu_{\R^+}$ by \eqref{eq:muGnod}, this completes the proof of \eqref{eq:asJ6} for graphs $\G$ with at least two pendants. Relying on \eqref{eq:asJpm} and arguing exactly as in the final part of the proof of \cite[Proposition 3.2]{D25} then gives 
		\[
		\lim_{\lambda\to+\infty}\|u_\lambda^\pm\|_2^2=\mu_{\R^+}
		\]
		uniformly on the set of nodal action ground states $u_\lambda\in\NN_{6,\G}^{nod}(\lambda)$, in turn implying \eqref{eq:asM6}.

		The remaining cases are treated identically, simply noting that if $\G$ has exactly one pendant, then with no loss of generality one can assume that it always belongs to $\G_\lambda^+$, so that the previous argument now leads to
		\[
		J_{6,\lambda}(u_\lambda^+,\G)=\frac{\lambda\mu_{\R^+}}2+o\left(e^{-\gamma\sqrt\lambda}\right)\,,\qquad J_{6,\lambda}(u_\lambda^-,\G)=\frac{\lambda\mu_{\R}}2+o\left(e^{-\gamma\sqrt\lambda}\right)
		\]
		and 
		\[
		\lim_{\lambda\to+\infty}\|u_\lambda^+\|_{2}^2=\mu_{\R^+}\,,\qquad\lim_{\lambda\to+\infty}\|u_\lambda^-\|_{2}^2=\mu_{\R}\,,
		\]
		whereas when $\G$ has no pendant the same will be true for both $\G_\lambda^+$ and $\G_\lambda^-$, entailing
		\[
		J_{6,\lambda}(u_\lambda^\pm,\G)=\frac{\lambda\mu_{\R}}2+o\left(e^{-\gamma\sqrt\lambda}\right)\,,\qquad \lim_{\lambda\to+\infty}\|u_\lambda^\pm\|_{2}^2=\mu_{\R}\,.\qedhere
		\]
		\end{proof} 
		
		\section{Proof of Theorems \ref{thm:nodex}--\ref{thm:nonun}}
		\label{sec:exun}
		
		This section is devoted to the proof of existence of least energy nodal normalized solutions as given in Theorem \ref{thm:nodex} and of the conditional non-uniqueness result in Theorem \ref{thm:nonun}.

		\begin{proof}[Proof of Theorem \ref{thm:nodex}]
			The main step is to show that, under the assumptions of Theorem \ref{thm:nodex}, the function $f_\mu:\R\to\R$ defined by
			\[
			f_\mu(\lambda):=\JJ_{p,\G}^{nod}(\lambda)-\frac{\lambda\mu}{2}
			\]
			has a global minimum point $\widetilde\lambda$ such that $\widetilde{\lambda}\in(-\lambda_2(\G),+\infty)$. Note first that $f_\mu$ is continuous on $\R$ by Proposition \ref{prop:NAGS}, and one has
			\begin{equation}
			\label{eq:fmu1}
			\JJ_{p,\G}^{nod}(\lambda)\geq\JJ_{p,\G}^{nod}(-\lambda_2(\G))\qquad\forall\lambda\leq-\lambda_2(\G)
			\end{equation}
			and, by the first estimate in \eqref{eq:lev_l_up_down},
			\begin{equation}
			\label{eq:fmu2}
			f_\mu(\lambda)-f_\mu(-\lambda_2(\G))=\JJ_{p,\G}^{nod}(\lambda)-\frac{(\lambda+\lambda_2(\G))\mu}{2}<0\qquad\text{as }\lambda\to-\lambda_2(\G)^+\,.
			\end{equation}
			Moreover, writing
			\[
			f_\mu(\lambda)=\lambda\left(\frac{\JJ_{p,\G}^{nod}(\lambda)}\lambda-\frac\mu2\right),
			\]
			as $\lambda\to+\infty$ by Proposition \ref{prop:as} we obtain
			\begin{equation}
			\label{eq:fmu3}
			f_\mu(\lambda)\to+\infty
			\end{equation}
			for every $\mu>0$ if $p\in(2,6)$ and every $\mu\in(0,\mu_\G^{nod})$ if $p=6$, whereas
			\begin{equation}
			\label{eq:fmu4}
			f_\mu(\lambda)\to0
			\end{equation}
			if $\mu=\mu_\G^{nod}$ and $p=6$. Combining \eqref{eq:fmu1}, \eqref{eq:fmu2} and \eqref{eq:fmu3} proves the existence of the desired global minimizer of $f_\mu$ for every $\mu>0$ if $p\in(2,6)$ and every $\mu\in(0,\mu_\G^{nod})$ if $p=6$. When $\mu=\mu_\G^{nod}$ and $p=6$, if we further assume that $\EE_{6,\G}^{nod}(\mu_\G^{nod})<0$, then there exist $\overline\lambda\in\R$ and $\overline u\in\NN_{6,\overline\lambda}^{nod}(\G)$ such that $\|\overline u\|_2^2=\mu_\G^{nod}$ and $E_{6}(\overline u,\G)<0$. Hence, 
			\[
			f_{\mu_\G^{nod}}(\overline\lambda)\leq J_{6,\overline\lambda}(\overline u,\G)-\frac{\overline\lambda\mu_\G^{nod}}2=E_6(\overline u,\G)<0\,,
			\]
			that together with \eqref{eq:fmu1}, \eqref{eq:fmu2} and \eqref{eq:fmu4} gives again a global minimum point $\widetilde\lambda$ for $f_{\mu_\G^{nod}}$. 
			
			We now show that $\JJ_{p,\G}^{nod}$ is differentiable at $\widetilde\lambda$. Indeed, taking $\varepsilon<0$, we have
			\[
			\begin{split}
				\frac{\JJ_{p,\G}^{nod}(\widetilde\lambda+\varepsilon)-\JJ_{p,\G}^{nod}(\widetilde\lambda)}{\varepsilon}-\frac\mu2&\,=\frac{\JJ_{p,\G}^{nod}(\widetilde\lambda+\varepsilon)-(\widetilde\lambda+\varepsilon)\mu/2-\JJ_{p,\G}^{nod}(\widetilde\lambda)+\widetilde\lambda\mu/2}{\varepsilon}\\
				&\,=\frac{f_\mu(\widetilde\lambda+\varepsilon)-f_\mu(\widetilde\lambda)}{\varepsilon}\leq0
			\end{split}
			\]
			leading to
			\[
			\limsup_{\varepsilon\to0^-}\frac{\JJ_{p,\G}^{nod}(\widetilde\lambda+\varepsilon)-\JJ_{p,\G}^{nod}(\widetilde\lambda)}{\varepsilon}\leq\frac\mu2\,.
			\]
			Arguing analogously with $\varepsilon>0$ yields
			\[
			\liminf_{\varepsilon\to0^+}\frac{\JJ_{p,\G}^{nod}(\widetilde\lambda+\varepsilon)-\JJ_{p,\G}^{nod}(\widetilde\lambda)}{\varepsilon}\geq\frac\mu2\,,
			\]
			and coupling these two estimates with \eqref{eq:derJnod} gives  $(\JJ_{p,\G}^{nod})'(\widetilde\lambda)=\mu/2$. By Proposition \ref{prop:NAGS}, this implies that every nodal action ground state $\widetilde u\in\NN_{p,\widetilde\lambda}^{nod}(\G)$ has mass $\mu$ (ground states exist because $\widetilde\lambda>-\lambda_2(\G)$). Moreover, if $w$ is any sign-changing solution to \eqref{eq:nls} with mass $\mu$ for some $\lambda\in\R$, then $w\in\NN_{p,\lambda}^{nod}(\G)$ and therefore
			\[
			\begin{split}
			E_p(w,\G)=J_{p,\lambda}(w,\G)-\frac{\lambda\|w\|_2^2}2=&\,J_{p,\lambda}(w,\G)-\frac{\lambda\mu}2\geq\JJ_{p,\G}^{nod}(\lambda)-\frac{\lambda\mu}2=f_\mu(\lambda)\geq f_\mu(\widetilde{\lambda})\\
			=&\,\JJ_{p,\G}^{nod}(\widetilde{\lambda})-\frac{\widetilde\lambda\mu}{2}=J_{p,\widetilde{\lambda}}(\widetilde u,\G)-\frac{\widetilde{\lambda}\|\widetilde{u}\|_2^2}2=E_p(\widetilde u,\G)\,.
			\end{split}
			\]
			i.e. $\widetilde u$ is a least energy nodal normalized solution of \eqref{eq:nls} with mass $\mu$. Finally, if $w$ is another least energy nodal normalized solution of \eqref{eq:nls} with mass $\mu$, every inequality in the previous lines becomes an equality. In particular, $J_{p,\lambda}(w,\G)=\JJ_{p,\G}^{nod}(\lambda)$ and $w$ is a nodal action ground state in $\NN_{p,\lambda}^{nod}(\G)$. 
			\end{proof}
		
			For the proof of the non-uniqueness result, we need first to establish some key properties for the Lagrange multipliers of least energy nodal normalized solutions. Denote by 
			\[
			\mathfrak{E}_{p,\G}^{nod}(\mu):=\left\{u\in H_\mu^1(\G)\,:\, u\text{ is a sign-changing solution of \eqref{eq:nls} for some }\lambda\in\R,\,E_p(u,\G)=\EE_{p,\G}^{nod}(\mu)\right\}
			\]
			the set of all least energy nodal normalized solutions with mass $\mu$, and set
			\[
			\Lambda_{p,\G}^{nod,-}(\mu):=\inf_{u\in\mathfrak{E}_{p,\G}^{nod}(\mu)}\LL_p(u,\G)\,,\qquad\Lambda_{p,\G}^{nod,+}(\mu):=\sup_{u\in\mathfrak{E}_{p,\G}^{nod}(\mu)}\LL_p(u,\G)\,.
			\]
			\begin{remark}
				\label{rem:Lnodatt}
				By Theorem \ref{thm:nodex}, for every $p\in(2,6)$ the set $\mathfrak{E}_{p,\G}^{nod}(\mu)$ is not empty for every $\mu>0$. Furthermore, $\Lambda_{p,\G}^{nod,\pm}(\mu)$ are attained. Indeed, let for instance $(u_n)_n\subset\mathfrak{E}_{p,\G}^{nod}(\mu)$ be such that $\LL_{p}(u_n,\G)\to\Lambda_{p,\G}^{nod,-}(\mu)$.  By the proof of Theorem \ref{thm:nodex}, each $u_n$ is a nodal action ground state in $\NN_{p,\lambda_n}^{nod}(\G)$, where $\lambda_n=\LL_p(u_n,\G)$ is a global minimum point of the function $f_\mu$ on $\R$. Since by \eqref{eq:fmu2} there exists $\varepsilon>0$ such that $f_\mu(\lambda)<f_\mu(-\lambda_2(\G))$ for every $\lambda\in[-\lambda_2(\G), -\lambda_2(\G)+\varepsilon]$, recalling also \eqref{eq:fmu1} shows that there exists $\alpha>-\lambda_2(\G)$ (depending only on $\mu$) such that $\lambda_n\geq\alpha$ for every $n$. Hence, arguing as in Step 3 of the proof of Proposition \ref{prop:NAGS} we obtain that there exists $\beta>0$ (depending again only on $\mu$) such that 
				\begin{equation}
				\label{eq:upm>0}
				\|u_n^\pm\|_p\geq\beta\qquad\forall n\,.
				\end{equation}
				Now, if $p\in(2,6)$, by $E_p(u_n,\G)=\EE_{p,\G}^{nod}(\mu)$ and the standard Gagliardo-Nirenberg inequality
				\begin{equation}
				\label{eq:GN}
				\|v\|_p^p\leq K_p\|v\|_2^{\frac p2+1}\|v\|_{H^1}^{\frac p2-1}\qquad\forall v\in H^1(\G)
				\end{equation}
				it follows that $(u_n)_n$ is bounded in $H^1(\G)$. Hence, up to subsequences $u_n\rightharpoonup u$ in $H^1(\G)$ and $u_n\to u$ in $L^q(\G)$ for every $q\geq2$, for some $u\in H^1(\G)$. Since $\|u_n\|_2^2=\mu$ for every $n$, then $\|u\|_2^2=\mu$ too, so that $u\not\equiv0$ on $\G$ and $u\in H_\mu^1(\G)$. Moreover, by \eqref{eq:upm>0} and the strong convergence of $u_n$ to $u$ in $L^p(\G)$ we also have that $u^\pm\not\equiv0$ on $\G$, that is $u$ is sign-changing. Indeed, if it were for instance $u\geq0$ on $\G$, by the fact that $u_n\to u$ a.e. we would have $u_n^-\to0$ a.e., contradicting \eqref{eq:upm>0}.  Furthermore, since $\LL_p(u_n,\G)\to\Lambda_{p,\G}^{nod,-}(\mu)$, then $u$ solves \eqref{eq:nls} with $\lambda=\Lambda_{p,\G}^{nod,-}(\mu)$.  Hence, $u\in\mathfrak{E}_{p,\G}^{nod}(\mu)$ and, by weak lower semicontinuity, 
				\[
				\EE_{p,\G}^{nod}(\mu)\leq E_p(u,\G)\leq\lim_{n \to\infty}E_p(u_n,\G)=\EE_{p,\G}^{nod}(\mu)\,,
				\]
				i.e. $u\in\mathfrak{E}_{p,\G}^{nod}(\mu)$ and $\Lambda_{p,\G}^{nod,-}(\mu)$ is attained. The same is true for $\Lambda_{p,\G}^{nod,+}(\mu)$. 
				
				Note that the same argument also shows that, if $u_\mu\in\mathfrak{E}_{p,\G}^{nod}(\mu)$ and $\mu\to\overline\mu>0$, then there exists $u_{\overline\mu}\in\mathfrak{E}_{p,\G}^{nod}(\overline\mu)$ such that, up to subsequences, $u_\mu\rightharpoonup u_{\overline\mu}$ in $H^1(\G)$ and $u_\mu\to u_{\overline\mu}$ in $L^q(\G)$ for every $q\geq2$. This is readily seen by complementing the previous discussion with the fact that the global minimum points of $f_\mu$ converge (up to subsequences) to global minimum points of $f_{\overline\mu}$ as $\mu\to\overline\mu$. This last property follows by the fact that global minimum points of $f_\mu$ are bounded from above and away from $-\lambda_2(\G)$ uniformly on $\mu$ in a neighborhood of any fixed $\overline\mu>0$ by \eqref{eq:fmu2} and \eqref{eq:fmu3}, and that $f_\mu$ is Lipschitz continuous with respect to $\mu$ when $\lambda$ varies in bounded sets of $\R$. 
			\end{remark}
			The next proposition collects the preliminary results we need to prove Theorem \ref{thm:nonun}, showing that properties known so far for energy ground states hold true also for least energy nodal normalized solutions.
			\begin{proposition}
				\label{prop:Lpm}
				Let $p\in(2,6)$. The following properties hold:
				\begin{itemize}
					\item[(i)] if $0<\mu_1<\mu_2$, then $\Lambda_{p,\G}^{nod,+}(\mu_1)<\Lambda_{p,\G}^{nod,-}(\mu_2)$;
					
					\item[(ii)] there exists an at most countable set $Z_{p,\G}^{nod}\subset\R^+$ such that
					\[
					\Lambda_{p,\G}^{nod,-}(\mu)=\Lambda_{p,\G}^{nod,+}(\mu)=:\Lambda_{p,\G}^{nod}(\mu)\qquad\forall\mu\in\R^+\setminus Z_{p,\G}^{nod}\,;
					\]
					\item[(iii)] the map $\Lambda_{p,\G}^{nod}:\R^+\setminus Z_{p,\G}^{nod}\to\R$ is strictly increasing, and
					\begin{equation}
						\label{eq:limLpm}
					\lim_{\mu\to0^+}\Lambda_{p,\G}^{nod}(\mu)=-\lambda_2(\G)\,,\qquad\lim_{\mu\to+\infty}\Lambda_{p,\G}^{nod}(\mu)=+\infty\,.
					\end{equation}
				\end{itemize}
			\end{proposition}
			\begin{proof}
				To prove (i), let $0<\mu_1<\mu_2$ and $u_1\in\mathfrak{E}_{p,\G}^{nod}(\mu_1)$, $u_2\in\mathfrak{E}_{p,\G}^{nod}(\mu_2)$ be such that $\Lambda_{p,\G}^{nod,+}(\mu_1)=\LL_p(u_1,\G)=:  \widehat{\lambda}_1$ and $\Lambda_{p,\G}^{nod,-}(\mu_2)=\LL_p(u_2,\G)=:  \widehat {\lambda}_2$.

(it is possible by Remark \ref{rem:Lnodatt}). By the proof of Theorem \ref{thm:nodex}, 
				\[
				f_{\mu_1}( \widehat{\lambda}_1)=\min_{\lambda\in\R}f_{\mu_1}(\lambda)\,,\qquad f_{\mu_2}( \widehat{\lambda}_2)=\min_{\lambda\in\R}f_{\mu_2}(\lambda)\,.
				\]
				Recalling the definition of $f_\mu$ then yields
				\[
				\begin{cases}
					\JJ_{p,\G}^{nod}(\widehat{\lambda}_1)-\frac{\widehat{\lambda}_1  \mu_1}2\leq\JJ_{p,\G}^{nod}( \widehat{\lambda}_2)-\frac{ \widehat{\lambda}_2 \mu_1}2 & \\
					\JJ_{p,\G}^{nod}( \widehat{\lambda}_2  )-\frac{ \widehat{\lambda}_2  \mu_2}2\leq\JJ_{p,\G}^{nod}(\widehat{\lambda}_1 )-\frac{ \widehat{\lambda}_1 \mu_2}2\,, & 
				\end{cases}
				\]
				that is
				\[
				\frac{\mu_1}2(\widehat{\lambda}_2  - \widehat{\lambda}_1  )\leq\JJ_{p,\G}^{nod}(\widehat{\lambda}_2 )-\JJ_{p,\G}^{nod}(\widehat{\lambda}_1  )\leq\frac{\mu_2}2( \widehat{\lambda}_2  -\widehat{\lambda}_1  )\,.
				\]
				By $\mu_2>\mu_1$ we then have $ \widehat{\lambda}_2  \geq \widehat{\lambda}_1  $. In fact, the inequality is strict, because if it were $\widehat{\lambda}_1  =\widehat{\lambda}_2  $, then $u_1,u_2$ would be two nodal action ground states in the same nodal Nehari set. Since $u_1,u_2$ are also least energy nodal normalized solutions by assumption, Theorem \ref{thm:nodex} would then imply that they have the same mass $\mu_1=\mu_2$, leading to a contradiction. Hence, $\widehat{\lambda}_1 <\widehat{\lambda}_2  $ and (i) is proved.
				
				We now prove (ii). By (i), the maps $\Lambda_{p,\G}^\pm$ are strictly increasing on $\R^+$. Hence, they have at most countably many discontinuity points. Moreover, for every fixed $\bar{\mu}>0$ and for sufficiently small $\delta>0$, by (i) we have
				\begin{equation*}
					\Lambda_{p,\G}^{nod,+}(\bar{\mu}-2\delta)<\Lambda_{p,\G}^{nod,-}(\bar{\mu}-\delta)
					<\Lambda_{p,\G}^{nod,-}(\bar{\mu})<\Lambda_{p,\G}^{nod,-}(\bar{\mu}+\delta)< \Lambda_{p,\G}^{nod,+}(\bar{\mu}+2\delta).
				\end{equation*}
				Therefore, if $\Lambda_{p,\G}^{nod,+}$ is continuous at $\bar{\mu}$, passing to the limit as $\delta \to 0^+$ shows that $\Lambda_{p,\G}^{nod,-}$ is continuous at $\bar\mu$ too, and $\Lambda_{p,\G}^{nod,-}(\bar\mu)=\Lambda_{p,\G}^{nod,+}(\bar\mu)$. The same argument at any continuity point of $\Lambda_{p,\G}^{nod,-}$ allows to conclude that $\Lambda_{p,\G}^{nod,\pm}$ are continuous exactly at the same points, where they also coincide. Since the set  $Z_{p,\mathcal{G}}^{nod} \subset \R^+$ where $\Lambda_{p,\G}^{nod,\pm}$ are not continuous is at most countable, (ii) follows.
				
				Since (i) ensures that the map $\Lambda_{p,\G}^{nod}:\R^+\setminus Z_{p,\G}^{nod}\to\R$ is strictly increasing, to conclude we are left to prove \eqref{eq:limLpm}. We start with the limit as $\mu\to0^+$. Let $u_\mu\in\mathfrak{E}_{p,\G}^{nod}(\mu)$ be a least energy nodal normalized solution with mass $\mu$. By Theorem \ref{thm:nodex}, it is a nodal action ground state in $\NN_{p,\lambda_\mu}^{nod}(\G)$ with $\lambda_\mu:=\LL_p(u_\mu,\G)$. To show that $\lambda_\mu\to-\lambda_2(\G)$ as $\mu\to0^+$, by Proposition \ref{prop:NAGS} it is enough to show that $\JJ_{p,\G}^{nod}(\lambda_\mu)=J_{p,\lambda_\mu}(u_\mu,\G)\to0$, that is $\|u_\mu\|_p\to0$. To see this, note first that, by $u_\mu\in\NN_{p,\lambda_\mu}^{nod}(\G)$, Gagliardo-Nirenberg inequality \eqref{eq:GN} and Young inequality, 
				\[
				\|u_\mu'\|_2^2+\lambda_\mu\|u_\mu\|_2^2=\|u_\mu\|_p^p\leq K_p\|u_\mu\|_2^{\frac p2+1}\|u_\mu\|_{H^1}^{\frac p2-1}\leq \frac{(6-p)K_p^{\frac4{6-p}}}{4}\mu^{\frac{p+2}{6-p}}+\frac{p-2}{4}\|u\|_{H^1}^2\,,
				\]
				in turn yielding
				\[
				\frac{6-p}4\|u_\mu'\|_2^2+\left(\lambda_\mu-\frac{p-2}{4}\right)\mu\leq\frac{(6-p)K_p^{\frac4{6-p}}}4\mu^{\frac{p+2}{6-p}}\,.
				\]
				Since $p\in(2,6)$ and $\lambda_\mu>-\lambda_2(\G)$ for every $\mu>0$, this shows that $\|u_\mu'\|_2$ is bounded as $\mu\to0^+$. Relying again on the Gagliardo-Nirenberg inequality \eqref{eq:GN}, this leads to $\|u_\mu\|_p\to0$ as desired, and the first part of \eqref{eq:limLpm} follows. 
				
				The limit for $\mu\to+\infty$ is a direct consequence of H\"older inequality and the first estimate in \eqref{eq:lev_l_up_down}, since
				\[
				\mu=\|u_\mu\|_2^2\leq |\G|^{1-\frac2p}\|u_\mu\|_p^2= |\G|^{1-\frac2p}\left(\frac{2p}{p-2}\JJ_{p,\G}^{nod}(\lambda_\mu)\right)^{\frac2p}\leq C(\lambda_\mu+\lambda_2(\G))^{\frac2{p-2}}
				\]
				for a suitable $C>0$ depending only on $p$ and $\G$, entailing $\lambda_\mu\to+\infty$ as $\mu\to+\infty$. 
			\end{proof}
			
			\begin{proof}[Proof of Theorem \ref{thm:nonun}]
				
				The line of the proof is the following. First, we prove that, as $p$ varies, the masses of nodal action ground states $Q_{p,\G}^{nod,\pm}$ satisfy, for every fixed $\lambda>-\lambda_2(\G)$ and $\bar p>2$,
				\begin{equation}
					\label{eq:ptopbar}
					\limsup_{p\to \bar{p}} Q_{p,\mathcal{G}}^{nod,+}(\lambda) \leq Q_{\bar{p},\mathcal{G}}^{nod,+}(\lambda)\,,\qquad   \liminf_{p\to \bar{p}}Q_{p,\mathcal{G}}^{nod,-}(\lambda) \geq Q_{\bar{p},\mathcal{G}}^{nod,-}(\lambda)\,.
				\end{equation}
				Second, we show that when $p=6$ there exist two values $\overline\lambda_2>\overline\lambda_1>-\lambda_2(\G)$ such that 
				\begin{equation}
					\label{eq:Qp=6}
					Q_{6,\G}^{nod,+}(\overline\lambda_2)<Q_{6,\G}^{nod,-}(\overline\lambda_1)\,.
				\end{equation}
				Third, by \eqref{eq:ptopbar} we obtain that the existence of two such values of $\lambda$ is preserved for every $p$ sufficiently close to $6$, and combining with the previously established properties of nodal action ground states and least energy nodal normalized solutions this will complete the proof of Theorem \ref{thm:nonun}.
				
				\smallskip
				{\em Step 1: proof of \eqref{eq:ptopbar}.} We begin by showing, for every fixed $\lambda>-\lambda_2(\G)$ and $\bar p>2$, the nodal action ground state level $\JJ_{p,\G}^{nod}(\lambda)$ is continuous as $p\to\bar p$.
				
				Let $p_n\to \bar p$ as $n\to\infty$ and, given any nodal action ground state $v\in\NN_{\bar p,\lambda}^{nod}(\G)$, observe that
				\[
				\lim_{n\to \infty}n_{p_n,\lambda}^{p_n-2}(v^{\pm})=\lim_{n\to \infty} \frac{\|(v^{\pm})'\|_{2}^{2}+\lambda \|v^{\pm}\|_{2}^{2}}{\|v^{\pm}\|_{{p_n}}^{p_n}}=\lim_{n\to \infty}\frac{\|v^{\pm}\|_{{\bar{p}}}^{\bar{p}}}{\|v^{\pm}\|_{{p_n}}^{p_n}}=1\,.
				\]  
				Hence, since by definition $n_{p_n,\lambda}(v^+)v^++n_{p_n,\lambda}(v^-)v^-\in\NN_{p_n,\lambda}^{nod}(\G)$,
				\begin{equation}
					\label{eq:limsupJ}
				\begin{split}
				\limsup_{n\to\infty}\JJ_{p_n,\G}^{nod}(\lambda)&\,\leq\limsup_{n\to\infty} J_{p_n,\lambda}(n_{p_n,\lambda}(v^+)v^++n_{p_n,\lambda}(v^-)v^-,\G)\\
				&\,=\left(\frac12-\frac1{\bar p}\right)\limsup_{n\to\infty}\left(n_{p_n,\lambda}(v^+)^p\|v^+\|_{p_n}^{p_n}+n_{p_n,\lambda}(v^-)^p\|v^-\|_{p_n}^{p_n}\right)\\
				&\,=\left(\frac12-\frac1{\bar p}\right)\left(\|v^+\|_{\bar p}^{\bar p}+\|v^-\|_{\bar p}^{\bar p}\right)=\left(\frac12-\frac1{\bar p}\right)\|v\|_{\bar p}^{\bar p}=\JJ_{\bar p,\G}^{nod}(\lambda)\,.
				\end{split}
				\end{equation}
				Conversely, let $v_n\in\NN_{p_n,\lambda}^{nod}(\G)$ be such that $J_{p_n,\lambda}(v_n,\G)=\JJ_{p_n,\G}^{nod}(\lambda)$. By \eqref{eq:limsupJ}, H\"older inequality, and the fact that
\[ J_{p_n,\lambda}(v_n,\G)=\left(\frac12-\frac1{p_n}\right)\|v_n\|_{p_n}^{p_n}=\left(\frac12-\frac1{p_n}\right)(\|v_n'\|_2^2+\lambda\|v_n\|_2^2)
\]
 and $\lambda>-\lambda_2(\G)$, it is readily seen that $(v_n)_n$ is bounded in $H^1(\G)$. Hence, there exists $\bar v\in H^1(\G)$ such that, up to subsequences, $v_n\rightharpoonup\bar v$ in $H^1(\G)$ and $v_n\to \bar v$ in $L^q(\G)$ for every $q\geq2$ as $n\to\infty$. Since $\lambda$ is fixed, by the second estimate in \eqref{eq:lev_l_up_down} we also have that $\JJ_{p_n,\G}^{nod}(\lambda)$ is bounded away from zero uniformly in $n$, and so is $\|v_n\|_{p_n}$. Since the boundedness in $H^1(\G)$ implies that in $L^\infty(\G)$, by the pointwise convergence almost everywhere of  $v_n$ to $\bar v$ and the Dominated Convergence Theorem we also have 
				\begin{equation}
					\label{eq:convnorm}
					\|v_n\|_{p_n}\to\|\bar v\|_{\bar p}\,,
				\end{equation}
				and thus $\bar v\not\equiv 0$ on $\G$. Moreover, if for instance $\bar v\geq0$ on $\G$, since $\bar v$ solves \eqref{eq:nlsnomass} with $p=\bar p$, then as usual $v>0$ on $\G$ and $\lambda>0$. However, this would then imply
				\[
				\liminf_{n\to\infty}\|v_n^-\|_{p_n}^{p_n}\geq\liminf_{n\to\infty}\frac{2p_n}{p_n-2}\JJ_{p_n,\G}(\lambda)=\frac{2\bar p}{\bar p-2}\JJ_{\bar p,\G}(\lambda)>0
				\]
				by Proposition \ref{prop:AGS} and the continuity in $p$ of $\JJ_{p,\G}$ (see e.g. \cite[Remark 2.3]{D25}). Since this contradicts the assumption that $\bar v^-\equiv0$, it follows that $\bar v$ changes sign on $\G$, i.e. $\bar v\in\NN_{\bar p,\lambda}^{nod}(\G)$. Hence, $\bar v^\pm\not\equiv0$ and arguing as for $v_n$ we also have
				\[
				\|v_n^\pm\|_{p_n}\to\|\bar v^\pm\|_{\bar p}\qquad\text{as }n\to\infty\,,
				\]
				that combined with the strong convergence in $L^{\bar p}(\G)$ gives 
				\[
				n_{\bar p,\lambda}(v_n^\pm)\to1\qquad\text{as }n\to\infty\,.
				\]
				Therefore,  
				\begin{equation*}
					\begin{aligned}
						\mathcal{J}_{\bar{p},\mathcal{G}}^{nod}(\lambda) &\leq J_{\bar{p},\lambda} \big( n_{\bar{p},\lambda}(v_n^+) v_{n}^+ + n_{\bar{p},\lambda}(v_n^-) v_{n}^- \big) =  \left(\frac{1}{2}-\frac{1}{\bar{p}} \right)  \left(n_{\bar{p},\lambda}^{\bar{p}}(v_n^+) \|v_n^+\|_{\bar{p}}^{\bar{p}} + n_{\bar{p},\lambda}^{\bar{p}}(v_n^-) \|v_n^-\|_{\bar{p}}^{\bar{p}}\right)    \\ 
						& =  \left(\frac{1}{2}-\frac{1}{p_n} \right) \left(  \|v_n^+\|_{p_n}^{p_n} + \|v_n^-\|_{p_n}^{p_n}\right)+o_{n}(1) = J_{p_n,\lambda}(v_{n}^{+}, \mathcal{G})+ J_{p_n,\lambda}(v_{n}^{-}, \mathcal{G}) + o_{n}(1) \\
						& = J_{p_n,\lambda}(v_{n}, \mathcal{G})+o_{n}(1)  = \mathcal{J}_{p_n,\mathcal{G}}^{nod}(\lambda) + o_{n}(1)\,,
					\end{aligned}
				\end{equation*}
				that is 
				\[
				\JJ_{\bar p,\G}^{nod}(\lambda)\leq\liminf_{n\to\infty}\JJ_{p_n,\G}^{nod}(\lambda)\,,
				\]
				that together with \eqref{eq:limsupJ} leads to
				\[
				\lim_{p\to\bar p}\JJ_{p,\G}^{nod}(\lambda)=\JJ_{\bar p,\G}^{nod}(\lambda)\,.
				\]
				Observe also that the previous argument shows that, up to subsequences, nodal action ground states $v_n\in\NN_{p_n,\lambda}^{nod}(\G)$ converge to nodal action ground states $\bar v\in\NN_{\bar p,\lambda}^{nod}(\G)$ if $p_n\to\bar p$ as $n\to\infty$, weakly in $H^1(\G)$ and strongly in $L^q(\G)$ for every $q\geq2$. Recalling the definition of $Q_{p,\G}^{nod,\pm}$ and Proposition \ref{prop:NAGS}, this yields \eqref{eq:ptopbar}.
				
				\smallskip
				{\em Step 2: proof of \eqref{eq:Qp=6}.} We want to prove that there exist $\overline\lambda_2>\overline\lambda_1>-\lambda_2(\G)$ for which \eqref{eq:Qp=6} holds. To this end, recall that, under the hypotheses of Theorem \ref{thm:nonun}, Theorem \ref{thm:nodex} ensures the existence of a least energy nodal normalized solution $\overline u\in\NN_{6,\overline\lambda}^{nod}(\G)$ with mass $\mu_\G^{nod}$ and such that $E_6(\overline u,\G)<0$, which is also a nodal action ground state in $\NN_{6,\overline\lambda}^{nod}(\G)$, for some $\overline\lambda>-\lambda_2(\G)$. Moreover, every nodal action ground state in $\NN_{6,\overline\lambda}^{nod}(\G)$ has mass $\mu_\G^{nod}$, so that $Q_{6,\G}^{nod,-}(\overline\lambda)=Q_{6,\G}^{nod,+}(\overline\lambda)=\mu_\G^{nod}$. Furthermore, by Proposition \ref{prop:as}, we also have that $\displaystyle\lim_{\lambda\to+\infty}Q_{6,\G}^{nod,\pm}(\lambda)=\mu_\G^{nod}$. Hence, to prove the existence of the desired $\overline\lambda_1, \overline\lambda_2$ it is enough to show that there exists $\widetilde\lambda>\overline\lambda$ such that $Q_{6,\G}^{nod,-}(\widetilde\lambda)=Q_{6,\G}^{nod,+}(\widetilde\lambda)\neq\mu_\G^{nod}$. 
				
				We argue by contradiction and assume that $Q_{6,\G}^{nod,\pm}(\lambda)=\mu_\G^{nod}$ for every $\lambda\in[\overline\lambda,+\infty)\setminus \widetilde{Z}_{6,\G}^{nod}$ (recall that $Z_{6,\G}^{nod}$ is the set defined in Proposition \ref{prop:NAGS}). By Proposition \ref{prop:NAGS}, we then have
				\[
				\JJ_{6,\G}^{nod}(\lambda)=\JJ_{6,\G}^{nod}(\overline\lambda)+\int_{\overline\lambda}^{\lambda}(\JJ_{6,\G}^{nod})'(s)\,ds=\JJ_{6,\G}^{nod}(\overline\lambda)+\frac{\mu_\G^{nod}}2(\lambda-\overline\lambda)\qquad\forall \lambda\in[\overline\lambda,+\infty)\setminus \widetilde{Z}_{6,\G}^{nod}\,,
				\]
				and thus everywhere on $[\overline\lambda,+\infty)$ by continuity of $\JJ_{6,\G}^{nod}$. Note that
				\[
				\JJ_{6,\G}^{nod}(\overline\lambda)=J_{6,\overline\lambda}(\overline u,\G)=E_6(\overline u,\G)+\frac{\overline\lambda\mu_\G^{nod}}2\,,
				\]
				which plugged into the previous identity gives
				\[
				\JJ_{6,\G}^{nod}(\lambda)-\frac{\lambda\mu_\G^{nod}}2=E_6(\overline u,\G)\qquad\forall\lambda\geq\overline\lambda\,.
				\]
				But this is impossible, since Proposition \ref{prop:as} ensures that the left hand side of the previous identity goes to zero as $\lambda\to+\infty$, while the right hand side is strictly negative and independent of $\lambda$. This provides the contradiction we seek and proves \eqref{eq:Qp=6}. 
				
				\smallskip
				{\em Step 3: conclusion of the proof. } Let $\overline\lambda_1,\overline\lambda_2$ be such that \eqref{eq:Qp=6} holds, and let $\sigma:=Q_{6,\G}^{nod,-}(\overline\lambda_1)-Q_{6,\G}^{nod,+}(\overline\lambda_2)>0$. By Step 1, there exists $\varepsilon>0$ such that
				\begin{equation}
				\label{eq:Qp<6}
				\inf_{p\in(6-\varepsilon,6)}Q_{p,\G}^{nod,-}(\overline\lambda_1)\geq Q_{6,\G}^{nod,-}(\overline\lambda_1)-\frac{\sigma}{3}>Q_{6,\G}^{nod,+}(\overline\lambda_2)+\frac\sigma3\geq\sup_{p\in(6-\varepsilon,6)}Q_{p,\G}^{nod,+}(\overline\lambda_2)\,.
				\end{equation}
				We now prove that for every $p\in(6-\varepsilon,6)$ the set $Z_{p,\G}^{nod}$ defined in Proposition \ref{prop:Lpm} is not empty. This will complete the proof of Theorem \ref{thm:nonun}, because for every $\mu\in Z_{p,\G}^{nod}$ there always exist two least energy nodal normalized solutions $u_1,u_2\in\mathfrak{E}_{p,\G}^{nod}(\mu)$ with mass $\mu$ and such that $\LL_p(u_1,\G)=\Lambda_{p,\G}^{nod,-}(\mu)\neq\Lambda_{p,\G}^{nod,+}(\mu)=\LL_p(u_2,\G)$. To prove this latter fact note that, if $\mu\in Z_{p,\G}^{nod}$, then $\displaystyle\lim_{\nu\to\mu^-}\Lambda_{p,\G}^{nod,-}(\nu)<\lim_{\nu\to\mu^+}\Lambda_{p,\G}^{nod,+}(\nu)$. By Remark \ref{rem:Lnodatt}, as $\nu\to\mu^-$ any sequence of least energy nodal normalized solutions $u_\nu\in\mathfrak{E}_{p,\G}^{nod}(\nu)$ such that $\LL_p(u_\nu,\G)=\Lambda_{p,\G}^{nod,-}(\nu)$ converges (up to subsequences) weakly in $H^1(\G)$ and strongly in $L^q(\G)$, $q\geq2$, to a least energy nodal normalized solution $u_\mu\in\mathfrak{E}_{p,\G}^{nod}(\mu)$ solving \eqref{eq:nls} with $\displaystyle\lambda=\lim_{\nu\to\mu^-}\Lambda_{p,\G}^{nod,-}(\nu)$. Hence, $\displaystyle\Lambda_{p,\G}^{nod,-}(\mu)\leq \lim_{\nu\to\mu^-}\Lambda_{p,\G}^{nod,-}(\nu)$. Analogously, taking $\nu\to\mu^+$ and considering any sequence of least energy nodal normalized solutions $v_\nu\in\mathfrak{E}_{p,\G}^{nod}(\nu)$ with $\LL_p(v_\nu)=\Lambda_{p,\G}^{nod,+}(\nu)$, one has $\displaystyle \Lambda_{p,\G}^{nod,+}(\mu)\geq\lim_{\nu\to\mu^+}\Lambda_{p,\G}^{nod,+}(\nu)$. Therefore, $\Lambda_{p,\G}^{nod,-}(\mu)\neq\Lambda_{p,\G}^{nod,+}(\mu)$, and since they are both attained by Remark \ref{rem:Lnodatt}, we obtain the two different least energy nodal normalized solutions with mass $\mu$ we are seeking.
				
				To show that $Z_{p,\G}^{nod}\neq\emptyset$ for every $p\in(6-\varepsilon,6)$, assume by contradiction that there exists $\overline p\in(6-\varepsilon,6)$ such that $Z_{\overline p,\G}^{nod}=\emptyset$. Then, by Proposition \ref{prop:Lpm}, the map $\Lambda_{\overline p,\G}^{nod}:\R^+\to(-\lambda_2,+\infty)$ is well-defined, strictly increasing and surjective onto $(-\lambda_2(\G),+\infty)$. Therefore, for every $\lambda>-\lambda_2(\G)$ there exists a unique $\mu_\lambda>0$ and a least energy nodal normalized solution $u\in\mathfrak{E}_{\overline p,\G}^{nod}(\mu_\lambda)$ such that $\LL_{\overline p}(u,\G)=\lambda$. By Theorem \ref{thm:nodex}, every nodal action ground state in $\NN_{\overline p,\lambda}^{nod}(\G)$ has then mass $\mu_\lambda$. In particular, $Q_{\overline p,\G}^{nod,\pm}(\overline\lambda_1)=\mu_{\overline\lambda_1}$ and $Q_{\overline p, \G}^{nod,\pm}(\overline\lambda_2)=\mu_{\overline\lambda_2}$. On the one hand, since $\overline\lambda_2>\overline\lambda_1$, the monotonicity of $\Lambda_{\overline p,\G}^{nod}$ implies $\mu_{\overline\lambda_2}>\mu_{\overline\lambda_1}$.  On the other hand, \eqref{eq:Qp<6} yields $\mu_{\overline\lambda_1}>\mu_{\overline\lambda_2}$. As this gives a contradiction, we conclude.
			\end{proof}
			
			\section{On the hypotheses of Theorem \ref{thm:nonun}}
			\label{sec:hyp}
			 In this section we discuss the actual applicability of the conditional non-uniqueness result of Theorem \ref{thm:nonun}. Precisely, we first prove Theorem \ref{thm:D}, showing that every graph containing a symmetric dumbbell subgraph with sufficiently long bridging edges always displays non-uniqueness of least energy nodal normalized solutions at some mass, for $p$ close to $6$. Then, we show that both on graphs with at least two pendants and on ones admitting cycle coverings there exists no $L^2$-critical sign-changing solution with negative energy at the threshold mass $\mu_\G^{nod}$, thus preventing the application of Theorem \ref{thm:nonun} to these classes of graphs.
			 
			 We begin by proving the non-uniqueness result of Theorem \ref{thm:D}. To recall the notation set in the Introduction, let $\D_\ell$ be the dumbbell graph obtained by gluing together two identical copies $\G_1,\G_2$ of a given compact graph with no pendant, joined by an additional edge of total length $2\ell$ with one endpoint at a vertex of $\G_1$ and the other endpoint at the corresponding vertex of $\G_2$. A compact graph $\G$ is then said to contain a copy of $\D_\ell$ if $\G$ is as in Figure \ref{fig:Dl}, namely $\G$ is given by a copy of $\D_\ell$ and a further compact subgraph with no pendant, attached to $\D_\ell$ at the middle point of the edge of length $2\ell$.
			 
			 \begin{proof}[Proof of Theorem \ref{thm:D}]
			 	
			 	Since $\G$ contains a copy of $\D_\ell$, there exist two edges $e_1,e_2$, each of length $\ell$, such that $\G$ contains a copy of $\G_1$ attached to $e_1$ and a copy of $\G_2$ attached to $e_2$. For simplicity, from now on we write $\G_1,\G_2$ to denote their copies in $\G$. Moreover, $\widetilde\G_1:=\G_1\cup e_1$ and $\widetilde\G_2:=\G_2\cup e_2$ are two identical copies of the same graph, and $e_1$ and $e_2$ are incident at the same vertex $\overline\vv$ of $\G\setminus\left(\G_1\cup\G_2\right)$. 
			 	
			 	In view of Theorem \ref{thm:nonun}, to prove Theorem \ref{thm:D} we simply have to show that $\EE_{6,\G}^{nod}(\mu_\G^{nod})<0$ whenever $\ell$ is large enough. We do this by constructing explicitly a sign-changing solution to \eqref{eq:nls} on $\G$ with mass $\mu_\G^{nod}$ and strictly negative energy. Note that, since by assumption $\G$ has no pendant, then $\mu_\G^{nod}=2\mu_\R$ by \eqref{eq:muGnod}. 
			 	
			 	Consider the following minimization problem
			 	\[
			 	\widetilde\EE_{6,\widetilde\G_1}(\mu_\R):=\inf\left\{E_{6}(u,\widetilde\G_1)\,:\,u\in H_{\mu_\R}^1(\widetilde\G_1)\,,\,\, u(\overline\vv)=0\right\}\,,
			 	\] 
			 	that is we minimize the $L^2$-critical energy $E_6$ on $\widetilde\G_1$ among functions with mass $\mu_\R$ that vanish at the final vertex of the edge $e_1$. Clearly, any such function can be thought of as a compactly supported function on the non-compact graph $\overline\G_1$ obtained replacing the edge $e_1$ in $\widetilde\G_1$ with a half-line. Since $\overline\G_1$ is a graph with a unique half-line, by \cite[Theorem 3.3]{ASTcmp} one has $-\infty<\EE_{6,\overline\G_1}(\mu_\R)<0$ and it is attained. Hence, for every $\ell>0$
			 	\[
			 	\widetilde\EE_{6,\widetilde\G_1}(\mu_\R)\geq\EE_{6,\overline\G_1}(\mu_\R)>-\infty\,.
			 	\]
			 	Moreover, if $u\in H_{\mu_\R}^1(\overline\G_1)$ is such that $E_6(u,\overline\G_1)=\EE_{6,\overline\G_1}(\mu_\R)$, it is easily seen that the $H^1$-norm of the restriction of $u$ to the subset $[\ell,+\infty)$ of the half-line of $\overline\G_1$ tends to zero as $\ell\to+\infty$. This follows e.g. by noting that $u$ solves \eqref{eq:nlsnomass} on $\overline\G_1$ for some $\lambda>0$, so its restriction to the half-line of $\overline\G_1$ coincides with the restriction to $\R^+$ of the soliton $\displaystyle\phi_\lambda(x)=\sqrt{\sqrt{3\lambda}\text{\normalfont sech}(2\sqrt\lambda|x-x_0|)}$, for some $x_0\in\R$. Clearly, $\phi_\lambda$ decays exponentially fast as $|x|\to+\infty$, and so does its first derivative. Hence, for sufficiently large $\ell$ we can approximate $u$ with a function $w\in H_{\mu_{\R}}^1(\widetilde\G_1)$ 
such that $\widetilde w(\overline\vv)=0$ and $E_6(w,\widetilde\G_1)<E_{6}(u,\overline\G_1)/2<0$. This yields $\widetilde\EE_{6,\widetilde{\G}_1}(\mu_\R)<0$, which guarantees that $\widetilde\EE_{6,\widetilde{\G}_1}(\mu_\R)$ is attained by some function $\widetilde u$ with mass $\mu_\R$ and such that $\widetilde u(\overline\vv)=0$. The fact that the strict negativity of the  ground state energy level is sufficient to guarantee existence of ground states on compact graphs even at the threshold mass is standard, and a detailed proof of it that extends to our setting here with no modification can be found in \cite[Section 5]{D18}.
 			 	
 			 	Since $E_6(\widetilde{u},\widetilde\G_1)=\widetilde\EE_{6,\widetilde{\G}_1}(\mu_\R)$, there exists $\lambda\in\R$ such that
 			 	\begin{equation}
 			 		\label{eq:nlsutilde}
 			 	\begin{cases}
 			 		\widetilde u''+|\widetilde u|^4\widetilde u=\lambda\widetilde u & \text{on each edge }e\in\mathbb{E}_{\widetilde\G_1}\\
 			 		\widetilde{u}\text{ is continuous} & \text{on }\widetilde\G_1\\
 			 		\sum_{e\succ\vv}\widetilde u_e'(\vv)=0 & \text{at every vertex }\vv\in\mathbb{V}_{\widetilde\G_1}\setminus\left\{\overline\vv\right\}\\
 			 		\widetilde{u}(\overline\vv)=0\,. & 
 			 	\end{cases}
 			 	\end{equation}
			 	Moreover, with no loss of generality we can further assume $\widetilde u\geq0$ on $\widetilde\G_1$. Now, since $\widetilde\G_1$ and $\widetilde\G_2$ are identical copies of the same graph and they share the vertex $\overline\vv$, let $\overline u:\G\to\R$ be defined by
			 	\[
			 	\overline u(x):=\begin{cases}
			 		\widetilde u(x) & \text{if }x\in\widetilde\G_1\\
			 		-\widetilde u(x) & \text{if }x\in\widetilde\G_2\\
			 		0 & \text{elsewhere on }\G\,.
			 	\end{cases}
			 	\]
			 	Then $\overline u\in H_{2\mu_\R}^1(\G)$ and it changes sign on $\G$. Moreover, by \eqref{eq:nlsutilde} and the fact that
			 	\[
			 	\sum_{e\succ\overline\vv}\overline u_e'(\overline\vv)=\overline u_{e_1}'(\overline\vv)+\overline u_{e_2}'(\overline\vv)=-\widetilde{u}'(\ell)+\widetilde{u}'(\ell)=0
			 	\]
			 	(where we identified $\overline\vv$ with $x_{e_1}=\ell$ on $\widetilde\G_1$ and with $x_{e_2}=\ell$ on $\widetilde\G_2$), we obtain that $\overline u$ solves \eqref{eq:nls} on $\G$ with mass $2\mu_\R$. Since by construction $E_6(\overline u,\G)=2E_6(\widetilde u,\widetilde\G_1)<0$, we conclude. 
			 \end{proof}
			 \begin{remark}
			 	The proof of Theorem \ref{thm:D} makes essential use of the fact that $\G_1$ and $\G_2$ are identical copies of the same graph. Indeed, if this assumption is removed, the minimization problems $\widetilde\EE_{6,\widetilde\G_1}(\mu_\R)$, $\widetilde{\EE}_{6,\widetilde\G_2}(\mu_\R)$ with the additional homogeneous Dirichlet condition at the vertex $\overline\vv$ will be genuinely different. In particular, it will be no longer guaranteed that a sign-changing function obtained gluing together two solutions of those problems solves \eqref{eq:nls} on $\G$. 
			 \end{remark}
			 
			 If Theorem \ref{thm:D} exhibits a class of graphs where the assumptions of Theorem \ref{thm:nonun} are satisfied, the next proposition goes in the opposite direction.
			 
			 \begin{proposition}
			 	\label{prop:E>0}
			 	If $\G$ has at least two pendants or admits a cycle covering, then $\EE_{6,\G}^{nod}(\mu_\G^{nod})\geq0$. 
			 \end{proposition}
			 \begin{proof}
			 	Assume first that $\G$ has at least two pendants, so that $\mu_\G^{nod}=2\mu_{\R^+}$ by \eqref{eq:muGnod}. 
			 	Assume by contradiction that there exists a sign-changing $u\in H_{2\mu_{\R^+}}^1(\G)$ solving \eqref{eq:nls} for some $\lambda\in\R$ and such that $E_6(u,\G)<0$. Then $u\in\NN_{6,\lambda}^{nod}(\G)$, and therefore
			 	\begin{equation}
			 	\label{eq:l>0}
			 	\lambda=\LL_6(u,\G)= \frac{\frac13\|u\|_6^6-E_6(u,\G)}{\mu_{\R^+}}>0\,.
			 	\end{equation}
			 	Furthermore, since $u^\pm$ vanish somewhere on $\G$, every value in $[0,\|u^\pm\|_\infty]$ is attained at least once. Hence, denoting by $(u^\pm)^*$ the monotone rearrangements on $\R^+$ of $u^\pm$, the standard theory of rearrangements (see e.g. \cite[Section 3]{AST}) ensures that $(u^\pm)^*\in H^1(\R^+)$ and
			 	\begin{equation}
			 		\label{eq:rearr}
			 	\|((u^\pm)^*)'\|_{L^2(\R^+)}\leq\|(u^\pm)'\|_2\,,\qquad\|(u^\pm)^*\|_{L^q(\R^+)}=\|u^\pm\|_{q}\quad\forall q\geq2\,.
			 	\end{equation}
			 	Since $u^\pm\in\NN_{6,\lambda}(\G)$, this then entails
			 	\[
			 	n_{6,\lambda}^4((u^\pm)^*)=\frac{\|((u^\pm)^*)'\|_{L^2(\R^+)}^2+\lambda\|(u^\pm)^*\|_{L^2(\R^+)}^2}{\|(u^\pm)^*\|_{L^6(\R^+)}^6}\leq \frac{\|(u^\pm)'\|_{2}^2+\lambda\|u^\pm\|_{2}^2}{\|u^\pm\|_{6}^6}=1
			 	\]
			 	(note that $n_{6,\lambda}((u^\pm)^*)$ is well-defined by \eqref{eq:l>0}). Therefore, 
			 	\[
			 	\begin{split}
			 	\frac{\lambda\mu_{\R^+}}{2}=\JJ_{6,\R^+}(\lambda)&\,\leq J_{6,\lambda}\left(n_{6,\lambda}((u^\pm)^*)(u^\pm)^*,\R^+\right)=\frac13n_{6,\lambda}((u^\pm)^*)^6\|(u^\pm)^*\|_{L^6(\R^+)}^6\\
			 	&\,\leq \frac13\|u^\pm\|_6^6= J_{6,\lambda}(u^\pm,\G)=E_6(u^\pm,\G)+\frac\lambda2\|u^\pm\|_2^2\,,
			 	\end{split}
			 	\]
			 	and summing together the two inequalities gives
			 	\[
			 	\lambda\mu_{\R^+}\leq E_6(u^+,\G)+E_6(u^-,\G)+\frac\lambda2\left(\|u^+\|_2^2+\|u^-\|_2^2\right)=E_6(u,\G)+\lambda\mu_{\R^+}\,,
			 	\]
			 	contradicting the assumption $E_6(u,\G)<0$. 
			 	
			 	If  $\G$ admits a cycle covering the proof is analogous, simply replacing $\mu_{\R^+}$ with $\mu_\R$, monotone rearrangements on $\R^+$ with symmetric rearrangements on $\R$, and recalling that, since now $u^\pm$ are compactly supported functions on a graph admitting a cycle covering, then almost every value in $[0,\|u^\pm\|_\infty]$ is attained at least twice. As pointed out in \cite{ASTcmp}, this last fact ensures that the properties in \eqref{eq:rearr} hold true also for the symmetric rearrangements of $u^\pm$. 
			 \end{proof}
			
			\section{Proof of Theorem \ref{thm:uniqnod}} 
			\label{sec:uniqnod}
			
			The last two sections of the paper are devoted to a thorough analysis of least energy normalized solutions on the segment $[0,1]$. In this section, we focus on the nodal setting, proving uniqueness of least energy nodal normalized solutions on $[0,1]$ for every mass $\mu>0$ and every $p\in(2,6)$, namely Theorem \ref{thm:uniqnod}. 
			
			Existence of least energy nodal normalized solutions is guaranteed by Theorem \ref{thm:nodex}, which also shows that every such solution is a nodal action ground state in $\NN_{p,\lambda}^{nod}(0,1)$, for some $\lambda>-\lambda_2(0,1)=-\pi^2$. Moreover, since nodal action ground states are solutions of \eqref{eq:main} with exactly two nodal regions by Remark \ref{rem:2nod}, it is straightforward to see that, up to a change of sign, they are non-increasing on $[0,1]$. Hence, to prove Theorem \ref{thm:uniqnod} it is enough to show that, for every $p\in(2,6)$ and every $\mu>0$, there exists a unique $\lambda>-\pi^2$ and a unique non-increasing, sign-changing solution to \eqref{eq:main}. We do this by showing that there exists a unique such solution for every $\lambda>-\pi^2$ and that the mass of such solution is a strictly increasing function of $\lambda$. 
			
			\begin{proposition}\label{thm:main}
				For every $p\in(2,6)$, the following holds.
				\begin{enumerate}[label=\textup{(\roman*)}]
					\item Problem \eqref{eq:main} admits exactly one sign-changing, non-increasing solution $u_\lambda$ for every $\lambda>-\pi^2$, and no such solution for every $\lambda\leq-\pi^2$.
					\item The quantity
					\[
					M_p(\lambda):=\int_0^1 u_\lambda(x)^2\,dx
					\]
					is a strictly increasing function of $\lambda$ on $(-\pi^2,+\infty)$.
				\end{enumerate}
			\end{proposition}
 			We begin with some preliminary considerations. Let $u$ be a sign-changing, monotone non-increasing solution of \eqref{eq:main}, so that
 			\[
 			u(0)=a>0,\qquad u(1)=-b<0
 			\]
 			for some $a,b>0$. Multiplying the equation by $u'$ shows that there exists $C\in\R$ such that 
 			\begin{equation}\label{eq:first_integral}
 				\frac12(u')^2-\frac\lambda2u^2+\frac1p|u|^p=C\qquad\text{on }[0,1],
 			\end{equation}
 			that, setting 
 			\begin{equation}\label{eq:W}
 				W_\lambda(r):=-\frac\lambda2r^2+\frac1p |r|^p\qquad\forall r\in\R
 			\end{equation}
 			can be rewritten as
 			\[
 			\frac12(u')^2+W_\lambda(u)=C\qquad\text{on }[0,1].
 			\]
 			This immediately shows that $u$ is strictly decreasing in $(0,1)$ and $b=a$. Indeed, if $u'$ vanished at an interior point $\bar x$, then $u'$ would have a local maximum at $\bar x$ and $u''(\bar x)=0$. By \eqref{eq:main}, we would then have $|u(\bar x)|^{p-2}=\lambda$. If $\lambda\le0$, this is clearly impossible. If $\lambda>0$, this is again impossible, because $\pm\lambda^{1/(p-2)}$ are the global minimum points of $W_\lambda$, so that $W_\lambda(u(\bar x))=W_\lambda(\pm\lambda^{1/(p-2)})$ forces $u'\equiv0$ on $[0,1]$ by \eqref{eq:first_integral}, contradicting the fact that $u$ is sign-changing. Hence, $u'<0$ on $(0,1)$. We now show that $a=b$.  Since $u'(0)=u'(1)=0$, identity \eqref{eq:first_integral} gives
 			\begin{equation}\label{eq:equal_endpoint_energy}
 				W_\lambda(a)=W_\lambda(b).
 			\end{equation}
 			If $\lambda\le0$, then $W_\lambda$ is strictly increasing on $[0,+\infty)$, so \eqref{eq:equal_endpoint_energy} immediately implies $a=b$. If $\lambda>0$, since the strict monotonicity of $u$ gives $u''(0)<0, u''(1)>0$, by \eqref{eq:main} one has
 			\[
 			u''(0)=\lambda a-a^{p-1}<0,
 			\qquad
 			u''(1)=-\lambda b+b^{p-1}>0,
 			\]
 			that is
 			\[
 			a^{p-2}>\lambda,
 			\qquad
 			b^{p-2}>\lambda.
 			\]
 			Since $W_\lambda$ is strictly increasing on $(\lambda^{1/(p-2)},+\infty)$, by \eqref{eq:equal_endpoint_energy} we obtain again $a=b$. Observe that when $\lambda>0$ the strict monotonicity of $u$ and the fact that it changes sign yield also $a^{p-2}>\frac p2\lambda$.
 		
 			Since $v(x):=-u(1-x)$ solves the same Cauchy problem for \eqref{eq:main} at $x=0$ as $u$, by uniqueness $u(x)=-u(1-x)$, so that in particular $u(1/2)=0$.
 			Since $u$ is positive and decreasing on $[0,1/2)$, by \eqref{eq:first_integral} and $u(0)=a, u'(0)=0$ we have
 			\begin{equation}\label{eq:up_negative}
 				u'=-\sqrt{\frac{2}{p}(a^p-u^p)-\lambda(a^2-u^2)}\qquad\text{on }[0,1/2]
 			\end{equation}
 			(note that the square root is well-defined on $[0,1/2]$ because $\frac2pa^p-\lambda a^2>0$ for every $\lambda$). Hence, integrating on $[0,1/2]$ and taking the change of variable $t=u(x)$ we obtain
 			\begin{equation}\label{eq:length_a_lambda}
 				\frac12=\int_0^a
 				\frac{\dd t}{\sqrt{\frac{2}{p}(a^p-t^p)-\lambda(a^2-t^2)}}.
 			\end{equation}
 			Performing the change of variables $t=as$ and setting 
 			\begin{equation}\label{eq:mu_def}
 				\xi:=\lambda a^{2-p},\quad
 				D_{p,\xi}(s):=\frac{2}{p}(1-s^p)-\xi(1-s^2) \qquad\forall s\in[0,1]\,,
 			\end{equation}
 			and
 			\begin{equation}\label{eq:I0_def}
 				I_0^{(p)}(\xi):=\int_0^1\frac{\dd s}{\sqrt{D_{p,\xi}(s)}},
 			\end{equation}
 			then \eqref{eq:length_a_lambda} becomes
 			\begin{equation}\label{eq:a_p_minus_two}
 				a^{p-2}=4\left(I_0^{(p)}(\xi)\right)^2
 			\end{equation}
 			and 
 			\begin{equation}\label{eq:lambda_mu}
 				\lambda=\lambda_p(\xi):=4\xi\left(I_0^{(p)}(\xi)\right)^2,
 				\qquad \xi\in\left(-\infty,\frac2p\right)
 			\end{equation}
 			(the range for the parameter $\xi$ is given by the non-negativity of $D_{p,\xi}$ on $[0,1]$).
 			
 			\begin{lemma}\label{lem:lambda_monotone}
 				For every $p\in(2,6)$, the map $\xi\mapsto\lambda_p(\xi)$ is strictly increasing on $(-\infty,2/p)$, and
 				\[
 				\lim_{\xi\to-\infty}\lambda_p(\xi)=-\pi^2,\qquad\lim_{\xi\to\frac2p^-}\lambda_p(\xi)=+\infty.
 				\]
 			\end{lemma}
 			
 			\begin{proof}
 				Differentiating \eqref{eq:lambda_mu}  gives 
 				\begin{equation}\label{eq:lambda_prime_start}
 					\lambda_p'(\xi)=4I_0^{(p)}(\xi)
 					\left(I_0^{(p)}(\xi)+2\xi \left(I_0^{(p)}\right)'(\xi)\right).
 				\end{equation}
 				By \eqref{eq:I0_def}, we have $I_0^{(p)}(\xi)>0$ and
 				\begin{equation}\label{eq:I0_prime}
 					\left(I_0^{(p)}\right)'(\xi)=\frac12\int_0^1
 					\frac{1-s^2}{D_{p,\xi}(s)^{3/2}}\dd s>0\,,
 				\end{equation}
 				so that
 				\begin{align}
 					I_0^{(p)}(\xi)+2\xi \left(I_0^{(p)}\right)'(\xi)
 					=\int_0^1
 					\frac{D_{p,\xi}(s)+\xi(1-s^2)}{D_{p,\xi}(s)^{3/2}}\dd s=\int_0^1
 					\frac{\frac2p(1-s^p)}{D_{p,\xi}(s)^{3/2}}\dd s>0\,,
 				\end{align}
 				in turn implying $\lambda_p'(\xi)>0$ on $(-\infty,2/p)$. Moreover, since there exists $C>0$ independent of $s$ and $\xi$ such that, as $\xi\to-\infty$,
 				\begin{equation}
 					\label{eq:largemu}
 					-\xi(1-s^2)\leq D_{p,\xi}(s)\leq -\xi(1-s^2)(1-C/\xi)\qquad\forall s\in[0,1],
 				\end{equation}
 				we have
 				\begin{equation*}
 					\frac1{\sqrt{1-C/\xi}\sqrt{|\xi|}}\int_0^1\frac{\dd s}{\sqrt{1-s^2}}\leq I_0^{(p)}(\xi)\leq  \frac1{\sqrt{|\xi|}}\int_0^1\frac{\dd s}{\sqrt{1-s^2}}\,,
 				\end{equation*}
 				which shows that 
 				\begin{equation}\label{eq:lambda_left_limit}
 					\lim_{\xi\to-\infty}\lambda_p(\xi)=\lim_{\xi\to-\infty}4\xi\left(I_0^{(p)}(\xi)\right)^2= -\pi^2.
 				\end{equation}
 				Take then $\xi\to\frac2p^-$. Since $D_{p,2/p}(s)=\frac2ps^2(1-s^{p-2})$, for every $M>0$ there exists $\delta>0$ such that 
 				\[
 				\int_\delta^{1/2}\frac{ds}{\sqrt{D_{p,2/p}(s)}}\geq M\,,
 				\]
 				and by the Dominated Convergence Theorem
 				\[
 				\lim_{\xi\to\frac2p^-} I_0^{(p)}(\xi)\geq\lim_{\xi\to\frac2p^-}\int_\delta^{1/2}\frac{ds}{\sqrt{D_{p,\xi}(s)}}=\int_\delta^{1/2}\frac{ds}{\sqrt{D_{p,2/p}(s)}}\geq M\,.
 				\]
 				The arbitrariness of $M$ then gives $\displaystyle\lim_{\xi\to\frac2p^-}I_0^{(p)}(\xi)=+\infty$, and the same holds for $\lambda_p(\xi)$.
 			\end{proof}
 			
 			\begin{proof}[Proof of Proposition \ref{thm:main}]
 				Since $\lambda_p$ is continuous, by Lemma \ref{lem:lambda_monotone} it maps $(-\infty,2/p)$ bijectively onto $(-\pi^2,+\infty)$. Hence, if \eqref{eq:main} admits a sign-changing, non-increasing solution, then $\lambda\in(-\pi^2,+\infty)$. Conversely, given any $\xi<2/p$, define $a>0$ by  \eqref{eq:a_p_minus_two} and let $\lambda=\lambda_p(\xi)$ be as in \eqref{eq:lambda_mu}. Solving the Cauchy problem with $u(0)=a$, $u'(0)=0$, by \eqref{eq:up_negative} and \eqref{eq:length_a_lambda} it follows that the solution vanishes at $x=1/2$. Extending it by $u(x)=-u(1-x)$ gives a $C^2$ solution on $[0,1]$ satisfying \eqref{eq:main} and $u(1)=-a$, showing that \eqref{eq:main} admits at least one sign-changing, non-increasing solution for every $\lambda>-\pi^2$. This shows that such a solution exists for \eqref{eq:main} if and only if $\lambda\in(-\pi^2,+\infty)$, and for every such $\lambda$ this solution is unique by Lemma \ref{lem:lambda_monotone}. This proves Proposition \ref{thm:main}(i).
 				
 				We now focus on Proposition \ref{thm:main}(ii). Since $u(x)=-u(1-x)$ for every $x\in[0,1]$, we have
 				\begin{equation}\label{eq:mass_half}
 					M_p(\lambda):=\int_0^1u^2\dd x=2\int_0^{1/2}u^2\dd x.
 				\end{equation}
 				On $[0,1/2]$, changing the variables $u(x)=t=as$ and combining with \eqref{eq:a_p_minus_two} we obtain
 				\begin{equation}\label{eq:mass_a_lambda}
 					M_p(\lambda_p(\xi))=2\int_0^a
 					\frac{t^2\dd t}{\sqrt{\frac2p(a^p-t^p)-\lambda_p(\xi)(a^2-t^2)}}=2\left(2I_0^{(p)}(\xi)\right)^{\frac{6-p}{p-2}}I_2^{(p)}(\xi)\,,
 				\end{equation}
 				where
 				\begin{equation}\label{eq:I2_def}
 					I_2^{(p)}(\xi):=\int_0^1
 					\frac{s^2\dd s}{\sqrt{D_{p,\xi}(s)}}\,.
 				\end{equation}
 				Since
 				\begin{equation}\label{eq:I2_prime}
 					\left(I_2^{(p)}\right)'(\xi)=\frac12\int_0^1
 					\frac{s^2(1-s^2)}{D_{p,\xi}(s)^{3/2}}\dd s>0\,,
 				\end{equation}
 				combining with \eqref{eq:I0_prime} shows that $M_p(\lambda_p(\xi))$ is the product of increasing functions for every $p\in(2,6)$, namely $M_p$ is a strictly increasing function of $\xi$. Since Lemma \ref{lem:lambda_monotone} gives the same for $\lambda_p(\xi)$, we conclude that $M_p$ is a strictly increasing function of $\lambda$ on $(-\pi^2,+\infty)$.
 			\end{proof}

 			\section{Proof of Theorem \ref{thm:uniq}}
 			\label{sec:uniqp=2}
			
			In this last section we consider energy ground states in $H_\mu^1(0,1)$ for $p$ close to $2$ and prove Theorem \ref{thm:uniq}, namely that (up to symmetries) they are unique for every mass $\mu>0$. 
			
			The proof of this result is quite long and technical, as it is based on a thorough analysis of strictly increasing solutions of the Cauchy problem 
			\begin{equation}\label{eq:main2}
				u''+u^{p-1}=\lambda u,
				\qquad u>0,\qquad u'(0)=u'(1)=0
			\end{equation}
			for $p$ sufficiently close to $2$. In particular, we will derive the following results.
			\begin{proposition}\label{prop:main}
				For every $p>2$, define
				\begin{equation}
				\label{eq:lambda*}
				\lambda_*(p):=\frac{\pi^2}{p-2}\,.
				\end{equation}
				A strictly increasing solution $u_\lambda$ to \eqref{eq:main2} exists if and only if $\lambda>\lambda_*(p)$. Up to the reflection $x\mapsto 1-x$, $u_\lambda$ is the unique strictly monotone solution of \eqref{eq:main2}.
			\end{proposition}
			\begin{proposition}
				\label{prop:mass}
				For every $p>2$ and every $\lambda>\lambda_*(p)$, let $u_\lambda$ be the solution of \eqref{eq:main2} found in Proposition \ref{prop:main}. There exists $\varepsilon>0$ such that, for every $p\in(2,2+\varepsilon)$, the map
				\[
				M_p(\lambda):=\int_0^1 u_\lambda(x)^2\,dx
				\]
				is strictly increasing on $(\lambda_*(p),+\infty)$ and $M_p((\lambda_*(p),+\infty))=(\mu_*(p),+\infty)$, where 
				\begin{equation}
					\label{eq:mu2}
					\mu_*(p):=\left(\frac{\pi^2}{p-2}\right)^{\frac2{p-2}}.
				\end{equation} 
			\end{proposition}
			
			Let us then prove Theorem \ref{thm:uniq} exploiting Propositions \ref{prop:main}--\ref{prop:mass}.
			
			\begin{proof}[Proof of Theorem \ref{thm:uniq}]
				Recall first that, up to a phase shift and a change of sign, energy ground states of $E_p$ on $H_\mu^1(0,1)$ are real and non-negative on $[0,1]$. Moreover, the same standard rearrangement arguments used in the proof of Proposition \ref{prop:E>0} above ensure that (up to the reflection $x\mapsto1-x$) every non-negative ground state is non-decreasing on $[0,1]$. Since ground states are solutions of \eqref{eq:nls}, a direct inspection of the phase plane then shows that either they are constant on the whole interval $[0,1]$, or they are strictly increasing on $[0,1]$. 
				
				Let $\varepsilon>0$ be as in Proposition \ref{prop:mass}, $p\in(2,2+\varepsilon)$, and $\mu_*(p)$ as in \eqref{eq:mu2}. For every $\mu\in(0,\mu_*(p)]$, Propositions \ref{prop:main}--\ref{prop:mass} show that there exists no strictly increasing solution of \eqref{eq:main2} with mass $\mu$. Hence, for each of these masses the unique positive energy ground state of $E_p$ in $H_\mu^1(0,1)$ must be the constant function $\kappa_\mu:=\sqrt\mu$. Let then $\mu>\mu_*(p)$. By Propositions \ref{prop:main}--\ref{prop:mass}, there exists a unique strictly increasing solution of \eqref{eq:main2} with mass $\mu$. Moreover, in \cite[Theorem 2.1 and Eq. (14)]{CDS} it has been proved that the constant function $\kappa_\mu$ is not even a local minimizer of $E_p$ in $H_\mu^1(0,1)$ whenever
				\[
				\mu>\left(\frac{\lambda_2(0,1)}{p-2}\right)^{\frac2{p-2}}=\mu_*(p)\,.
				\]
				Hence, $\kappa_\mu$ is not a ground state of $E_p$ in $H_\mu^1(0,1)$ if $\mu>\mu_*(p)$, in turn implying that the positive ground state is unique and corresponds to the unique strictly increasing solution of \eqref{eq:main2} with the desired mass.
			\end{proof}

			The rest of the section is devoted to the proof of Propositions \ref{prop:main}--\ref{prop:mass}. 
			
			\begin{proof}[Proof of Proposition \ref{prop:main}] Note first that \eqref{eq:main2} admits no solution if $\lambda\leq0$, because every $u$ solving \eqref{eq:main2} satisfies $u^{p-2}\leq\lambda$ at each of its minimum points. Since $\lambda>0$, any solution $u$ of \eqref{eq:main2} can be written as
				\begin{equation}\label{eq:scaling}
					u(x)=\lambda^{1/(p-2)}u_1(\sqrt\lambda x),
				\end{equation}
				where $u_1$ is a solution of 
				\begin{equation}\label{eq:normalized}
					-u_1''+u_1=u_1^{p-1},\qquad u_1>0,\qquad u_1'(0)=u_1'(L)=0
				\end{equation}
				with $\sqrt\lambda=L$.
				If $u_1$ is a strictly increasing solution of \eqref{eq:normalized} such that $u_1(0)=\alpha$, for some $0<\alpha<1$, then $u_1(L)=\beta_p(\alpha)$ for the unique $1<\beta_p(\alpha)<(p/2)^{1/(p-2)}$ such that
				\begin{equation}\label{eq:same-energy}
					F_p(\alpha)=F_p(\beta_p(\alpha)),
				\end{equation}
				where
				\begin{equation}\label{eq:Fq}
					F_p(t):=\frac{t^2}2-\frac{t^p}{p}\,,\qquad t\geq0\,.
				\end{equation}
				Note that $F_p$ coincides with $-W_1$ as defined in \eqref{eq:W}. However, since in this section we consider the limit $p\to2^+$, we introduce this notation to highlight the dependence on $p$. 
				
				Arguing exactly as in the previous section immediately shows that the length $L$ is determined by $\alpha$ as
				\begin{equation}
					\label{eq:La}
					L_p(\alpha):=\int_\alpha^{\beta_p(\alpha)}
					\frac{dt}{\sqrt{2(F_p(t)-F_p(\alpha))}}\,.
				\end{equation}
				For every $p>2$, it is well-known that $L_p\in C^1(0,1)$, $L_p$ is strictly decreasing on $(0,1)$, and
				\begin{equation}\label{eq:L-bif-limit}
					\lim_{\alpha\to1^-}L_p(\alpha)=\frac{\pi}{\sqrt{p-2}},\qquad  \lim_{\alpha\to0^+}L_p(\alpha)=+\infty\,.
				\end{equation}
				Proofs of the regularity and the monotonicity of $L_p$ and of the first identity in \eqref{eq:L-bif-limit} can be found e.g. in \cite{AN, BDL}, whereas the second identity in \eqref{eq:L-bif-limit} follows by $F_p(\alpha)\to0$ as $\alpha\to0^+$ and $F_p(t)=t^2/2+o(t^2)$ when $t$ is close to $0$. 
				
				Since the map $\alpha\mapsto L_p(\alpha)$ is strictly decreasing on $(0,1)$, by the relation between \eqref{eq:main2} and \eqref{eq:normalized} it follows that, for every $\alpha\in (0,1)$,  the unique strictly increasing solution of \eqref{eq:normalized} with $u_1(0)=\alpha$ identifies a unique strictly increasing solution of \eqref{eq:main2} with $\lambda=\lambda_p(\alpha)=L_p(\alpha)^2$. Therefore, for every $\lambda\in\R$ there exists at most one strictly monotone solution to \eqref{eq:main2} (up to the symmetry $x\mapsto 1-x$). By \eqref{eq:L-bif-limit} such solution exists if and only if $\displaystyle\lambda>\lim_{\alpha\to1^-}\lambda_p(\alpha)=\pi^2/(p-2)=\lambda_*(p)$, and we conclude.
			\end{proof}
			
			The proof of Proposition \ref{prop:mass} is more involved. The aim is to prove that the first derivative of $M_p(\lambda)$ is strictly positive on $(\lambda_*(p),+\infty)$ uniformly in $p$ on a right neighborhood of $2$. To do this, it is convenient to recall again that the strictly increasing solution $u_\lambda$ to \eqref{eq:main2} satisfies $u_\lambda(x)=\lambda^{1/(p-2)}u_1(\sqrt\lambda x)$, where $u_1$ is the unique strictly increasing solution to  \eqref{eq:normalized} determined by the initial value $u_1(0)=\alpha$. Exploiting this identity, it is possible to rewrite both $\lambda$ and $M_p(\lambda)$ as functions of the parameter $\alpha\in(0,1)$. Indeed, by the relation between \eqref{eq:main2} and \eqref{eq:normalized}, if $u_1(0)=\alpha$, we have $\lambda(\alpha)=L_p(\alpha)^2$ with $L_p$ as in \eqref{eq:La}, and
			\begin{equation}
				\label{eq:M-param}
				\begin{split}
					M_p(\alpha)&:=M_p(\lambda(\alpha))=\int_0^1 u_{\lambda(\alpha)}(x)^2\,dx=\lambda(\alpha)^{\frac2{p-2}}\int_0^1 u_1\left(\sqrt{\lambda(\alpha)}x\right)^2\,dx\\
					&=\lambda(\alpha)^{\frac2{p-2}-\frac12}\int_0^{L_p(\alpha)}u_{1}(y)^2\,dy=L_p(\alpha)^{\frac4{p-2}-1}N_p(\alpha)\,,
				\end{split}
			\end{equation}
			where
			\begin{equation*}
				N_p(\alpha):=\int_\alpha^{\beta_p(\alpha)}
				\frac{t^2\,dt}{\sqrt{2(F_p(t)-F_p(\alpha))}}\,.\label{eq:Nq}
			\end{equation*}
			As already pointed out for $L_p$, it is easily seen that $N_p\in C^1(0,1)$ and 
			\begin{equation}\label{eq:N-bif-limit}
				\lim_{\alpha\to1^-}N_p(\alpha)=\frac{\pi}{\sqrt{p-2}},\qquad \lim_{\alpha\to0^+}N_p(\alpha)
				=\int_0^{(\frac p2)^{\frac1{p-2}}}
				\frac{t^2\,dt}{\sqrt{2F_p(t)}}<\infty.
			\end{equation}
			Since $\alpha\mapsto\lambda_p(\alpha)$ is strictly decreasing, $M_p(\lambda)$ being strictly increasing on $(\lambda_*(p),+\infty)$ is equivalent to $M_p(\alpha)$ being strictly decreasing on $\alpha\in(0,1)$. Hence, we are left to show that $M_p'(\alpha)<0$ for every $\alpha\in(0,1)$ and $p\in(2,2+\varepsilon)$. In fact, relying on the monotonicity of the logarithm, we will often equivalently focus on the sign of 
			\begin{equation}
				\label{eq:derlog}
				\left(\log M_{2+q}(\alpha)\right)'=\frac{4-q}{q}\frac{L_{2+q}'(\alpha)}{L_{2+q}(\alpha)}+\frac{N_{2+q}'(\alpha)}{N_{2+q}(\alpha)}\,.
			\end{equation}
			We split the analysis in two parts. First, we show that there exist $\varepsilon_1>0$ and $\underline\alpha\in(0,1)$ such that $M_p'(\alpha)<0$ for every $(p,\alpha)\in(2,2+\varepsilon_1)\times(\underline\alpha,1)$. Then, we prove that, for every given $\overline\alpha\in(0,1)$, there exists $\varepsilon_2>0$ (depending only on $\overline\alpha$) such that $M_p'(\alpha)<0$ for every $(p,\alpha)\in(2,2+\varepsilon_2)\times(0,\overline\alpha]$. Once these two steps are completed, Proposition \ref{prop:mass} will follow directly by taking $\overline\alpha>\underline\alpha$ and $\varepsilon=\min\left\{\varepsilon_1,\varepsilon_2\right\}$. 
			
			\subsection{Monotonicity of $M_p(\alpha)$ close to $\alpha=1$}
			
			The regime $\alpha\in(\underline\alpha, 1)$ is addressed through a bifurcation analysis. Precisely, since $\alpha$ close to $1$ corresponds to $\lambda$ close to $\lambda_*(p)$, we develop a Lyapunov-Schmidt reduction to prove that, locally around $\lambda_*(p)$, the monotone solutions to \eqref{eq:main2} bifurcate along a unique smooth branch from the constant solution. We derive an explicit expansion for such solutions and, relying on the relation between $\alpha$ and $\lambda$, we use it to show that the derivative of $M_p$ with respect to $\alpha$ is strictly negative in a uniform left neighborhood of $1$ when $p$ is close to $2$. 
			
			Since we need estimates uniform in $p$ around $2$, from now on we set $q:=p-2$. We write again any monotone solution $u_\lambda$ to \eqref{eq:main2} as $u_\lambda(x)=\lambda^{1/q}u_1(\sqrt\lambda\,x)$, where $u_1$ solves \eqref{eq:normalized} with $L=\sqrt\lambda$. Setting $v(y)=u_1(\sqrt\lambda y/\pi)$ and dividing the equation by $q$ we obtain that $v(y)$ satisfies
			\begin{equation}
				\label{eq:v2}
			\rho v''+\frac{v^{1+q}-v}q=0\quad\text{on }(0,\pi)\,,\qquad v>0\,,\qquad v'(0)=v'(\pi)=0
			\end{equation}
			with 
			\begin{equation}
				\label{eq:rho}
			\rho:=\frac{\pi^2}{q\lambda}\,.
			\end{equation}
			Since we want to develop a bifurcation analysis around the constant solution $v\equiv1$, note that, if we write $v=1+w$, then $w$ solves the equation
			\begin{equation}
				\label{eq:w}
			\rho w''+g(q,w)=0\quad\text{on }(0,\pi), \qquad w'(0)=w'(\pi)=0\,,
			\end{equation}
			where, for every $q\neq0$, 
			\begin{equation}
			\label{eq:g}
			g(q,w):=\frac{(1+w)^{1+q}-(1+w)}q\,.
			\end{equation}
			\begin{remark}
				\label{rem:g}
				Given $\eta\in(0,1)$, if we set
				\[
				g(0,s):=(1+s)\log(1+s)\qquad\forall s\in[-\eta,\eta]\,,
				\]
				combining with \eqref{eq:g} gives a $C^\infty$ function $g$ on a neighborhood of $\left\{0\right\}\times[-\eta,\eta]$, with 
				\begin{equation}
					\label{eq:g-derivatives-origin}
					g(q,0)=0,
					\qquad
					\partial_sg(q,0)=1\,,
				\end{equation}
				and
				\begin{equation}\label{eq:g-taylor}
					g(q,s)
					=s+\frac{1+q}{2}s^2+\frac{(1+q)(q-1)}6s^3+s^4R(q,s),
				\end{equation}
				where $R$ is $C^\infty$ on compact subsets of its domain. Indeed, that such $g$ defines a smooth function on a neighborhood of $\left\{0\right\}\times[-\eta,\eta]$ is a direct consequence of the identity 
			\[	
			 g(q,s)=(1+s)\sum_{k=1}^{\infty}\frac{q^{k-1}(\log(1+s))^k}{k!}\qquad\forall s>-1\,,
			 \] 
			and direct differentiations and a Taylor expansion with integral remainder then give \eqref{eq:g-derivatives-origin} and \eqref{eq:g-taylor}. As a consequence, setting
			\[
			H_N^2(0,\pi):=\left\{w\in H^2(0,\pi)\,:\,w'(0)=w'(\pi)=0\right\}\,,
			\] 
			the map $(q,w)\mapsto g(q,w(\cdot))$ is well-defined and $C^\infty$ from $(-\underline q,\underline q)\times U$ to $L^2(0,\pi)$, for some $\underline q>0$ and a neighborhood $U$ of $0$ in $H^2_N(0,\pi)$, with derivatives with respect to $w$ given by $\displaystyle D_w^k g(q,w)[h_1,\dots,h_k]=\partial_s^k g(q,w)h_1\dots h_k$, for $h_1,\dots, h_k\in H_N^2(0,\pi)$. 
			\end{remark}
			Recall that here we are considering values of $\lambda$ in \eqref{eq:main2} close to the threshold $\lambda_*(p)$, that, by \eqref{eq:rho}, corresponds to values of $\rho$ close to $1$.  Note that, when $\rho=1$, by \eqref{eq:g-derivatives-origin} the linearization of \eqref{eq:w} at $w\equiv0$ is $h''+h=0$, whose solutions in $H_N^2(0,\pi)$ are given by the span of $\psi(y):=\cos y$. Hence, to construct a branch of solutions to \eqref{eq:w} bifurcating from $w\equiv0$ when $\rho$ is close to $1$ and $q>0$ is small enough, we perform a Lyapunov-Schmidt reduction and look for solutions in the form
			\begin{equation}
				\label{eq:formw}
			w(y)=r\psi(y)+z(y)\,,
			\end{equation}
			with $r\in\R$ and $z\in H_N^2(0,\pi)$ such that $\langle z,\psi\rangle_{L^2(0,\pi)}=0$. To do this, set
			\[
			\mathcal F(q,w,\rho):=\rho w''+g(q,w)\,,
			\]
			which is well-defined and $C^\infty$ from a neighborhood of $(0,0,1)$ in $\R\times H_N^2(0,\pi)\times\R$ to $L^2(0,\pi)$. Let then $P:L^2(0,\pi)\to\text{\normalfont span}\left\{\psi\right\}$ be the orthogonal projection on the span of $\psi$, and $P^\perp:=I-P$ (where $I$ denotes the identity on $L^2(0,\pi)$). That a function $w$ in the form \eqref{eq:formw} solves \eqref{eq:w} can thus be equivalently rewritten as
			\[
			\mathcal F^\perp(q,r,z,\rho):=P^\perp\mathcal F(q,r\psi+z,\rho)=0\,,\qquad \langle\mathcal F(q,r\psi+z,\rho),\psi\rangle_{L^2(0,\pi)}=0\,.
			\]
			In the next lemmas we identify $z$ as a function of $q,r,\rho$ to satisfy the first identity, and then find $\rho$ as a function of $q,r$ to fulfill the second condition.
			
			\begin{lemma}
				\label{lem:LSort}
				There exist $q_0,r_0,\eta_0>0$ and a unique $C^\infty$ map $\displaystyle z:(-q_0,q_0)\times(-r_0,r_0)\times(1-\eta_0,1+\eta_0)\to H_N^2(0,\pi)$, such that $\displaystyle \langle z(q,r,\rho),\psi\rangle_{L^2(0,\pi)}=0$, 
				\begin{equation}\label{eq:range-solved}
					\mathcal F^\perp(q,r,z(q,r,\rho),\rho)=0
				\end{equation}
				and
				\begin{equation}\label{eq:z-zero}
					z(q,0,\rho)=0\,,\qquad \partial_rz(q,0,\rho)=0\,.
				\end{equation}
			\end{lemma}
			\begin{proof}
				By Remark \ref{rem:g},
				\[
				D_z\mathcal F^\perp (q,0,0,1)h=P^\perp(h''+h)\,,
				\]
				so that, denoting by 
				\[
				X:=\left\{z\in H_N^2(0,\pi)\,:\,\langle z,\psi\rangle_{L^2(0,\pi)}=0\right\},\qquad Y:=\left\{z\in L^2(0,\pi)\,:\,\langle z,\psi\rangle_{L^2(0,\pi)}=0\right\},
				\]
				 $D_z\mathcal F^\perp(q,0,0,1):X\to Y$ is a bounded linear isomorphism, since both $X$ and $Y$ involve only functions orthogonal in $L^2(0,\pi)$ to the solutions of $h''+h=0$ in $H_N^2(0,\pi)$. Hence, a direct application of the Implicit Function Theorem \cite[Theorem 1.4]{AM} shows that there exists a unique $C^\infty$ map $z(q,r,\rho)$ defined on a neighborhood of $(0,0,1)$ satisfying \eqref{eq:range-solved}. The first identity in \eqref{eq:z-zero} then follows by $\mathcal F(q,0,\rho)=0$ for every $q,\rho$, while differentiating \eqref{eq:range-solved} with respect to $r$ at $r=0$ yields
				\[
				D_z\mathcal F^\perp(q,0,0,\rho)\partial_r z(q,0,\rho)+P^\perp D_w\mathcal F(q,0,\rho)\psi=0\,.
				\]
				Since $P^\perp D_w\mathcal F(q,0,\rho)\psi=P^\perp(\rho\psi''+\psi)=P^\perp((1-\rho)\psi)=0$, and $D_z\mathcal F^\perp(q,0,0,\rho)h=P^\perp(\rho h''+h)$ remains invertible from $X$ to $Y$ by continuity at $\rho=1$ (possibly reducing the considered neighborhood of $(0,0,1)$), we obtain the second part of \eqref{eq:z-zero}. 
			\end{proof}
			\begin{lemma}
				\label{lem:LSfd}
				There exist $q_1\in(0,q_0]$, $r_1\in(0,r_0]$ and a unique $C^\infty$ map $\rho:(-q_1,q_1)\times(-r_1,r_1)\to(1-\eta_0,1+\eta_0)$ such that 
				\begin{equation}\label{eq:rho-zero}
					\rho(q,0)=1
				\end{equation}
				and
				\begin{equation}\label{eq:reduced-solved}
					\langle\mathcal F(q,r\psi+z(q,r,\rho(q,r)),\rho(q,r)),\psi\rangle_{L^2(0,\pi)}=0\,,
				\end{equation}
				with $q_0,r_0, z(q,r,\rho)$ as in Lemma \ref{lem:LSort}. 
			\end{lemma}
			\begin{proof}
				Set $\displaystyle B(q,r,\rho):=\frac2\pi\langle\mathcal F(q,r\psi+z(q,r,\rho),\rho),\psi\rangle_{L^2(0,\pi)}$ and $\displaystyle\Psi(q,r,\rho):=\int_0^1\partial_r B(q,\theta r,\rho)\,d\theta$. Since $B$ is $C^\infty$ (because $\mathcal F$ and $z$ are) and $B(q,0,\rho)=0$ by Lemma \ref{lem:LSort} and $\mathcal{F}(q,0,\rho)=0$, then $\Psi$ is $C^\infty$ too. Moreover, by \eqref{eq:z-zero} and $\psi(y)=\cos y$,
				\begin{equation}
					\label{eq:Psi(q,0,rho)}
				\Psi(q,0,\rho)=\partial_rB(q,0,\rho)=\frac{2}{\pi}\langle D_w\mathcal F(q,0,\rho)\psi, \psi\rangle_{L^2(0,\pi)}=\frac{2}{\pi}\langle(1-\rho)\psi, \psi\rangle_{L^2(0,\pi)}=1-\rho,
				\end{equation}
				so that $\partial_\rho\Psi(q,0,\rho)=-1$. Therefore, we can apply the Implicit Function Theorem to $\Psi$ to obtain a unique $C^\infty$ map $\rho:(-q_1,q_1)\times(-r_1,r_1)\to(1-\eta_0,1+\eta_0)$ satisfying $\Psi(q,r,\rho(q,r))=0$. Since by definition $\Psi(q,r,\rho)=B(q,r,\rho)/r$ for every $r\neq0$, this directly implies \eqref{eq:reduced-solved}, and \eqref{eq:rho-zero} follows by the uniqueness of the map $\rho(q,r)$ and $\Psi(q,0,1)=0$ by \eqref{eq:Psi(q,0,rho)}.
			\end{proof}
			Given $(q,r)\in(-q_1,q_1)\times(-r_1,r_1)$, denote by $\rho_{q,r}:=\rho(q,r)$ and $z_{q,r}:=z(q,r,\rho(q,r))$ the maps obtained in Lemmas \ref{lem:LSort}--\ref{lem:LSfd}. Then, the function 
			\begin{equation}
				\label{eq:vqr}
				v_{q,r}(y):=1+r\psi(y)+z_{q,r}(y)
			\end{equation}
			 solves \eqref{eq:v2} with $\rho=\rho_{q,r}$. By construction, $v_{q,r}>0$ on $[0,\pi]$ and it is non-constant when $r\neq0$. It is in fact easy to see that $v_{q,r}$ is strictly decreasing on $[0,\pi]$ for every $r>0$, while it is strictly increasing on $[0,\pi]$ for every $r<0$.
			Indeed, Lemma \ref{lem:LSort} implies
			\begin{equation}\label{eq:profile-C1}
				\left\lVert\frac{v_{q,r}-1}{r}-\psi\right\rVert_{C^1(0,\pi)}\le C|r|
			\end{equation}
			for $(q,r)$ in a suitable neighborhood of $(0,0)$, with $C>0$ independent of $q$ and $r$. Combining with the fact that $v_{q,r}$ solves \eqref{eq:v2}, $g(q,s)=s+O(s^2)$ for $s$ close to $0$ uniformly on $q$, $\rho_{q,r}=1+O(r)$ by Lemma \ref{lem:LSfd} uniformly on $q$, and $\psi''(y)=-\cos y=-\psi(y)$, we obtain
			\begin{equation}\label{eq:profile-C2}
				\left\lVert\frac{v_{q,r}''}r+\psi\right\rVert_{L^\infty(0,\pi)}\le C|r|\,.
			\end{equation}
			Fix then $\delta\in(0,\pi/4)$ such that $|\psi(y)|\geq1/2$ for every $y\in[0,\delta]\cup[\pi-\delta,\pi]$. Since $\psi'(y)=-\sin(y)<0$ on $[\delta,\pi-\delta]$, by \eqref{eq:profile-C1} we have
			\begin{equation}
				\label{eq:v'<0}
			\frac{v_{q,r}'(y)}r\leq\psi'(y)+C |r|<0\qquad\text{on }[\delta,\pi-\delta]
			\end{equation}
			for every $r\in[-r_1,r_1]$ (possibly reducing $r_1$). For $r\in(0,r_1]$, this implies $v_{q,r}'<0$ on $[\delta,\pi-\delta]$. Moreover, by \eqref{eq:profile-C2} and the choice of $\delta$, we also have (again up to a further reduction of $r_1$)
			\[
			\begin{split}
			&\frac{v_{q,r}''(y)}r\leq-\psi(y)+C|r|\leq -\frac14\qquad\,\forall y\in[0,\delta]\\
			&\frac{v_{q,r}''(y)}r\geq-\psi(y)-C|r|\geq \frac14\qquad\quad\forall y\in[\pi-\delta,\pi]\,,
			\end{split}
			\]
			namely $v_{q,r}''(y)<0$ if $y\in[0,\delta]$ and $v_{q,r}''(y)>0$ if $y\in[\pi-\delta,\pi]$, for every $r>0$. Together with \eqref{eq:v'<0} and $v_{q,r}'(0)=v_{q,r}'(\pi)=0$, this yields $v_{q,r}'<0$ on $(0,\pi)$ for every $r>0$. Hence, $v_{q,r}$ is strictly decreasing on $[0,\pi]$ for every $r>0$. Since $\cos(\pi-y)=-\cos y$ for every $y\in[0,\pi]$, the invariance of \eqref{eq:v2} with respect to the reflection $y\mapsto\pi-y$ and the uniqueness of the maps $\rho_{q,r}$, $z_{q,r}$ for $(q,r)$ in a neighborhood of $(0,0)$ then ensure that $v_{q,-r}(y)=v_{q,r}(\pi-y)$ for every $q,r,y$. Hence, $v_{q,r}$ is strictly increasing on $[0,\pi]$ for every $r<0$. 
			
			Recall that the aim of this subsection is to prove that the mass of the strictly increasing solution to \eqref{eq:main2} is a strictly increasing function of $\lambda$ in a right neighborhood of $\lambda_*(p)$, uniformly on $p$ in a right neighborhood of $2$. In terms of $v_{q,r}$, this requires to investigate the dependence of the mass on the parameter $r$ locally around $0$ with $r<0$, uniformly on $q$ in a right neighborhood of $0$. To do this, we need the following asymptotic expansions for $\rho_{q,r}$ and $z_{q,r}$.
			\begin{lemma}
				\label{lem:expzrho}
				There exist $C>0$, $q_2\in(0,q_1]$, $r_2\in(0,r_1]$ such that, for every $(q,r)\in[-q_2,q_2]\times[-r_2,r_2]$ and every $y\in[0,\pi]$, it holds
				\begin{equation}
					\label{eq:expzrho}
					z_{q,r}(y)=\left(-\frac{1+q}{4}+\frac{1+q}{12}\cos(2y)\right)r^2+R_z(q,r,y)\,,\qquad \rho_{q,r}=1-\frac{(1+q)(4+q)}{12}r^2+R_{\rho}(q,r)\,,
				\end{equation}
				with
				\begin{equation}
					\label{eq:resti}
					\|R_z(q,r,\cdot)\|_{H^2 {(0,\pi)} }+|R_\rho(q,r)|\leq Cr^3\,,\qquad|\partial_r R_\rho(q,r)|\leq Cr^2\,.
				\end{equation}
			\end{lemma}
			\begin{proof}
			Recalling Lemmas \ref{lem:LSort}--\ref{lem:LSfd}, write
			\begin{equation}
				\label{eq:exp1}
				z_{q,r}(y)=z_2r^2+z_3r^3+O(r^4)\,,\qquad\rho_{q,r}=1+k_1r+k_2r^2+O(r^3)
			\end{equation}
			locally around $r=0$, for suitable $z_2,z_3$ depending on $q$ and $y$, and $k_1,k_2$ depending on $q$. Plugging into the equation \eqref{eq:v2} satisfied by $v_{q,r}$ and expanding $g(q,s)$ around $s=0$ as in Remark \ref{rem:g} then yields
			\begin{equation}
				\label{eq:expeq}
			\begin{split}
			(1+k_1r&\,+k_2r^2+O(r^3))(r\psi''+r^2z_2''+r^3z_3''+O(r^4))+r\psi+r^2z_2+r^3z_3+O(r^4)\\
			&\,+\frac{q+1}{2}\left(r\psi+r^2z_2+r^3z_3+O(r^4)\right)^2+\frac{(q+1)(q-1)}{6}\left(r\psi+r^2z_2+r^3z_3+O(r^4)\right)^3\\
			&\qquad\,\,\,\qquad\qquad\qquad\qquad\qquad\qquad\qquad+O\left(\left(r\psi+r^2z_2+r^3z_3+O(r^4)\right)^4\right)=0\,.
			\end{split}
			\end{equation}
			For the identity to be satisfied locally around $r=0$, the coefficient of each power of $r$ must be equal to zero. This is automatically true for the term $r$, since it is multiplied by $\psi''+\psi$, which is identically zero on $[0,\pi]$ by definition of $\psi$. Collecting all the terms involving $r^2$ gives the condition
			\[
			z_2''+k_1\psi''+z_2+\frac{q+1}{2}\psi^2=0\,.
			\]
			Since $\displaystyle\langle z_{q,r},\psi\rangle_{L^2(0,\pi)}=0$, the same is true for $z_2$, so that $\langle z_2'',\psi\rangle_{L^2(0,\pi)}=-\langle z_2,\psi\rangle_{L^2(0,\pi)}=0$ (the first identity follows by $\psi''=-\psi$). Therefore, multiplying the previous equation by $\psi$, integrating over $[0,\pi]$ and recalling that $\displaystyle\int_0^\pi\psi(y)^3\,dy=0$ and $\displaystyle\int_0^\pi\psi''(y)\psi(y)\,dy=-\int_0^\pi\psi(y)^2\,dy\neq0$, we obtain 
			\begin{equation}
				\label{eq:k1}
				k_1=0\,.
			\end{equation}
			In turn, this implies
			\[
			z_2''+z_2=-\frac{q+1}{2}\psi^2\qquad\forall y\in[0,\pi]\,, 
			\]
			and the unique solution of this equation satisfying $\displaystyle z_2'(q,0)=z_2'(q,\pi)=0$ for every $q$ is
			\begin{equation}
				\label{eq:z2}
			z_2(q,y)=-\frac{1+q}{4}+\frac{1+q}{12}\cos(2y)\,.
			\end{equation}
			We then focus on the coefficients of $r^3$ in \eqref{eq:expeq}, yielding the condition
			\begin{equation}
				\label{eq:cond3}
			z_3''+k_2\psi''+z_3+(q+1)\psi z_2+\frac{(q+1)(q-1)}{6}\psi^3=0\,.
			\end{equation}
			Since $z_3$ is orthogonal in $L^2(0,\pi)$ to $\psi$ (because so is $z$), then $\langle z_3''+z_3,\psi\rangle_{L^2(0,\pi)}=0$. Moreover, by \eqref{eq:z2},
			\[
			\langle\psi z_2,\psi\rangle_{L^2(0,\pi)}=-\frac{1+q}{4}\int_0^\pi\cos^2(y)\,dy+\frac{1+q}{12}\int_0^\pi\cos^2(y)\cos(2y)\,dy=-\frac{5(1+q)\pi}{48}\,,
			\]
			so that multiplying \eqref{eq:cond3} by $\psi$ and integrating over $[0,\pi]$ leads to
			\[
			-\frac\pi2k_2-\frac{5(1+q)^2\pi}{48}+\frac{3(q+1)(q-1)\pi}{48}=0\,,
			\]
			namely 
			\begin{equation}
				\label{eq:k2}
				k_2=-\frac{(1+q)(4+q)}{12}\,.
			\end{equation}
			Combining \eqref{eq:exp1}, \eqref{eq:k1}, \eqref{eq:z2} and \eqref{eq:k2} gives \eqref{eq:expzrho}. The uniform estimates \eqref{eq:resti} then follow by the regularity of $z_{q,r}$ and $\rho_{q,r}$.
			\end{proof}
			We can now prove the main result of this subsection.
			\begin{proposition}
				\label{prop:monl*}
				Let $M_p(\alpha)$ be the function defined in \eqref{eq:M-param}. There exist $\varepsilon_1>0$ and $\underline\alpha\in(0,1)$ such that $M_p'(\alpha)<0$ for every $\alpha\in(\underline\alpha,1)$ and every $p\in(2,2+\varepsilon_1)$. 
			\end{proposition}
			\begin{proof}
				Recall that by definition $q=p-2$. The discussion developed so far shows that the strictly increasing solution $u_\lambda$ to \eqref{eq:main2} with $\lambda$ close to $\lambda_*(p)$ corresponds to the solution $v_{q,r}$ to \eqref{eq:v2} with $r<0$ through the transformation $u_\lambda(x)=\lambda^{1/q}v_{q,r}(\pi x)$. Hence, by \eqref{eq:rho} we have
				\begin{equation}
					\label{eq:Mr}
				M_p(\lambda)=\int_0^1 u_\lambda(x)^2\,dx=\frac{\lambda(q,r)^{\frac2q}}\pi\int_0^\pi v_{q,r}(y)^2\,dy=\frac{\pi^{\frac4q-1}}{(q\rho_{q,r})^{2/q}}\int_0^\pi v_{q,r}(y)^2\,dy=:g_q(r)
				\end{equation}
				with $r<0$. Write
				\[
				\partial_r\log g_q(r)=-\frac{2\partial_r\rho_{q,r}}{q\rho_{q,r}}+\frac{\partial_r S_{q,r}}{S_{q,r}}\,,
				\]
				with $\displaystyle S_{q,r}:=\int_0^\pi v_{q,r}(y)^2\,dy$. By Lemma \ref{lem:expzrho},
				\[
				-\frac{2\partial_r\rho_{q,r}}{q\rho_{q,r}}=\frac{(1+q)(4+q)r+O(r^2)}{3q(1+O(r^2))}
				\]
				for $(q,r)\in(0,q_2]\times[-r_2,0)$, whereas by \eqref{eq:vqr} and Lemmas \ref{lem:LSort}--\ref{lem:LSfd} we immediately have 
				\[
				\frac\pi2\leq S_{q,r}\leq\frac{3\pi}2\qquad\forall (q,r)\in(0,q_2]\times[-r_2,0)
				\]
				(up to reducing $q_2$ and $r_2$ if necessary). Moreover, since $\displaystyle\partial_r v_{q,r}=\psi+\partial_r z_{q,r}$ and $\partial^2_{rr}v_{q,r}=\partial^2_{rr}z_{q,r}$, by Lemma \ref{lem:expzrho} it holds
				\[
				\partial_r S_{q,0}=2\int_0^\pi v_{q,0}(y)\partial_r v_{q,0}(y)\,dy=2\int_0^\pi\psi(y)\,dy=0
				\]
				and thus
				\[
				|\partial_r S_{q,r}|\leq\int_r^0|\partial^2_{tt}S_{q,t}  | \,dt\leq C |r|
				\]
				for every $(q,r)\in(0,q_2]\times[-r_2,0)$. All in all, we obtain
				\[
				\partial_r\log g_q(r)\leq \frac{(1+q)(4+q)r}{3q}-Cr\leq\left(\frac{4}{3q_2}-C\right)r<0
				\]
				for every $(q,r)\in(0,q_2]\times[-r_2,0)$, up to possibly reducing $q_2$ if needed. By the monotonicity of the logarithm, this shows that $g_q(r)$ is strictly decreasing with respect to $r\in[-r_2,0)$ for every $q\in(0,q_2]$. 
				
				Moreover, by \eqref{eq:vqr}, Lemma \ref{lem:expzrho} and $v_{q,r}(0)=u_1(0)=\alpha(q,r)$ for every $r<0$, we obtain $\alpha(q,r)=1+r+O(r^2)$ uniformly on $q\in(0,q_2]$, that is $\partial_{r}\alpha(q,r)=1+O(r)$. Hence, up to possibly reducing $r_2$, we have $\displaystyle \partial_{r}\alpha(q,r)\geq1/2$ uniformly in $(q,r)\in(0,q_2]\times[-r_2,0)$. This guarantees that $\alpha(q,r)$ is strictly increasing with respect to $r\in[-r_2,0)$ uniformly in $q\in(0,q_2]$, and together with $\alpha(q,0)=1$ also shows that $\alpha(q,-r_2)\leq1-r_2/2=:\underline\alpha$ for every $q\in(0,q_2]$. Therefore, the image of $\alpha(q,r)$ as $r$ varies in $[-r_2,0)$ contains the interval $[\underline\alpha,1)$ for every $q\in(0,q_2]$. Combining with \eqref{eq:Mr} and the monotonicity of $g_q(r)$, we obtain that $M_p(\alpha)$ as in \eqref{eq:M-param} is strictly decreasing with respect to $\alpha\in(\underline\alpha,1)$, for every $q\in(0,q_2]$.
			\end{proof}
			
			\subsection{Monotonicity of $M_p$ for $\alpha$ close to $0$}
			The main result of this subsection is the following.
			\begin{proposition}
				\label{prop:a=0}
				There exist $\overline q>0$, $\alpha_0>0$, $C>0$ such that
				\begin{equation}
					\label{eq:boundLN}
					\frac{L_{2+q}'(\alpha)}{L_{2+q}(\alpha)}\leq-C\,,\qquad\left|\frac{N_{2+q}'(\alpha)}{N_{2+q}(\alpha)}\right|\leq C\qquad\forall (q,\alpha)\in(0,\overline q]\times(0,\alpha_0]\,.
				\end{equation}
			In particular, $M_{2+q}'(\alpha)<0$ for every $(q,\alpha)\in(0,\overline q]\times(0,\alpha_0]$.
			\end{proposition}
			To prove Proposition \ref{prop:a=0} we need various intermediate results. We begin with the first bound in \eqref{eq:boundLN}, splitting $L_{2+q}(\alpha)$ as 
			\begin{equation}
				\label{eq:splitL}
			L_{2+q}(\alpha)=\int_\alpha^1\frac{dt}{\sqrt{2\left(F_{2+q}(t)-F_{2+q}(\alpha)\right)}}+\int_1^{\beta_{2+q}(\alpha)}\frac{dt}{\sqrt{2\left(F_{2+q}(t)-F_{2+q}(\alpha)\right)}}=:I_{q,1}(\alpha)+I_{q,2}(\alpha)
			\end{equation}
			and treating separately the two integrals. Let us start with the integral on $[\alpha,1]$, for which it is convenient to set
			\begin{equation}
				\label{eq:az}
			a:=-\log\alpha\,,\qquad z:=qa\,,
			\end{equation}
			so that $\alpha=e^{-z/q}$ and, changing the variable $t=e^{-s/q}$,
			\begin{equation}
				\label{eq:H}
			qI_{q,1}(\alpha)=\int_{0}^{z}\frac{ds}{\sqrt{h_q(s)-e^{-2(z-s)/q}h_q(z)}}=:H_q(z)\,,
			\end{equation}
			where
			\begin{equation}
				\label{eq:hq}
			h_q(s):=1-\frac{2}{2+q}e^{-s}\qquad\forall s\in\R\,.
			\end{equation}
		\begin{remark}
			\label{rem:hq}
			For later use, observe that
			\begin{equation}
				\label{eq:stimeh1}
				h_q'(s)=\frac2{2+q}e^{-s}\,,\qquad h_q(s)+h_q'(s)=1\,,
			\end{equation}
		and that there exist universal constants $c,C>0$ such that
		\begin{equation}
			\label{eq:stimeh2}
			c\min\left\{1,q+s\right\}\leq h_q(s)\leq C\min\left\{1,q+s\right\}\qquad\forall q\in(0,1],\,s\geq0\,.
		\end{equation}
		\end{remark}
			The next lemma establishes useful estimates on $H_q$ and its first derivative for $q$ close to $0$ and $a$ large, which corresponds to $\alpha$ close to $0$.
			\begin{lemma}
				\label{lem:Hq}
				For every $\sigma>0$ there exist $a_\sigma>0$ and $q_\sigma>0$ such that, for every $(q,a)\in(0,q_\sigma)\times(a_\sigma,+\infty)$, it holds
				\begin{equation}
					\label{eq:HK}
					\left|\frac{H_q(z)}{K(z)}-1\right|\leq\sigma\,,\qquad\left|\frac{H_q'(z)}{K'(z)}-1\right|\leq\sigma\,,\qquad\forall z\geq a_\sigma q\,,
				\end{equation}
			with 
			\begin{equation}
				\label{eq:K}
				K(z):=\int_0^z\frac{ds}{\sqrt{1-e^{-s}}}\,.
			\end{equation}
			\end{lemma}
		\begin{proof}
			Throughout, $C>0$ denotes a universal constant that may change from line to line.
			
			Observe first that, for every $q\in(0,1)$ and $z>0$, if we set 
			\begin{equation}
				\label{eq:Htilde}
				\widetilde H_q(z):=\int_0^z\frac{ds}{\sqrt{1-\frac2{2+q}e^{-s}}}=\int_0^z\frac{ds}{\sqrt{h_q(s)}}\,,
			\end{equation}
		then there exists $C>0$ such that
		\begin{equation}
			\label{eq:KHtilde}
			0\leq K(z)-\widetilde H_q(z)\leq C\sqrt q\,,\qquad\left|K'(z)-\widetilde H_q'(z)\right|\leq \frac{Cq}{(1-e^{-z})^{3/2}}\,.
		\end{equation}
		Indeed, the lower bound in the first estimate is immediate. As for the upper bound, it is enough to note that, setting $\eta_q:=q/(2+q)$, we have 
		\[
		\int_0^{\eta_q}\left|\frac{1}{\sqrt{1-\frac2{2+q}e^{-s}}}\right|\,ds+\int_0^{\eta_q}\left|\frac{1}{\sqrt{1-e^{-s}}}\right|\,ds\leq C \int_0^{\eta_q}\frac{ds}{\sqrt s}=2C\sqrt{\eta_q}\,,
		\]
		whereas, by the Mean Value Theorem applied to the function $x\mapsto x^{-1/2}$,
		\[
		 \int_{\eta_q}^1\left(\frac{1}{\sqrt{1-e^{-s}}}-\frac{1}{\sqrt{1-\frac2{2+q}e^{-s}}}\right)\,ds\leq C\int_{\eta_q}^1\frac{\eta_q}{s^{3/2}}\,ds\leq 2C\sqrt{\eta_q}
		\]
		and
		\[
		\int_1^{+\infty}\left(\frac{1}{\sqrt{1-e^{-s}}}-\frac{1}{\sqrt{1-\frac2{2+q}e^{-s}}}\right)\,ds\leq C\eta_q\int_1^{+\infty}e^{-s}\,ds\leq C\eta_q\,.
		\]
		All in all, this shows that $K(z)-\widetilde H_q(z)\leq C\sqrt{\eta_q}\leq C\sqrt q$ for every $q\in(0,1)$. The second estimate in \eqref{eq:KHtilde} follows analogously, noting that $\displaystyle K'(z)=(1-e^{-z})^{-1/2}$, $\displaystyle \widetilde H_q'(z)=h_q(z)^{-1/2}$, and by the Mean Value Theorem
		\[
		\left|K'(z)-\widetilde H_q'(z)\right|\leq \frac{C\eta_q}{(1-e^{-z})^{3/2}}\leq\frac{Cq}{(1-e^{-z})^{3/2}} \qquad\forall z>0\,.
		\]
		In view of \eqref{eq:KHtilde}, we claim that to complete the proof of Lemma \ref{lem:Hq} it is enough to show that there exist constants $\overline a>0$ and $\overline q\in(0,1)$ such that
		\begin{equation}
			\label{eq:Rq}
			0\leq  H_q(z)-\widetilde H_q(z)\leq\frac{Cq}{\sqrt{h_q(z)}}\,,\qquad\left|H_q'(z)-\widetilde H_q'(z)\right|\leq \frac{Cq}{h_q(z)^{3/2}}\,,\qquad\forall q\in(0,\overline q),\, z\geq \overline a q\,.
		\end{equation}
		Indeed, assume that \eqref{eq:Rq} holds. Fix $\sigma>0$ and let $a_\sigma\geq \overline a$ and $q_\sigma\in(0,\overline q)$ be two constants to be determined. For every $q\in(0,q_\sigma)$ and $z\geq a_\sigma q$,  we have 
		\begin{equation}
			\label{eq:q/r}
			\left|\frac{h_q(z)}{1-e^{-z}}-1\right|=\frac{qe^{-z}}{(2+q)(1-e^{-z})}\leq C \left(\frac1{a_\sigma}+q_\sigma\right)
		\end{equation}
		for some universal constant $C>0$. Moreover, by \eqref{eq:K} we have both $K(z)\geq 2\sqrt z$ (since $1-e^{-s}\leq s$ for every $s>0$) and $K(z)\geq z/\sqrt{1-e^{-z}}$ (since the map $s\mapsto (1-e^{-s})^{-1/2}$ is strictly decreasing on $(0,+\infty)$). Combining with \eqref{eq:KHtilde}, \eqref{eq:Rq} and \eqref{eq:q/r} then gives
		\[
	\frac{|H_q(z)-K(z)|}{K(z)}\leq \frac{C\sqrt q}{K(z)}+\frac{Cq}{K(z)\sqrt{h_q(z)}}\leq C\left(\sqrt\frac q z +\frac{q}{z}\sqrt\frac{1-e^{-z}}{h_q(z)}\right)\leq C\left(\frac1{\sqrt{a_\sigma}}+\frac1{a_\sigma}\right)
		\]
		for every $z\geq a_\sigma q$, namely the first estimate in \eqref{eq:HK} provided $a_\sigma$ is large enough. Furthermore, since $K'(z)=(1-e^{-z})^{-1/2}$, using again \eqref{eq:KHtilde}, \eqref{eq:Rq} and \eqref{eq:q/r} yields
		\[
		\frac{\left|H_q'(z)-K'(z)\right|}{K'(z)}\leq Cq\sqrt{1-e^{-z}}\left(\frac1{h_q(z)^{3/2}}+\frac1{(1-e^{-z})^{3/2}}\right)\leq \frac{Cq}{1-e^{-z}}\leq C \left(\frac1{a_\sigma}+q_\sigma\right)\,,
		\]
		that is the second estimate in \eqref{eq:HK} as soon as $a_\sigma$ is large and $q_\sigma$ is small. 
		
		To conclude, we are thus left to prove \eqref{eq:Rq}. That $H_q(z)\geq\widetilde H_q(z)$ for every $z>0$ follows directly by the definition of the two functions as in \eqref{eq:H} and \eqref{eq:Htilde}. 
		Let then $\overline a>0$ and $\overline q>0$ be constants to be determined, and set $R_q(z):=H_q(z)-\widetilde H_q(z)$. Recalling again \eqref{eq:H}, \eqref{eq:Htilde},  and changing the variable $r=(z-s)/q$, we can write
		\[
		R_q(z)=q\int_0^{z/q}A_q(z,r)\,dr\qquad\forall z\geq \overline a q,\, q \in(0,\overline q)\,,
		\]
		where
		\[
		A_q(z,r):=\frac{\varphi(w_q(z,r))}{\sqrt{h_q(z-qr)}}\,,\qquad \varphi(w)=(1-w)^{-1/2}-1\,,\qquad w_q(z,r):=e^{-2r}\frac{h_q(z)}{h_q(z-qr)}\,.
		\]
		Note that
		\begin{equation}
			\label{eq:varphi_est}
			\varphi(w)\leq C(1-w)^{-1/2}\quad \forall w\in(0,1),\qquad \varphi(w)\leq C_{w_0}w\quad\forall w\in(0,w_0]\,,
		\end{equation}
		where $C>0$ is a universal constant and $C_{w_0}>0$ depends only on $w_0\in(0,1)$. Differentiating $w_q(z,r)$ with respect to $r$ gives
		\begin{equation}
			\label{eq:derw_r}
		\partial_r w_q(z,r)=w_q(z,r)\left(-2+q\frac{h_q'(z-qr)}{h_q(z-qr)}\right)=-2w_q(z,r)\frac{1-e^{-(z-qr)}}{h_q(z-qr)}<0\qquad\forall r\in(0,z/q]\,,
		\end{equation}
		which, together with $w_q(z,0)=1$, gives $w_q(z,r)\in(0,1)$ for $r\in(0,z/q]$.
		Furthermore, since \eqref{eq:stimeh2} and $z\geq \overline a q$ ensure that $\displaystyle h_q(z/2)/h_q(z)\geq C$ for some $C>0$ provided $\overline a$ is large enough, together with \eqref{eq:derw_r} and the fact that $z-qr\in[z/2,z]$ as $r\in[0,z/(2q)]$ we obtain that $\partial_r w_q(z,r)\leq -Cw_{q}(z,r)$ for every $r\in[0,z/(2q)]$, with $C>0$ independent of $q,r,z$, in turn implying that
		\begin{equation}
		\label{eq:w1}
		e^{-2r}\leq w_q(z,r)\leq e^{-cr}\qquad\forall r\in[0,z/(2q)]
		\end{equation}
		for some universal constant $c>0$. Combining with \eqref{eq:varphi_est} then gives
		\[
		A_q(z,r)\leq \frac{C}{\sqrt{h_q(z)}}\left(r^{-1/2}\chi_{(0,1)}(r)+e^{-cr}\chi_{[1,+\infty)}(r)\right)\qquad\forall r\in(0,z/(2q)]\,,
		\]
		namely
		\begin{equation}
			\label{eq:Rq_0z/2q}
			q\int_0^{\frac z{2q}}A_q(z,r)\,dr\leq \frac{Cq}{\sqrt{h_q(z)}}\,.
		\end{equation}
		Conversely, if $r\in(z/(2q),z/q]$, then choosing $\overline a$ large enough \eqref{eq:w1} guarantees $\displaystyle w_q(z,r)\leq 1/2$, so that taking $w_0=1/2$ in \eqref{eq:varphi_est} gives
		\begin{equation}
		\label{eq:R2}
		\begin{split}
		q\int_{\frac{z}{2q}}^{\frac zq}A_q(z,r)\,dr\leq Cq \int_{\frac{z}{2q}}^{\frac zq} \frac{w_q(z,r)}{\sqrt{h_q(z-qr)}}\,dr&\,= Cqh_q(z)\int_{\frac{z}{2q}}^{\frac zq}\frac{e^{-2r}}{h_q(z-qr)^{3/2}}\,dr\\
		&\,=\frac{Cq}{\sqrt{h_q(z)}}\int_0^{\frac{z}{2q}}e^{-2(z/q-u)}\left(\frac{h_q(z)}{h_q(qu)}\right)^{3/2}\,du\,,
		\end{split}
		\end{equation}
		where we used the change of variable $u=z/q-r$ for the last identity. Note, however, that
		\begin{equation}
			\label{eq:sup_m}
			\sup_{\substack{q\in(0,\overline q) \\ z/q\geq \overline a}}\int_0^{\frac{z}{2q}}e^{-2(z/q-u)}\left(\frac{h_q(z)}{h_q(qu)}\right)^m\,du<+\infty\qquad\forall m>1\,.
		\end{equation}
		Indeed, by Remark \ref{rem:hq}, for every $q\in(0,1)$ and $u\geq0$ we have
		\[
		\begin{split}
		h_q(z)&\,\leq C\min\left\{1,q(1+z/q)\right\}\leq Cq(1+z/q)\,,\\
		h_q(qu)&\,\geq c\min\left\{1,q(1+u)\right\}\geq c\min\left\{1,q\right\}=cq\,,
		\end{split}
		\]
		entailing $\displaystyle h_q(z)/h_q(qu)\leq C(1+z/q)/c$, and thus 
		\[
		\int_0^{\frac{z}{2q}}e^{-2(z/q-u)}\left(\frac{h_q(z)}{h_q(qu)}\right)^m\,du\leq \frac{Cz(1+z/q)^me^{-z/q}}{2q}\,,
		\]
		that gives \eqref{eq:sup_m} since the right-hand side is uniformly bounded for $z/q\geq\overline a$. Therefore, applying \eqref{eq:sup_m} with $m=3/2$ in \eqref{eq:R2} leads to 
        \[
		q  \int_{\frac z{2q}}^{\frac{q}{z}} A_q(z,r)\,dr\leq \frac{Cq}{\sqrt{h_q(z)}}\,,
		\]
		that coupled with \eqref{eq:Rq_0z/2q} proves the first estimate in \eqref{eq:Rq}. 
		
		We then turn to the estimate for $R_q'(z)$. For fixed $r>0$, differentiating $w_q$ with respect to $z$ and using \eqref{eq:stimeh1} gives
		\begin{equation}
		\label{eq:derw_z}
		\partial_z w_q(z,r)=w_q(z,r)\left(\frac{h_q'(z)}{h_q(z)}-\frac{h_q'(z-qr)}{h_q(z-qr)}\right)=w_q(z,r)\left(\frac1{h_q(z)}-\frac1{h_q(z-qr)}\right)\,,
		\end{equation}
		whereas differentiating $A_q$ with respect to $z$ leads to
		\begin{equation}
			\label{eq:derA_z}
		\partial_z A_q(z,r)=-\frac{h_q'(z-qr)}{2h_q(z-qr)^{3/2}}\varphi(w_q(z,r))+\frac{\varphi'(w_q(z,r))\partial_z w_q(z,r)}{\sqrt{h_q(z-qr)}}\,.
		\end{equation}
		Since $\displaystyle 0\leq h_q(z)-h_q(z-qr)=\int_{z-qr}^z h_q'(s)\,ds\leq qr$, recalling that $h_q(z-qr)\geq h_q(z/2)\geq Ch_q(z) $ for every $r\in(0,z/(2q))$ and for some universal constant $C>0$, by \eqref{eq:derw_z} we obtain
		\[
		|\partial_z w_q(z,r)|\leq \frac{Cqr}{h_q(z)^2}w_q(z,r)\qquad\forall r\in(0,z/(2q))\,.
		\]
		Coupling with $\varphi'(w)=(1-w)^{-3/2}/2$, \eqref{eq:varphi_est}, \eqref{eq:w1} and \eqref{eq:derA_z} yields
		\begin{equation}
			|\partial_z A_q(z,r)|\leq\frac{C}{h_q(z)^{3/2}}\left(1+\frac{q}{h_q(z)}\right)\left(r^{-1/2}\chi_{(0,1)}(r)+(1+r)e^{-cr}\chi_{[1,+\infty)}(r)\right)
		\end{equation}
		for every $r\in(0,z/(2q)]$, that by \eqref{eq:stimeh2} and $z\geq \overline aq$ can be further simplified as
		\begin{equation}
			\label{eq:estdA_z1}
				|\partial_z A_q(z,r)|\leq\frac{C}{h_q(z)^{3/2}}\left(r^{-1/2}\chi_{(0,1)}(r)+(1+r)e^{-cr}\chi_{[1,+\infty)}(r)\right)\qquad\forall r\in(0,z/(2q)].
		\end{equation} 
		For $r\in(z/(2q),z/q]$, we already observed that $w_q(z,r)\leq1/2$ provided $\overline a$ is taken sufficiently large, so that \eqref{eq:varphi_est}, \eqref{eq:derw_z}, \eqref{eq:derA_z}, the fact that $h_q(z-qr)\leq h_q(z)$, and the definition of $w_q$ give 
		\begin{equation}
		\label{eq:estdA_z2}
		|\partial_z A_q(z,r)|\leq \frac{Ce^{-2r}h_q(z)}{h_q(z-qr)^{5/2}}\qquad\forall r\in(z/(2q),z/q]\,.
		\end{equation}
		Let now $z_0>q\overline a$ be fixed and take $\delta>0$ such that $z_0-\delta>\overline aq$. For every $s\in(-\delta,\delta)$, the definition of $R_q$ gives
		\[
		\frac{R_q(z_0+s)-R_q(z_0)}{s}=q\int_{0}^{\frac{z_0}{q}}\frac{A_q(z_0+s,r)-A_q(z_0,r)}{s}\,dr+\frac qs\int_{\frac{z_0}q}^{\frac{z_0+s}q}A_q(z_0+s,r)\,dr\,.
		\]
Since \eqref{eq:estdA_z1} and \eqref{eq:estdA_z2} show that $\partial_z A_q(z,r)$ is dominated by an integrable function of $r$ on $[0,z/q]$ uniformly in $z\in(z_0-\delta,z_0+\delta)$, the Dominated Convergence Theorem implies
		\[
		\lim_{s\to0}\int_{0}^{\frac{z_0}q}\frac{A_q(z_0+s,r)-A_q(z_0,r)}{s}\,dr=\int_0^{\frac{z_0}q}\partial_z A_q(z_0,r)\,dr\,,
		\]
		whereas the continuity of $A_q$ guarantees that
		\[
		\lim_{s\to0}\frac qs\int_{\frac{z_0}q}^{\frac{z_0+s}q}A_q(z_0+s,r)\,dr=\lim_{s\to0}\frac{1}{\frac{z_0+s}q-\frac{z_0}q}\int_{\frac{z_0}q}^{\frac{z_0+s}q}A_q(z_0+s,r)\,dr=A_q(z_0,z_0/q)\,.
		\] 
		All in all, this shows that 
		\begin{equation}
		\label{eq:R'}
		R_q'(z_0)=q\int_0^{\frac{z_0}q}\partial_z A_q(z_0,r)\,dr+A_q(z_0,z_0/q)\,.
		\end{equation}
		By \eqref{eq:estdA_z1} and \eqref{eq:estdA_z2} we then have
		\begin{equation}
			\label{eq:R'1}
		\begin{split}
			\left|q\int_0^{z_0/q}\partial_z A_q(z_0,r)\,dr\right|\leq&\,\frac{Cq}{h_q(z_0)^{3/2}} \int_0^{\frac{z_0}{2q}}|r^{-1/2}\chi_{(0,1)}(r)+(1+r)e^{-cr}\chi_{[1,+\infty)}(r)|\,dr\\
			&\qquad+\frac{Cq}{h_q(z_0)^{3/2}}\int_{\frac{z_0}{2q}}^{\frac{z_0}q}e^{-2r}\left(\frac{h_q(z_0)}{h_q(z_0-qr)}\right)^{5/2}\,dr\\
			\leq&\,\frac{Cq}{h_q(z_0)^{3/2}}\,,
		\end{split}
		\end{equation}
		where the second integral has been bounded relying on the change of variable $u=z_0/q-r$ and applying \eqref{eq:sup_m} with $m=5/2$. Moreover, by the definition of $A_q$ and $w_q$, the fact that $w_q(z_0,z_0/q)\leq1/2$, and \eqref{eq:varphi_est}, we also have
		\begin{equation}
			\label{eq:R'2}
		A_q(z_0,z_0/q)\leq \frac{Ce^{-2z_0/q}h_q(z_0)}{h_q(0)^{3/2}}\leq \frac{Ce^{-2z_0/q}h_q(z_0)}{q^{3/2}}\leq \frac{Cq}{h_q(z_0)^{3/2}}\,,
		\end{equation}
		because by \eqref{eq:stimeh2} one has $e^{-2z_0/q}(h_q(z_0)/q)^{5/2}\leq C(z_0/q+1)^{5/2}e^{-2z_0/q}$, which is bounded uniformly in $z_0\geq\overline a q$. Coupling \eqref{eq:R'}, \eqref{eq:R'1} and \eqref{eq:R'2} gives the second estimate of \eqref{eq:Rq} and completes the proof of the lemma.
		\end{proof}
		
		We address now the second integral $I_{q,2}$ in \eqref{eq:splitL}.

		\begin{lemma}
			\label{lem:I2}
			Given $a,z$ as in \eqref{eq:az}, set $U_q(z):=\sqrt q I_{q,2}(e^{-z/q})$. Then there exist $\overline a>0$, $\overline q>0$, and $C>0$ such that, for every $(q,a)\in(0,\overline q)\times(\overline a,+\infty)$, it holds
			\begin{equation}
			\label{eq:U}
			|U_q(z)|\leq C\,,\qquad |U_q'(z)|\leq \frac{Ce^{-2a}(1-e^{-z})}{q^{2}}\qquad\forall z>0.
			\end{equation}
		\end{lemma}
		\begin{proof}
			We introduce the rescaled function
			\begin{equation}
				\label{eq:Phi}
				\Phi_q(t):=\frac{F_{2+q}(t)}{q}=\frac1q\left(\frac{t^2}{2}-\frac{t^{2+q}}{2+q}\right),\qquad t\geq0\,,
			\end{equation}
			which as $q\to0^+$ converges uniformly in $C^2$ on compact subsets of $(0,+\infty)$ to the function $\displaystyle\Phi(t):=\frac{t^2}{4}-\frac{t^2\log t}2$. Setting $b_q:=(1+q/2)^{1/q}$, we have $\Phi_q(b_q)=0$, $b_q\to\sqrt e$ as $q\to0^+$, and $\Phi_q'(b_q)=-b_q/2\leq-C$, for some universal constant $C>0$ and for every $q$ in a right neighborhood of $0$. By the convergence of $\Phi_q$ to $\Phi$, it then follows that there exist $q_0>0$ and $E_0>0$ such that, for every $(q,E)\in[0,q_0]\times[0,E_0]$, the equation $\Phi_q(t)=E$ has a unique solution $b_q(E)>1$, such that $b_q\in C^1(0,E_0)$ and $\Phi_q'(b_q(E))\leq-C<0$. Letting $t=b_q(E)-(b_q(E)-1)y^2$ for $y\in[0,1]$, we can rewrite
			\[
			\begin{split}
				\Phi_q(t)-E&\,=\Phi_q(b_q(E)-(b_q(E)-1)y^2)-\Phi_q(b_q(E))=:y^2 D_q(E,y)\,,
			\end{split}
			\]
			where $\displaystyle D_q(E,y):=-(b_q(E)-1)\int_0^1\Phi_q'(b_q(E)-\tau (b_q(E)-1)y^2)\,d\tau$ is well-defined for every $(q,E,y)\in[0,q_0]\times[0,E_0]\times[0,1]$. Moreover, $D_q\in C^1([0,E_0]\times[0,1])$, because $b_q\in C^1(0,E_0)$ by construction and the argument of $\Phi_q'$ in the integrand is bounded away from zero uniformly on $[0,q_0]\times[0,E_0]\times[0,1]$. Furthermore, since $b_q(E)-1>0$, the monotonicity of $\Phi_q$ on $(1,b_q(E))$ ensures that $D_q(E,y)>0$ for every $y>0$, whereas $D_q(E,0)=-(b_q(E)-1)\Phi_q'(b_q(E))\geq C>0$ uniformly on $[0,q_0]\times[0,E_0]$. Hence, by compactness there exists $c>0$ such that $D_q(E,y)\geq c$ uniformly on $[0,q_0]\times[0,E_0]\times[0,1]$, ensuring that there exists $C>0$ such that
			\begin{equation}
				\label{eq:D}
			\left|\int_0^1\frac{dy}{\sqrt{D_q(E,y)}}\right|+\left|\partial_E\int_0^1\frac{dy}{\sqrt{D_q(E,y)}}\right|\leq C
			\end{equation}
			uniformly on $[0,q_0]\times[0,E_0]$.
			
			Since 
			\begin{equation}
			\label{eq:I2new}
			\sqrt qI_{q,2}(\alpha)=\int_1^{b_q(\Phi_q(\alpha))}\frac{dt}{\sqrt{2\left(\Phi_{q}(t)-\Phi_{q}(\alpha)\right)}}=\sqrt2 (b_q(\Phi_q(\alpha))-1)\int_0^1\frac{dy}{\sqrt{D_q(\Phi_q(\alpha),y)}}\,,
			\end{equation}
			and by \eqref{eq:az}
			\[
			\Phi_q(\alpha)=\Phi_q(e^{-z/q})=:E_q(z)\,,
			\]
			so that
			\[
			E_q'(z)=-\frac{e^{-z/q}}{q}\Phi_q'(e^{-z/q})=-\frac{e^{-2z/q}-e^{-(q+2)z/q}}{q^2}=-\frac{e^{-2a}(1-e^{-z})}{q^2}\,,
			\]
			\eqref{eq:U} follows.
		\end{proof}
	
		Coupling Lemmas \ref{lem:Hq}--\ref{lem:I2} yields the first estimate in \eqref{eq:boundLN}.
		\begin{corollary}
			\label{cor:L}
			There exist $\overline q>0$, $\alpha_0>0$ and $C>0$ such that $\displaystyle \frac{L_{2+q}'(\alpha)}{L_{2+q}(\alpha)}\leq-C$ for every $(q,\alpha)\in(0,\overline q]\times(0,\alpha_0]$. 
		\end{corollary}
		\begin{proof}
			Fix $\sigma>0$ small and take $a_\sigma, q_\sigma$ as in Lemma \ref{lem:Hq}. By Lemma \ref{lem:I2},  $K(z)\geq2\sqrt z=2\sqrt{qa}$, and $K'(z)=(1-e^{-z})^{-1/2}=(1-e^{-qa})^{-1/2}$, we have
			\[
			\frac{\sqrt q|U_q(z)|}{K(z)}\leq \frac{C}{\sqrt a}\leq\frac{C}{\sqrt{a_\sigma}}\,,\qquad\frac{\sqrt q|U_q'(z)|}{K'(z)}\leq Ce^{-2a}\left(\frac{1-e^{-z}}{q}\right)^{3/2}\leq Ce^{-2a_\sigma}a_\sigma^{3/2}
			\]
			for some universal constant $C>0$ and for every $a\geq a_\sigma$ (for the second estimate we used the elementary inequality $1-e^{-qa}\leq qa$ and the monotonicity of $s\mapsto e^{-2s}s^{3/2}$ for large $s$). Observe that the right hand sides of both estimates can be made arbitrarily small by enlarging $a_\sigma$, so that by Lemma \ref{lem:Hq} we obtain
			\begin{equation}
			\label{eq:finalHU}
			\left|\frac{H_q(z)+\sqrt q U_q(z)}{K(z)}-1\right|\leq 2\sigma\,,\qquad\left|\frac{H_q'(z)+\sqrt q U_q'(z)}{K'(z)}-1\right|\leq 2\sigma
			\end{equation}
			for every $(q,a)\in(0,q_\sigma]\times [a_\sigma,+\infty)$ (recall that here $z=qa$ by \eqref{eq:az}, so that the condition $z\geq a_\sigma q$ of Lemma \ref{lem:Hq} is automatically satisfied). Since $H_q(z)+\sqrt q U_q(z)=q L_{2+q}(e^{-z/q})$ by \eqref{eq:splitL} and $z=qa=-q\log\alpha$, we also have $\displaystyle H_q'(z)+\sqrt q U_q'(z)=-\alpha L_{2+q}'(\alpha)$. Therefore, \eqref{eq:finalHU} gives
			\[
			\begin{split}
			\frac{L_{2+q}'(\alpha)}{L_{2+q}(\alpha)}&\,=-\frac q\alpha\frac{H_q'(z)+\sqrt q U_q'(z)}{H_q(z)+\sqrt q U_q(z)}=\frac{1}{\alpha\log\alpha}\frac{z\left(H_q'(z)+\sqrt q U_q'(z)\right)}{H_q(z)+\sqrt q U_q(z)}\\
			&\,\leq \frac1{\alpha\log\alpha}\frac{1-2\sigma}{1+2\sigma}\frac{zK'(z)}{K(z)}\leq\frac1{\alpha\log\alpha}\frac{1-2\sigma}{2(1+2\sigma)}\,.
			\end{split}
			\]
			The last bound follows by the concavity of the map $s\mapsto1-e^{-s}$, implying $\displaystyle \frac{\sqrt{z}}{\sqrt{s}\sqrt{1-e^{-z}}}\geq\frac{1}{\sqrt{1-e^{-s}}}$ for every $0<s\leq z$, which integrating over $[0,z]$ leads to $zK'(z)/K(z)\geq1/2$ for every $z>0$. 
			
			Since $\sigma$ is fixed so small that $1-2\sigma>0$ and the previous bound is uniform in $q\in(0,q_\sigma]$, it shows that $\displaystyle \frac{L_{2+q}'(\alpha)}{L_{2+q}(\alpha)}$ is bounded from above by a negative constant on every interval of the form $(0,\alpha_0]$, for any given small $\alpha_0\in(0,1)$. 
		\end{proof}
		We are now in position to prove Proposition \ref{prop:a=0}.
		
		\begin{proof}[Proof of Proposition \ref{prop:a=0}]
			In view of Corollary \ref{cor:L}, we are left to establish the second estimate in \eqref{eq:boundLN}. To this end, fix $\delta\in(0,1)$ and set
			\begin{equation}
				\label{eq:nq}
			n_q(\alpha):=\int_\alpha^{\beta_{2+q}(\alpha)}\frac{t^2dt}{\sqrt{2(\Phi_q(t)-\Phi_q(\alpha))}}=\sqrt q N_{2+q}(\alpha)
			\end{equation}
			and 
			\[
			n_q^-(\alpha):=\int_\alpha^{\delta}\frac{t^2dt}{\sqrt{2(\Phi_q(t)-\Phi_q(\alpha))}}\,,\qquad n_q^+(\alpha):=\int_\delta^{\beta_{2+q}(\alpha)}\frac{t^2dt}{\sqrt{2(\Phi_q(t)-\Phi_q(\alpha))}}\,.
			\]
			The proof of Corollary \ref{cor:L} ensures that we can take $\alpha_0<\delta/2$.  We will now show that $n_q$ is bounded away from zero and $n_q'$ is bounded uniformly in $(q,\alpha)\in(0,\overline q]\times(0,\alpha_0]$. Since by definition $\displaystyle n_q'(\alpha)/n_q(\alpha) = N_{2+q}'(\alpha)/N_{2+q}(\alpha)$, this will conclude the proof of Proposition \ref{prop:a=0}.
			
			We begin with $n_q^-$. Set
			\[
			t=\sqrt{\alpha^2+y^2}\,,
			\qquad
			Y(\alpha):=\sqrt{\delta^2-\alpha^2}\,,
			\]
			and 
			\begin{equation}\label{eq:Bcal-def}
				B_{q,\alpha}(y)
				:=\frac{2(\Phi_q(\sqrt{\alpha^2+y^2})-\Phi_q(\alpha))}{y^2}
				=\int_0^1
				\frac{1-(\alpha^2+\tau y^2)^{q/2}}q\,\dd\tau,
			\end{equation}
			where the first formula gives a rigorous definition for every $y>0$, but the second identity extends it by continuity also to $y=0$. 
			Since \(\sqrt{\alpha^2+y^2}\le\delta\) for $y\in(0,Y(\alpha)]$, 
			\begin{equation}\label{eq:Bcal-lower}
				B_{q,\alpha}(y)
				\ge\frac{1-\delta^q}{q}\ge c_\delta>0
			\end{equation}
			for all sufficiently small \(q>0\), so that changing the variable in the definition of $n_q^-$ yields
			\begin{equation}\label{eq:nminus-transformed}
				n_q^-(\alpha)
				=\int_0^{Y(\alpha)}
				\frac{\sqrt{\alpha^2+y^2}}{\sqrt{B_{q,\alpha}(y)}}\,\dd y
			\end{equation}
			and shows that \(n_q^-\) is bounded on $(0,\overline q]\times(0,\alpha_0]$.
			Moreover, differentiating \eqref{eq:Bcal-def} entails
			\begin{equation}\label{eq:Bcal-alpha}
				\partial_\alpha B_{q,\alpha}(y)
				=-\alpha\int_0^1(\alpha^2+\tau y^2)^{q/2-1}\,\dd\tau\,,
			\end{equation}
			so that differentiating \eqref{eq:nminus-transformed} and using \eqref{eq:Bcal-lower} gives
			\begin{equation}
					\label{eq:nminus-derivative-bound}
				|(n_q^-)'(\alpha)|\le C\left(\alpha
				+\int_0^{Y(\alpha)}\frac{\alpha}{\sqrt{\alpha^2+y^2}}\,\dd y
				+\alpha\int_0^{Y(\alpha)}\sqrt{\alpha^2+y^2}
				\int_0^1(\alpha^2+\tau y^2)^{q/2-1}\,\dd\tau \dd y\right),
			\end{equation}
			where the first term on the right hand side comes from $\displaystyle Y'(\alpha)\delta/\sqrt{B_{q,\alpha}(Y(\alpha))}$, which is bounded because
			$Y(\alpha)\ge\sqrt{3}\delta/2$ whenever $\alpha\leq\alpha_0<\delta/2$. Since 
			\[
			\int_0^{Y(\alpha)}\frac{\alpha}{\sqrt{\alpha^2+y^2}}\,\dd y=\alpha\operatorname{arcsinh}\frac{Y(\alpha)}\alpha
			\le C\alpha(1+|\log\alpha|),
			\]
			that is uniformly bounded, it remains to estimate the last integral in \eqref{eq:nminus-derivative-bound}.
			For \(0\le y\le\alpha\), since $q$ is small
			\[
			\int_0^1(\alpha^2+\tau y^2)^{q/2-1}\dd\tau
			\le\alpha^{q-2},
			\]
			and hence 
			\[
			\int_0^\alpha\sqrt{\alpha^2+y^2}
			\int_0^1(\alpha^2+\tau y^2)^{q/2-1}\dd\tau\dd y\leq C\alpha^{q}\leq C\,.
			\]
			For \(y\ge\alpha\), the identity
			\begin{equation}\label{eq:inner-integral-exact}
				\int_0^1(\alpha^2+\tau y^2)^{q/2-1}\,\dd\tau
				=\frac{2\bigl((\alpha^2+y^2)^{q/2}-\alpha^q\bigr)}{qy^2}
			\end{equation}
			combined with
			\[
			\frac{T^q-\alpha^q}{q}
			=\int_\alpha^T r^{q-1}\dd r
			\le\int_\alpha^T\frac{\dd r}{r}
			=\log\frac T\alpha,
			\qquad T\le\delta<1,
			\]
			gives
			\[
			\int_0^1(\alpha^2+\tau y^2)^{q/2-1}\dd\tau
			\le\frac{C}{y^2}\log\frac{Cy}{\alpha}\,,
			\]
			so that
			\[
			\alpha\int_\alpha^{Y(\alpha)}\sqrt{\alpha^2+y^2}
			\int_0^1(\alpha^2+\tau y^2)^{q/2-1}\dd\tau\dd y\leq C\alpha\int_\alpha^\delta\frac{1}{y}
			\log\frac{Cy}{\alpha}\dd y
			\le C\alpha\left(1+\log^2\frac C\alpha\right),
			\]
			which is uniformly bounded on \((0,\alpha_0]\).  This shows that
			\begin{equation}\label{eq:nminus-C1}
				\sup_{\substack{0<q<q_0\\0<\alpha\le\alpha_0}}
				\bigl(|n_q^-(\alpha)|+|(n_q^-)'(\alpha)|\bigr)<\infty,
			\end{equation}
			while that 
			\begin{equation}\label{eq:nplus-C1}
				\sup_{\substack{0<q<q_0\\0<\alpha\le\alpha_0}}
				\bigl(|n_q^+(\alpha)|+|(n_q^+)'(\alpha)|\bigr)<\infty.
			\end{equation}
			can be proved arguing exactly as in the proof of Lemma \ref{lem:I2}. This shows that both $n_q$ and $n_q'$ are bounded uniformly on $(0,\overline q]\times(0,\alpha_0]$. As for the boundedness away from zero of $n_q(\alpha)$ it is enough to observe that $\Phi_q$ is positive on $[\delta,1]$ for every $q\in(0,\overline q]$, while $\Phi_q(\alpha)$ is uniformly small on $(0,\overline q]\times(0,\alpha_0]$, provided $\alpha_0$ is small enough. Hence, 
			\[
			n_q(\alpha)\geq\int_\delta^1\frac{t^2dt}{\sqrt{2(\Phi_q(t)-\Phi_q(\alpha))}}\geq C>0
			\]
			for some universal constant $C$ and for every $(q,\alpha)\in(0,\overline q]\times(0,\alpha_0]$. 
			
			Since both estimates in \eqref{eq:boundLN} have now been established, by \eqref{eq:derlog} we have
			\[
			(\log M_{2+q}(\alpha))'\leq -C\left(\frac{4-q}{q}-1\right)\leq -\frac C{\overline q}<0\,,
			\]
		showing that $M_{2+q}'(\alpha)<0$ for every $(q,\alpha)\in(0,\overline q]\times(0,\alpha_0]$ and concluding the proof of Proposition \ref{prop:a=0}.
		\end{proof}			
			\subsection{Proof of Proposition \ref{prop:mass}}
				Let $\varepsilon=\min\left\{\varepsilon_1,\overline q\right\}>0$, where $\varepsilon_1,\overline q$ are the numbers identified in Propositions \ref{prop:monl*}--\ref{prop:a=0} respectively. For every $p\in(2,2+\varepsilon)$, we then have that $M_p'(\alpha)<0$ on $(0,\alpha_0]\cup[\underline\alpha,1)$. Moreover, by \eqref{eq:L-bif-limit}, \eqref{eq:M-param} and \eqref{eq:N-bif-limit}, it follows $\displaystyle M_p(\alpha)\to\mu_*(p)$ as $\alpha\to1$, with $\mu_*(p)$ as in \eqref{eq:mu2}. Therefore, to conclude we are left to show that $M_p(\alpha)$ is strictly decreasing on $[\alpha_0,\underline\alpha]$. 
				
				To do this, fix $\overline\alpha\in(\underline\alpha,1)$, set again $q=p-2$, and take $\Phi_q, \Phi$ as in the proof of Lemma \ref{lem:I2}. Recall that we can rewrite
				\[
				(\log M_2+q(\alpha))'=\left(\frac4{q}-1\right)\frac{\ell_q'(\alpha)}{\ell_q(\alpha)}+\frac{n_q'(\alpha)}{n_q(\alpha)}
				\]
				where $n_q$ is the function defined in \eqref{eq:nq}, and
				\[
				\ell_q(\alpha):=\int_\alpha^{\beta_{2+q}(\alpha)}\frac{dt}{\sqrt{2\left(\Phi_q(t)-\Phi_q(\alpha)\right)}}=\sqrt q L_{2+q}(\alpha)\,.
				\]
				We now prove that, as $q\to0^+$,
				\begin{equation}
					\label{eq:convbeta}
					\beta_{2+q}\to \beta\qquad\text{in }C^1(\alpha_0,\overline\alpha)\,,
				\end{equation}
				where $\beta(\alpha)\in (1,\sqrt{e})$ is the unique solution of $\Phi(\beta(\alpha))=\Phi(\alpha)$ for every $\alpha\in [\alpha_0,\overline\alpha]$, and
				\begin{equation}
					\label{eq:convln}
					\ell_q\to\ell\,,\qquad \ell_q'\to\ell'\,,\qquad n_q\to n\,,\qquad n_q'\to n'
				\end{equation}
				uniformly on $[\alpha_0,\overline\alpha]$, where
				\begin{equation}
				\label{eq:ln}
				\ell(\alpha):=\int_\alpha^{\beta(\alpha)}\frac{dt}{\sqrt{2\left(\Phi(t)-\Phi(\alpha)\right)}}\,,\quad n(\alpha):=\int_\alpha^{\beta(\alpha)}\frac{t^2dt}{\sqrt{2\left(\Phi(t)-\Phi(\alpha)\right)}}\,.
				\end{equation}
				As for \eqref{eq:convbeta}, it is enough to note that $\beta_{2+q}(\alpha)$ is defined by the identity $\Phi_q(\beta_{2+q}(\alpha))-\Phi_q(\alpha)=0$, and the same for $\beta(\alpha)$ with $\Phi(\beta(\alpha))-\Phi(\alpha)=0$. Since $\beta_{2+q},\beta$ are bounded away from $1$ uniformly in $\alpha\in[\alpha_0,\overline\alpha]$ and $\beta_{2+q}'(\alpha)=\Phi_q'(\alpha)/\Phi_q'(\beta_{2+q}(\alpha))$, $\beta'(\alpha)=\Phi'(\alpha)/\Phi'(\beta(\alpha))$, the uniform $C^2$ convergence of $\Phi_q$ to $\Phi$ on compact subsets of  $(0,+\infty)$ and the fact that $\Phi_q'(s),\Phi'(s)$ are bounded away from 0 as $s$ is bounded away from 1 directly imply the uniform $C^1$ convergence of $\beta_{2+q}$ to $\beta$ on $[\alpha_0,\overline\alpha]$ as $q\to0^+$.
				
				We then turn to \eqref{eq:convln}. Note first that, changing the variable $\displaystyle t=T_q(\alpha,\theta):=\alpha+(\beta_{2+q}(\alpha)-\alpha)\sin^2\theta$,
				\begin{equation}
					\label{eq:newlnp}
					\ell_q(\alpha)=\sqrt2 \int_0^{\frac\pi2}A_q(\alpha,\theta)^{-1/2}\,d\theta\,,\qquad n_q(\alpha)=\sqrt2\int_0^{\frac\pi2} T_q(\alpha,\theta)^2 A_q(\alpha,\theta)^{-1/2}\,d\theta
				\end{equation}
				with
				\[
				A_q(\alpha,\theta):=\frac{\Phi_q(T_q(\alpha,\theta))-\Phi_q(\alpha)}{(\beta_{2+q}(\alpha)-\alpha)^2\sin^2\theta\cos^2\theta}
				\]
				for every $\theta\in(0,\pi/2)$. Analogously, setting $T(\alpha,\theta):=\alpha+(\beta(\alpha)-\alpha)\sin^2\theta$ and changing the variable in the definition of $\ell$ and $n$ we obtain
				\begin{equation}
					\label{eq:newlp}
					\ell(\alpha)=\sqrt2 \int_0^{\frac\pi2}A(\alpha,\theta)^{-1/2}\,d\theta\,,\qquad n(\alpha)=\sqrt2\int_0^{\frac\pi2} T(\alpha,\theta)^2 A(\alpha,\theta)^{-1/2}\,d\theta
				\end{equation}
				with
				\[
				A(\alpha,\theta):=\frac{\Phi(T(\alpha,\theta))-\Phi(\alpha)}{(\beta(\alpha)-\alpha)^2\sin^2\theta\cos^2\theta}\,,\qquad\forall \theta\in(0,\pi/2)\,.
				\]
				Note that $A_q(\alpha,\theta), A(\alpha,\theta)$ extend to $C^1$ functions on $\displaystyle [\alpha_0,\overline\alpha]\times[0,\pi/2]$ (it is enough to write $s=\sin^2\theta$ and recall that $\Phi_q,\Phi$ are $C^2$ on $[\alpha_0,\overline\alpha]$). Moreover, the uniform $C^2$ convergence of $\Phi_q$ to $\Phi$ on $[\alpha_0,\overline\alpha]$ as $q\to0^+$ directly yields the uniform $C^1$ convergence of $A_q$ to $A$ on $[\alpha_0,\overline\alpha]\times[0,\pi/2]$ as $q\to 0^+ $.

Furthermore, since the continuous extensions at the endpoints $\theta=0$ and $\theta=\pi/2$ give
				\[
				A(\alpha,0)=\frac{\Phi'(\alpha)}{\beta(\alpha)-\alpha}=\frac{-\alpha\log\alpha}{\beta(\alpha)-\alpha}>0\,,\qquad A(\alpha,\pi/2)=-\frac{\Phi'(\beta(\alpha))}{\beta(\alpha)-\alpha}=\frac{\beta(\alpha)\log\beta(\alpha)}{\beta(\alpha)-\alpha}>0
				\]
				uniformly on $[\alpha_0,\overline\alpha]$, it is readily seen that $A$ is strictly positive on $[\alpha_0,\overline\alpha]\times[0,\pi/2]$, and by continuity the same is true for $A_q$ whenever $q$ is sufficiently close to $0$. Therefore, combining with \eqref{eq:newlnp} and \eqref{eq:newlp} proves \eqref{eq:convln}, in turn yielding (by the strict positivity of $\ell, n$ on $[\alpha_0,\overline\alpha]$)
				\begin{equation}
					\label{eq:convfrac}
					\frac{\ell_q'}{\ell_q}\to\frac{\ell'}{\ell}\,,\qquad\frac{n_q'}{n_q}\to\frac{n'}n\qquad\text{as }q\to0^+
				\end{equation}
				uniformly on $[\alpha_0,\overline\alpha]$. Now, the Appendix below proves that $\ell'(\alpha)<0$ on $[\alpha_0,\overline\alpha]$, that together with the strict positivity of $\ell$ gives 
				\[
				\max_{\alpha\in[\alpha_0,\overline\alpha]}\frac{\ell'(\alpha)}{\ell(\alpha)}<0\,.
				\]
				By \eqref{eq:convfrac} we then have
				\[
				\max_{\alpha\in[\alpha_0,\overline\alpha]}\frac{\ell_q'(\alpha)}{\ell_q(\alpha)}<0\,,\qquad\max_{\alpha\in[\alpha_0,\overline\alpha]}\left|\frac{n_q'}{n_q}\right|\leq C\,,
				\]
				for a suitable constant $C>0$ and for every $q$ close enough to $0$. Combining with \eqref{eq:derlog} shows that, whenever $q$ is sufficiently close to $0$, 
				\[
				(\log M_{2+q}(\alpha))'<0\qquad\forall \alpha\in[\alpha_0,\overline\alpha]\,,
				\]
				completing the proof of Proposition \ref{prop:mass}.
		
			\section*{Appendix}
			For the sake of completeness, here is a self-contained proof of the fact that $\ell'(\alpha)<0$ for every $\alpha\in(0,1)$, where $\ell(\alpha)$ is the function defined in \eqref{eq:ln}.  
			
			Set $\displaystyle V(s):=\Phi(1)-\Phi(s)=\frac14-\frac{s^2}{2}+\frac{s^2\log s}{2}$, so that $V(1)=V'(1)=0$ and $V$ has a global minimum point at $s=1$.  To any number $h\in(0,1/4)$ it corresponds a periodic orbit of the ODE $x''(t)+V'(x(t))=0$ with semiperiod given by
			\[
			T(h)=\int_{x_-(h)}^{x_+(h)}\frac{dy}{\sqrt{2(h-V(y))}}\,,
			\]
			where $0<x_-(h)<1<x_+(h)$ are such that $x'(0)=x'(T(h))=0$, $x(0)=x_-(h)$, $x(T(h))=x_+(h)$. 
			
			Set then $X(y):=\text{\normalfont sgn}(y-1)\sqrt{2V(y)}$. Since $\displaystyle V(s)=\frac{(s-1)^2}{2}+O((s-1)^3)$ locally around $s=1$, $X$ is smooth on $(0,+\infty)$. Moreover, $X'(y)=\text{\normalfont sgn}(y-1)V'(y)/\sqrt{2V(y)}>0$ for every $y>0$, $y\neq1$, and $X'(1)=1$ by continuity. Hence, $X(y)$ is a smooth, strictly increasing function on $(0,+\infty)$, and it thus admits a smooth, strictly increasing inverse function $y(X)$ with $y(0)=1$. By definition, 
			\begin{equation}
			\label{eq:VX}
			V(y(X))=\frac{X^2}{2}\,,
			\end{equation}
			so that, setting $R:=\sqrt{2h}$, we have $\displaystyle X(x_-(h))=-R$, $\displaystyle X(x_+(h))=R$. Therefore, changing the variable in the formula for $T(h)$ above, we obtain
			\[
			T(h)=\int_{-R}^{R}\frac{y'(X)dX}{\sqrt{R^2-X^2}}=\int_{-\frac\pi2}^{\frac\pi2}y'(R\sin\theta)\,d\theta=:f(R)\,.
			\]
			By definition, $\ell(\alpha)$ as in \eqref{eq:ln} satisfies $\ell(\alpha)=f(R(\alpha))$, where $R(\alpha)=\sqrt{2h(\alpha)}=\sqrt{2V(\alpha)}>0$, for every $\alpha\in(0,1)$. As a consequence, $\ell'(\alpha)=f'(R(\alpha))R'(\alpha)$, and since $\displaystyle R'(\alpha)=\frac{V'(\alpha)}{\sqrt{2V(\alpha)}}=\frac{V'(\alpha)}{R(\alpha)}<0$ as $\alpha\in(0,1)$, to show that $\ell'(\alpha)<0$ it is enough to prove that $f'(R(\alpha))>0$ for every $\alpha\in(0,1)$. Observe that, differentiating inside the integral and then integrating by parts, 
			\begin{equation}
			\label{eq:f'}
			f'(R)=\int_{-\frac\pi2}^{\frac\pi2}y''(R\sin\theta)\sin\theta\,d\theta=R\int_{-\frac\pi2}^{\frac\pi2}y'''(R\sin\theta)\cos^2\theta\,d\theta\,.		
			\end{equation}
			Moreover, differentiating \eqref{eq:VX} yields $\displaystyle y'(X)=\frac{X}{V'(y(X))}$, so that setting $\displaystyle C(s):=\frac{V(s)}{V'(s)^2}$, we have 
			\[
			\frac{y'(X)^2}{2}=C(y(X))\,.
			\]
			Note that the previous expansion of $V(s)$ at $s=1$ ensures that $C\in C^2(0,+\infty)$ (the values at $s=1$ being defined by continuity).  Differentiating twice the previous identity leads to
			\begin{equation}
			\label{eq:y'''}
			y'''(X)=C''(y(X))y'(X).
			\end{equation}
			Since $y(X)$ is strictly increasing, $y'(X)>0$. Moreover, $C''(s)>0$ for every $s>0$. Indeed, differentiating twice the definition of $C(s)$ gives
			\[
			C''(s)=\frac{H(\log s)}{2(s\log s)^4}\,,\qquad H(s):=e^{2s}\left(s^2+s-3\right)+3s^2+5s+3\,.
			\]
			Since
			\[
			H'(s)=e^{2s}(2s^2+4s-5)+6s+5,\quad H''(s)=e^{2s}(4s^2+12s-6)+6,\quad H'''(s)=e^{2s}(8s^2+32s),
			\]
			it is readily seen that $s=0$ is the unique global minimum point of $H(s)$, that together with $H(0)=0$ entails $H(s)>0$ for every $s\neq0$. Combining with \eqref{eq:f'} and \eqref{eq:y'''} and recalling that $R(\alpha)>0$ for every $\alpha\in(0,1)$, we obtain $f'(R(\alpha))>0$ whenever $\alpha\in(0,1)$, and we conclude.
					
	\section*{Statements and Declarations}
	
  \noindent The authors declare that they have  no conflict of interest and that data sharing is not applicable to this article as it has no associated data. 
  
  \section*{Acknowledgements}	
  \noindent Lun Guo is supported by the Program of China Scholarship Council (Grant number: 202508420089).

	\end{document}